\documentclass[12pt]{amsart}

\usepackage{calc}  
\usepackage{amsmath}  
\usepackage{amsthm}
\usepackage{amsfonts}
\usepackage{amssymb}
\usepackage{amstext}
\usepackage{amscd}
\usepackage{enumerate}

\newtheorem{theorem}{Theorem}[section]
\newtheorem{proposition}[theorem]{Proposition}
\newtheorem{corollary}[theorem]{Corollary}
\newtheorem{lemma}[theorem]{Lemma}

\theoremstyle{definition}

\newtheorem{claim}[theorem]{Claim}

\newtheorem{definition}[theorem]{Definition}
\newtheorem{remark}[theorem]{Remark}

\newtheorem{hypothesis}[theorem]{Hypothesis}
\newtheorem{conjecture}[theorem]{Conjecture}
\newtheorem{acknow}{Acknowledgments}

\def\Z{\mathbb{Z}}
\def\Q{\mathbb{Q}}
\def\R{\mathbb{R}}
\def\C{\mathbb{C}}

\def\bE{\mathbb{E}}

\def\CP{\mathbb{C} \mathbb{P}}
\def\barCP2{\overline{\mathbb{CP}}\ \! \!^2}

\def\<{\left\langle}
\def\>{\right\rangle}
\def\({\left(}
\def\){\right)}

\def\cA{\mathcal{A}}
\def\cB{\mathcal{B}}

\def\cD{\mathcal{D}}

\def\cG{\mathcal{G}}

\def\cL{\mathcal{L}}

\def\fg{\mathfrak{g}}

\def\su{\mathfrak{su}}

\def\Aut{\operatorname{Aut}}

\def\const{\operatorname{const}}

\def\Im{\operatorname{Im}}

\def\Ind{\operatorname{Ind}}
\def\ind{\operatorname{ind}}

\def\supp{\operatorname{supp}}
\def\Sym{\operatorname{Sym}}

\def\Dir{\not \! \! \mathfrak{D}}

\def\Lam{\Lambda(\rho, \sigma)}
\def\lam{\lambda}

\def\unA{[\underline{A}]}

\def\B{\cB_{Y \times \R}^*(\rho, \sigma)}
\def\tB{\tilde{\cB}^*_{Y \times \R}(\rho, \sigma)}
\def\X0B{\cB^*_{\hat{X}_0}(\rho)}
\def\tX0B{\tilde{\cB}^*_{ \hat{X}_0}(\rho)}

\def\aR{\underline{\R}(\theta_a, \rho)}
\def\bR{\underline{\R}(\sigma, \theta_b)}
\def\aC{\underline{\C}(\theta_{\gamma,a}, \rho)}
\def\bC{\underline{\C}(\sigma, \theta_{\gamma, b})}

\def\aLam{\Lambda^{(a)}(\rho, \sigma)}
\def\1L{\cL_{\Gamma}^{(1)}(\rho,\sigma)}
\def\a0Lam{\Lambda^{(a)}(\rho)}
\def\S0L{\cL_{ \hat{\Sigma}_0 }^{(1)}(\rho)}

\def\aInd{\Ind^{-+}\Dir_{A(\theta_a, \rho)}}
\def\bInd{\Ind^{-+} \Dir_{A(\sigma, \theta_b)}}
\def\abInd{\Ind^{-+} \Dir_{A(\theta_a, \theta_b)}}

\def\gamaInd{\Ind^{-+} \bar{\partial}_{A(\theta_{\gamma, a},\rho)} }
\def\gambInd{\Ind^{-+} \bar{\partial}_{A(\sigma,\theta_{\gamma, b})}}
\def\gamabInd{\Ind^{-+} \bar{\partial}_{A(\theta_{\gamma, a}, \theta_{\gamma, b})}}

\def\M{M_{Y \times \R}(\rho, \sigma)}
\def\0M{M_{Y \times \R}^0(\rho, \sigma)}
\def\ML{M_{Y \times \R}^0(\rho, \sigma;L)}
\def\iM{M_{Y \times \R}(\rho(i-1), \rho(i))}
\def\i0M{M^0_{Y \times \R}(\rho(i-1), \rho(i))}
\def\bM{M_{Y \times \R}^0(\rho, \sigma;L_{\beta})}
\def\3M{M_{Y \times \R}^0(\rho, \sigma;L_{3})}
\def\b12M{M_{Y \times \R}^0(\rho, \sigma;L_{\beta_2} \backslash L_{\beta_1})}

\def\MX0{ M_{ \hat{X}_0 }(\rho) }

\def\ungam{\underline{\gamma}}

\def\unlam{\underline{\lambda}}
\def\unn{\underline{n}}

\def\sl{s_l(\rho, \sigma)}
\def\isl_{s_l(\rho(i-1), \rho(i)}

\begin{document}


\title[Floer homology for  2-torsion instanton invariants]
{Floer homology for  \\  2-torsion instanton invariants}
\author{Hirofumi Sasahira}
\date{}

\renewcommand{\thefootnote}{\fnsymbol{footnote}}
\footnote[0]{2000\textit{ Mathematics Subject Classification}.
Primary 57R57, 57R58.}

\address{
Graduate school of Mathematics, Nagoya University,\endgraf 
Furocho, Chikusaku, Nagoya, Japan.
}
\email{hsasahira@math.nagoya-u.ac.jp}

\maketitle



\begin{abstract}
We construct a variant of Floer homology groups and prove a gluing formula for a variant of Donaldson invariants.
As a corollary, the variant of Donaldson invariants is non-trivial for connected sums of 4-manifolds which satisfy a condition for Donaldson invariants.
We also show a non-existence result of compact, spin 4-manifolds with boundary some homology 3-spheres and with certain intersection forms.

\end{abstract}



\section{Introduction}

In this paper, we construct a variant of instanton Floer homology groups and prove a gluing formula for a variant of Donaldson invariants introduced in \cite{S}.

As S. K. Donaldson showed in \cite{appli, conn, h cobordism, poly}, the moduli spaces of instantons have crucial information about the topology of closed, oriented, smooth 4-manifolds $X$.
Mainly there are two methods to draw out the information from the moduli spaces. First one is to describe the structures of the singular points and the ends of the moduli spaces \cite{appli, conn}. The description gives strong restrictions to the realization of unimodular quadratic forms as the intersection form of a smooth 4-manifold.
Second one is to integrate cohomology classes over moduli spaces \cite{h cobordism, poly}. This gives differential-topological invariants of 4-manifolds which distinguish different smooth structures on the same topological 4-manifold.

It used to be hard to compute Donaldson invariants in general. However instanton Floer theory gives us a way to compute Donaldson invariants $\Psi_X$ when $X$ has a decomposition $X=X_0 \cup_{Y} X_1$ for some compact 4-manifolds $X_0$ and $X_1$ with boundary $Y$ and $\bar{Y}$. Here $Y$ is a closed, oriented 3-manifold and $\bar{Y}$ is $Y$ with the opposite orientation.
A. Floer introduced instanton homology groups $HF_*(Y)$ for homology 3-spheres $Y$ in \cite{Fl}. Floer's groups allow us to generalize Donaldson invariants for compact, oriented 4-manifolds $X_0$ whose boundaries are homology 3-spheres $Y$. The relative invariant $\Psi_{X_0}$ is an element of $HF_*(Y)$, which is defined by integrating cohomology classes of the moduli spaces of instantons over $\hat{X}_0 = X_0 \cup (Y \times  \R_{ \geq 0 })$. Assume that a closed 4-manifold $X$ has a decomposition $X = X_0 \cup_{Y} X_1$ for some homology 3-sphere  $Y$ and two compact 4-manifolds $X_0, X_1$ with boundary $Y$,  $\bar{Y}$ and with $b^+(X_0), b^+(X_1)>1$. Here $b^+$ stands for the dimension of a maximal positive subspace of the intersection form on $H_2(X)$. Then we have the relative invariants $\Psi_{X_0} \in HF_*(Y)$, $\Psi_{X_1} \in HF_*( \bar{Y} )$. There is a natural pairing 
\[
< \ , \ > : HF_*(Y) \otimes HF_*( \bar{Y} ) \longrightarrow \Q,
\]
and we have a relation
\begin{equation} \label{eq intro 1}
\Psi_{X} = < \Psi_{X_0}, \Psi_{X_1} >.
\end{equation}

Instanton Floer homology groups also enable us to generalize the results on the non-existence of closed 4-manifolds with $b^+ = 0$ and with non-standard intersection forms. Donaldson \cite{D Floer} showed that if $HF_1(Y)=0$ then there is a restriction to the realization of unimodular quadratic forms as the intersection form of a 4-manifold $X_0$ with boundary $Y$ and with $b^+ =0$.
The proof involves the description of the singular points and the ends of the moduli spaces of instantons over $\hat{X}_0$ as in the case of closed 4-manifolds. We also refer to K. A. Froyshov's papers \cite{Fr SW, Fr Equiv, Fr h inv}. He gave obstructions for (rational) homology 3-spheres to bound a 4-manifold, using Seiberg-Witten theory and ``h invariant'' defined by instanton Floer theory.

\vspace{2mm}

On the other hand, there are some variants of Donaldson invariants. R. Fintushel and R. Stern defined a variant of Donaldson invariants for closed, simply connected, spin 4-manifolds in \cite{FS}. They used a 2-torsion cohomology class $u_1 \in H^1(M_P;\Z_2)$ of moduli spaces $M_P$ which was originally defined by Donaldson in \cite{conn}.
In \cite{AMR}, S. Akbulut, T. Mrowka and Y. Ruan extended the construction of $u_1$ to non-spin 4-manifolds and they showed a universal constraint of Donaldson invariants for non-spin 4-manifolds.
Using $u_1$ for non-spin 4-manifolds, the author \cite{S} defined a variant $\Psi_{X}^{u_1}$ of Donaldson invariants for non-spin 4-manifolds, which can be regarded as an extension of Fintushel-Stern's invariant to non-spin 4-manifolds.
A remarkable feature of these variants is that they are non-trivial for connected sums of the form $X = Y \# S^2 \times S^2$ in general.

In \cite{Fur}, it was announced (without proof) that K. Fukaya, M. Furuta and H. Ohta showed a non-existence result of compact, spin 4-manifolds with $b^+ = 1$ and with some intersection form when their boundaries are diffeomorphic to some homology 3-spheres.  
In \cite{Fur}, Furuta explained that a variant of instanton Floer homology groups and an extension of $\Psi_X^{u_1}$ to compact, spin $4$-manifolds with boundary a homology 3-sphere are used in the proof. This is a nice generalization of a result in \cite{conn}, however, explicit examples were not given.

\vspace{2mm}

The main purpose of this paper is to construct a variant of instanton Floer homology groups and to extend the relation (\ref{eq intro 1}) to $\Psi_X^{u_1}$. As a corollary, we obtain a non-vanishing result for connected sums of some 4-manifolds. Moreover we will write down a proof of Fukaya-Furuta-Ohta's non-existence result  using the variant of instanton Floer homology groups.
In the proof, we make use of the structure of the ends of some moduli spaces as in \cite{conn}. There is a difference between our variant of instanton Floer homology and Fukaya-Furuta-Ohta's one. In our construction we use the determinant line bundles of families of $\bar{\partial}$-operators over $\gamma \times \R$. Here $\gamma$ is a loop in $Y$. This enables our gluing formula to be applied to not only the case when 4-manifolds are spin but also the case when 4-manifold are non-spin.

\vspace{2mm}

In \cite{S}, the author showed that the variant of Donaldson invariants is non-trivial for $2\CP^2 \# \barCP2$. As we will explain in Section \ref{s torsion}, this calculation is suggesting that we need to take into account the following aspects in the construction of our variant of instanton Floer homology groups:

\vspace{2mm}

\begin{itemize}

\item
The trivial flat connection on $Y$ should play an important role in the gluing formula for $\Psi_{X}^{u_1}$.

\vspace{2mm}

\item
We need a similar construction to Fukaya's extension  \cite{Fu} of instanton Floer homology groups.

\end{itemize}

\vspace{2mm}

In the gluing formula for the usual Donaldson invariants, the trivial flat connection has no contribution. Hence the contribution of the trivial flat connection to the gluing formula for $\Psi^{u_1}_X$ is a new phenomenon. See Donaldson's book \cite{D Floer} and Froyshov's papers \cite{Fr Equiv, Fr h inv} for other treatments of the trivial flat connection.
In Fukaya's construction, we use the determinant line bundles of $\bar{\partial}$-operators over $\gamma \times \R$, and we need sections of the determinant line bundles with certain properties. The properties have to do with non-compactness of the moduli spaces over $Y \times \R$ and transversality. (See Proposition \ref{prop section}.) We call sections having the properties admissible sections. One of the main purposes of this paper is to give a construction of admissible sections of determinant line bundles and it will be done in Section \ref{section Fukaya-Floer}.

\vspace{2mm}

Instanton Floer homology groups are defined for any homology 3-spheres. For other 3-manifolds $Y$, there are some problems coming from reducible flat connections over 3-manifolds and the decomposition of homology classes of $X$ along $Y$. In order to overcome the first difficulty, D. M. Austin and P. J. Braam \cite{AB} introduced an equivariant version of instanton Floer homology for 3-manifolds $Y$ with $b_1(Y) = 0$ or $H_1(Y;\Z)$ torsion free. Kronheimer and Mrowka \cite{KM} gave another method to treat reducible solutions in the context of Seiberg-Witten theory. On the other hand, Fukaya's construction \cite{Fu} gives an effective way to overcome the second difficulty when there are no reducible flat connections.  However it does not seem that we have a complete method to construct suitable instanton Floer homology groups fitting to the situation where we treat both the problems at the same time. To construct a gluing formula for our invariant $\Psi_X^{u_1}$, we must treat both the problems at the same time, even when the 3-manifold $Y$ is a homology 3-sphere. Although we treat only homology 3-spheres in this paper, we can regard our construction as a first step to give a method simultaneously applicable  to both the problems.

\vspace{2mm}

The organization of this paper is as follows.
In Section \ref{section Fukaya-Floer}, we recall Fukaya's construction.
We will basically follow \cite{Fu} and \cite{BD}, however, with some modification.
In Section \ref{s torsion}, making use of techniques developed in Section \ref{section Fukaya-Floer}, we will introduce a variant $I^{(a)}_*(Y; \gamma)$ of instanton Floer homology groups for homology 3-spheres $Y$, loops $\gamma$ in $Y$ and $a \in \{ 0, 1 \}$. We define relative invariants $\Psi_{X_0}^{u_1} \in I^{(a)}_*(Y; \gamma)$ for compact 4-manifolds $X_0$ with boundary $Y$. We also prove a formula  for $\Psi_{X}^{u_1}$ similar to (\ref{eq intro 1}) (Theorem \ref{thm main}).
In particular, we deduce a non-vanishing result for connected sums of 4-manifolds (Corollary \ref{coro connected sum}).
In Section \ref{s Thm B}, we will prove the result on the non-existence of spin 4-manifolds with boundaries some homology 3-spheres and with some intersection forms (Theorem \ref{thm B}).

\begin{acknow}
The author is grateful to Mikio Furuta for his suggestions.
He would also like to thank Yukio Kametani and Nobuhiro Nakamura for useful conversations.
\end{acknow}

\section{Fukaya-Floer homology} \label{section Fukaya-Floer}

Let $X$ be a closed, oriented 4-manifold with $b^+ > 1$ and take a $U(2)$-bundle $P$ over $X$ and assume that the dimension of the moduli space of instantons on $P$ is $2d$ for some integer $d \geq 0$.
We can associate cohomology classes $\mu([\Sigma])$ of degree 2 on the moduli space to homology classes $[\Sigma] \in H_2(X;\Z)$. Here $\Sigma$ is a closed, oriented surface embedded in $X$ which represent the homology class.
The cohomology classes $\mu([\Sigma])$ are the first Chern classes of the determinant line bundles of a family of twisted $\bar{\partial}$ operators on $\Sigma$. 
Roughly speaking Donaldson invariants are the evaluations of cup products $\mu([\Sigma_1]) \cup \cdots \cup \mu([\Sigma_d])$ on the moduli spaces.

Suppose that we  have a decomposition $X = X_0 \cup_{Y} X_1$ and that the surfaces $\Sigma_l$ are split into two surfaces with boundary $\gamma_{l} \cong S^1$ along $Y$.
We consider how to recover the Donaldson invariants of $X$ from relative invariants of $X_0$ and $X_1$.
Fukaya's construction \cite{Fu} provide us suitable homology groups to define the relative invariants.
In the construction, we need sections of determinant line bundles of families of twisted $\bar{\partial}$ operators on $\gamma_l \times \R$ which behave nicely on the ends of moduli spaces of instantons over $Y \times \R$ and satisfy certain transversality conditions.
We call such sections admissible.
The main point of this section is the construction of admissible sections of the determinant line bundles (Proposition \ref{prop section}).

\subsection{Outline of proof of gluing formula}
\label{ss outline}

In this subsection, we give a sketch of the construction of a gluing formula for Donaldson invariants.

Let $X$ be a closed, oriented, simply connected smooth 4-manifolds with $b^+ > 1$ and odd, and $P$ be a $U(2)$-bundle over $X$ with $w_2(P)$ non-trivial.
We write $P_{\det}$ for the $U(1)$-bundle over $X$ induced by $P$, and fix a connection $a_{\det}$ on $P_{\det}$.
We take a Riemannian metric on $g$ on $X$ and suppose that the virtual dimension of $M_P$ is $2d$ with $d \geq 0$.
Here $M_P = M_P(g)$ is the moduli space of instantons on $P$ which induce the connection $a_{\det}$ on $P_{\det}$.
The moduli space $M_P$ is a smooth manifold for generic metrics $g$, and a choice of an orientation of the space $H_g^+(X)$ of self-dual harmonic 2-forms on $X$ orients $M_P$.
The moduli space $M_P$ gives us the Donaldson invariant
\[
\Psi_{X,P}:H_2(X;\Z)^{\otimes d} \longrightarrow \Q.
\]
This is defined as follows.

Let $[\Sigma]$ be a class in $H_2(X;\Z)$.
We denote by $\tilde{\cB}_{\Sigma}^*$ the space of gauge equivalence classes of framed, irreducible connections on the restriction $P|_{\Sigma}$ of $P$ to $\Sigma$ which are compatible with $a_{\det}|_{\Sigma}$.
We have the determinant line bundle
\[
\tilde{\cL}_{\Sigma} = \det \Ind \{ \bar{\partial}^*_{A} \}_{ [A] \in \tilde{\cB}_{\Sigma} } \stackrel{\C}{\longrightarrow} \tilde{\cB}_{\Sigma}
\]
of the family of operators $\{ \bar{ \partial }_{A}^* \}_{ [A] \in \tilde{\cB}_{\Sigma} }$. 
Here  
\[
\bar{\partial}_A:
\Omega_{\Sigma}^0 (E|_{\Sigma} \otimes K_{\Sigma}^{\frac{1}{2}})
\longrightarrow
\Omega_{\Sigma}^{0,1}(E|_{\Sigma} \otimes K_{\Sigma}^{\frac{1}{2}})
\]
is the $\bar{\partial}$ operator twisted by the rank-two complex vector bundle $E$ associated to $P$ and a square root $K^{\frac{1}{2}}_{\Sigma}$ of the canonical line bundle of $\Sigma$, and $\bar{\partial}_{A}^*$ is the adjoint.

There is a natural action of $SU(2)$ on $\tilde{\cB}_{\Sigma}^*$, and the action of  $\{ \pm 1 \} \subset SU(2)$ is trivial. Hence we have the $SO(3)=SU(2)/\pm 1$ action on $\tilde{\cB}_{\Sigma}^*$. The quotient space $\tilde{\cB}^*_{\Sigma}/SO(3)$ is the space $\cB_{\Sigma}^*$ of irreducible connections on $P|_{\Sigma}$ which are compatible with $a_{\det}|_{\Sigma}$.
Since the action of $\{ \pm 1 \} \subset SU(2)$ on $\tilde{\cL}_{\Sigma}^{\otimes 2}$ is trivial, we have the line bundle
\[
\cL_{\Sigma}^{ \otimes 2 } := 
\tilde{\cL}_{\Sigma}^{ \otimes 2}/SO(3)
\stackrel{\C}{\longrightarrow}
\cB_{\Sigma}^*.
\]
Note that $\cL_{\Sigma}^{\otimes 2}$ may not be a square of a genuine line bundle.
Take a section $s_{\Sigma}$ of $\cL_{\Sigma}^{ \otimes 2 }$ and denote the zero locus by $V_{\Sigma}$. We define $M_P \cap V_{\Sigma}$ by
\[
M_P \cap V_{\Sigma} = \{ \ [A] \in M_P \ | \ [ A|_{\Sigma} ] \in V_{\Sigma} \ \}.
\]
Since the restriction of elements of $M_P$ to $\Sigma$ may not be irreducible, this is not well-defined. But we can avoid this problem by replacing $\Sigma$ with a small neighborhood $\nu(\Sigma)$ of $\Sigma$. (See \cite{DK}.)

Take $d$ homology classes $[\Sigma_1], \dots, [\Sigma_d] \in H_2(X;\Z)$. We can show that the intersection
\[
M_P \cap V_{\Sigma_1} \cap \cdots \cap V_{\Sigma_d}
\]
is transverse and  finite for generic surfaces $\Sigma_1, \dots, \Sigma_d$ and sections $s_{\Sigma_1}, \dots, s_{\Sigma_d}$. Since $M_P$ and $V_{\Sigma_i}$ are oriented, we can associate each point of the intersection with a sign. We define $\Psi_{X, P}([\Sigma_1], \dots, [\Sigma_d])$ to be $1/2^d$ times the number of points of the intersection counted with sign, i.e.
\[
\Psi_{X,P}([\Sigma_1], \dots, [\Sigma_d]) =
\frac{1}{2^d} \# ( M_{P} \cap V_{\Sigma_1} \cap \cdots \cap V_{\Sigma_d}).
\]
We can see that this is independent of the choices of metric and sections.

\vspace{2mm}

Suppose that $X$ has a decomposition $X = X_0 \cup_Y X_1$.
Choose a Riemannian metric $g_Y$ on $Y$ and let $g_0, g_1$ be Riemannian metrics on $X_0, X_1$ whose restrictions to $Y$ are equal to $g_Y$.
Assume that the dimension of the moduli space $M_P$ is $2d$ with $0 \leq d \leq 3$. Take homology classes $[\Sigma_1], \dots, [\Sigma_d]$ of $X$ represented by surfaces $\Sigma_1, \dots, \Sigma_d$ which intersect with $Y$. 
Put $\Sigma_l' = \Sigma_l \cap X_0$, $\Sigma_l'' = \Sigma_l \cap X_1$ and $\gamma_l = \Sigma_l \cap Y$.
We assume that $\gamma_l$ are diffeomorphic to $S^1$ for all $l$.
We show how to compute $\Psi_X([\Sigma_1], \dots, [\Sigma_d])$ from data of $X_0, X_1$ briefly under this situation. We use some facts about instantons which can be found in \cite{D Floer}.

\vspace{2mm}

Take a sequence $\{ T^{\alpha} \}_{\alpha=1}^{\infty}$ of positive real numbers which diverges to infinity. We have manifolds $X^{\alpha} = X_0 \cup ( Y \times [-T^{\alpha}, T^{\alpha}] ) \cup X_1$ which are diffeomorphic to $X$. The Riemannian metrics $g_Y$, $g_0$, $g_1$ induce Riemannian metrics $g^{\alpha}$ on $X^{\alpha}$. Take instantons $[A^{\alpha}] \in M_{P}(g^{\alpha})$. Then there is a subsequence $\{ [A^{\alpha'}] \}_{ \alpha' }$ such that
\[
[A^{\alpha'}] \longrightarrow ([A_0^{\infty}],\dots, [A_{r}^{\infty}]).
\]
Here $[A_0^{\infty}] \in M_{ \hat{X}_0  }(\rho(0))$, $[A_r^{\infty}] \in M_{ \hat{X}_1 }( \rho(r-1) )$ are instantons over $\hat{X}_0 = X_0 \cup Y \times \R_{ \geq 0 }$, $\hat{X}_1=X_1 \cup Y \times \R_{ \geq 0 }$ which converge to projectively flat connections $\rho(0)$, $\rho(r-1)$ at infinity.
$M_{ \hat{X}_0 }(\rho(0))$ is a moduli space of instantons over $\hat{X}_0$ with limit $\rho(0)$ and similarly for $M_{\hat{X}_1}(\rho(r-1))$.
For $i=1,\dots, r-1$, $[A_i^{\infty}] \in M^0_{ Y \times \R }( \rho(i-1), \rho(i) ) = M_{Y \times \R}( \rho(i-1), \rho(i) ) / \R$ are instantons over $Y \times \R$ with limits $\rho(i-1)$, $\rho(i)$. ( The action of $\R$ on $M_{Y \times \R}(\rho(i-1), \rho(i))$ is defined by translations. ) 
Since $\dim M_P$ is less than $8$, bubbling phenomena do not occur. As we will see below, we can take sections $s_l^{\alpha}$ of $\cL^{\otimes 2}_{\Sigma_l} \rightarrow M_{P}(g^{\alpha})$ such that
\begin{equation} \label{eq section tensor}
s^{\alpha}_{\Sigma_l}([A^{\alpha}]) \longrightarrow
s_{\hat{\Sigma}_l'}([A_0^{\infty}]) \boxtimes s_{\Gamma_l}([A_1^{\infty}]) \boxtimes \cdots \boxtimes s_{ \Gamma_l} ([A_{r-1}^{\infty}]) \boxtimes s_{ \hat{\Sigma}_l'' }([A_r^{\infty}]),
\end{equation}
where 
$\hat{\Sigma}_l' = \Sigma_{l}' \cup (\gamma_l \times \R_{ \geq 0})$,
$\hat{\Sigma}_l'' = \Sigma_{l}'' \cup (\gamma_l \times \R_{ \geq 0})$,
$\Gamma_l = \gamma_l \times \R$,
and
$s_{\hat{\Sigma}_l'}$, $s_{ \gamma_l \times \R }$, $s_{ \hat{\Sigma}_l'' }$ are sections of line bundles defined by families of twisted $\bar{\partial}$ operators on the surfaces. 
(See Definition \ref{def section covergence} and Proposition \ref{prop section}.)
Suppose that all $[A^{\alpha}]$ lie in the intersection
\[
M_{P}(g^{\alpha}) \cap V_{\Sigma_1} \cap \cdots \cap V_{\Sigma_d},
\]
then at least one of the components of the limit of $s_{ \Sigma }^{ \alpha }([A^{\alpha}])$ vanishes.
A dimension counting argument shows that $r$ is $1$ and that 
\[
[A_0^{\infty}] \in 
M_{ \hat{X}_0 }(\rho(0);L) = 
M_{\hat{X}_0}(\rho(0)) \cap \bigcap_{l \in L} V_{ \hat{\Sigma}_l' },
\quad
[A_1^{\infty}] \in 
M_{\hat{X}_1} (\rho(0);L^c) =
M_{\hat{X}_1} (\rho(0)) \cap \bigcap_{l \in L^c} V_{ \hat{\Sigma}_l'' } .
\]
Here $L$ is a subset of $\{ 1, \dots, d \}$ and $L^c$ is its complement.
We denote the number of elements of $L$ by $|L|$.
Then we also have
\begin{equation} \label{eq L rho}
\dim M_{ \hat{X}_0 }(\rho(0)) = 2|L|, 
\quad
\dim M_{ \hat{X}_1 }(\rho(0))=2|L^c|
\end{equation}
and $M_{ \hat{X}_0 }(\rho(0); L)$, $M_{ \hat{X}_1 }(\rho(0); L^c)$ are finite (if the sections over the moduli spaces satisfy some transversality conditions and behave nicely on the end of the moduli spaces).
A standard theory of gluing of instantons shows that
\[
M_P (g^{\alpha}) \cap V_{\Sigma_1} \cap \cdots \cap V_{\Sigma_d} \cong
\bigcup_{L} \bigcup_{[\rho]} M_{ \hat{X}_0 }( \rho; L) \times M_{ \hat{X}_1 }(\rho; L^c).
\]
Here $[\rho]$ runs over the gauge equivalence classes of flat connections satisfying (\ref{eq L rho}).
This implies that
\begin{equation} \label{eq gl formula}
\Psi_{X,P}([\Sigma_1], \dots, [\Sigma_d])=
\sum_{L} \sum_{ [\rho] }
\# M_{ \hat{X}_0 }(\rho; L) \cdot \# M_{ \hat{X}_1 }(\rho; L^c).
\end{equation}
From this formula, formal sums
\[
\psi_{X_0} = \sum_{L} \sum_{ [\rho] } n_{ 0 }(\rho;L) \cdot [\rho] \otimes \gamma_{L},
\quad
\psi_{X_1} = \sum_{L} \sum_{ [\rho] } n_{1}(\rho;L^c) \cdot [\rho] \otimes \gamma_{L^c}
\]
recover the Donaldson invariant.
Here
\[
n_{ 0 }(\rho;L) := \# M_{ \hat{X}_0 }(\rho; L),
\quad
n_{1}(\rho;L^c) := \# M_{ \hat{X}_1 }(\rho; L^c).
\]
We consider $\psi_{X_0}$ as an element of the vector space $CFF(Y; \ungam)$ spanned by a set
\[
\{ \ [\rho] \otimes \gamma_{L} \ | \ \text{$L$ and $\rho$ satisfy (\ref{eq L rho})} \  \}.
\]
The formal sums $\psi_{X_0}$, $\psi_{X_1}$ depend on the metrics and sections. We will define a boundary map
\[
\partial:CFF(Y;\ungam) \longrightarrow CFF(Y;\ungam)
\] 
such that the composition $\partial \circ \partial$ is identically zero, and show that $\partial \psi_{X_0}=0$ and the class $\Psi_{X_0} = [\psi_{X_0}] \in HFF(Y; \ungam) =H(CFF(Y;\ungam), \partial)$ is independent of the metric and sections.

There is a pairing
\[
< \ , \ > : CFF(Y; \ungam) \otimes CFF( \bar{Y}; \ungam) \longrightarrow \Q
\]
such that $[\rho] \otimes \gamma_{L}$ and $[\rho] \otimes \gamma_{L^c}$ are dual to each other. We can see that the pairing induces a pairing
\[
HFF(Y; \ungam) \otimes HFF( \bar{Y}; \ungam) \longrightarrow \Q.
\]
The formula (\ref{eq gl formula}) implies
\[
\Psi_{X,P}([\Sigma_1], \dots, [\Sigma_d]) = < \Psi_{X_0}, \Psi_{X_1}>.
\]
This is the gluing formula.

\subsection{Fukaya's construction}
\label{ss Fukaya}

Let $Y$ be a closed, oriented 3-manifold. Take a Riemannian metric $g_Y$ on $Y$ and a $U(2)$-bundle $Q$ over $Y$, and fix a connection $a_{\det}$ on the $U(1)$-bundle $Q_{\det}$ induced by $Q$. We consider connections on $Q$ which induce the fixed connection $a_{\det}$ on $Q_{\det}$.
Let $\cA_Q$ be the space of connections on $Q$ with fixed determinant and $\cG_Q$ be the space of automorphisms of $Q$ of determinant $1$.
The Chern-Simons functional is an $S^1$-valued functional on the quotient space $\cB_Q = \cA_Q/ \cG_Q$.
The critical points are the gauge equivalence classes of projectively flat connections. If the Hessian of the Chern-Simons functional at a projectively flat connection is non-generate, we say that the projectively flat connection is non-degenerate.
For simplicity, we will refer projectively flat connections over $Y$ with fixed determinant as ``flat connections".
Throughout this section we assume the following hypothesis.

\begin{hypothesis} \label{ass flat conn}
All flat connections on $Q$ are irreducible and non-degenerate.
\end{hypothesis}

We are always able to perturb the Chern-Simons functional such that any critical points are non-degenerate. See \cite{Fl}, \cite{D Floer}.

Let $R(Y)=R(Y, Q)$ be the set of gauge equivalence classes of flat connections on $Q$. It follows from Hypothesis \ref{ass flat conn} that $R(Y)$ is a finite set.
We define a $\Z_8$-grading function $\delta_Y$ on $R(Y)$ as follows. 
Let $\pi$ be the projection from $Y \times \R$ to $Y$.
Fix a  flat connection $\rho_0$ on $Q$. For each flat connection $\rho$ on $Q$, choose a connection $A_0=A_0(\rho,\rho_0)$ over $Y \times \R$, which is compatible with $\pi^* a_{\det}$, such that
\[
A_0=\left\{
\begin{array}{ll}
\pi^* \rho & \text{on $Y \times (-\infty, -1)$}, \\
\pi^* \rho_0 & \text{on $Y \times (1, \infty)$.}
\end{array}
\right.
\]
Then we have an operator
\begin{equation} \label{eq D}
D_{A_0}=d_{A_0}^* + d_{A_0}^+:
L_4^2( \Lambda_{Y \times \R}^1 \otimes \pi^* \fg_{Q} )
\longrightarrow
L_3^2 ( (\Lambda^0_{ Y \times \R } \oplus \Lambda_{Y \times \R}^+) \otimes \pi^* \fg_{Q} ).
\end{equation}
Here $\fg_Q$ is the bundle of trace free, skew adjoint, endomorphisms of the rank-two complex vector bundle $E$ associated with $Q$.
Under Hypothesis \ref{ass flat conn}, this operator is a Fredholm operator and we have the numerical index $\ind D_{A_0} \in \Z$.
We set
\[
\delta_Y([\rho]) \equiv \ind D_{A_0} \mod 8.
\]
We can show that this depends only on the gauge equivalence class of $\rho$ (and $\rho_0$).

For $j \in \Z$ we write $CF_j(Y)$ for the $\Q$-vector space spanned by 
\[
\{ \ [\rho] \in R(Y) \ | \ \delta_Y([\rho]) \equiv j \mod 8 \ \}.
\]
Let $d$ be an integer with $1 \leq d \leq 3$ and $\gamma_l \cong S^1$ be a loop in $Y$ for $l = 1, \dots, d$. We write $\ungam$ for $\{ \gamma_l \}_{l=1}^d$.
We define the Fukaya-Floer chain group $CFF_*(Y;\ungam)$ by
\[
CFF_j(Y; \underline{\gamma}) := 
\bigoplus_{\beta=0}^d
\bigoplus_{
\begin{subarray}{c}
L \subset \{ 1, \dots, d \} \\
|L| = \beta
\end{subarray}
} 
CF_{j-2 \beta}(Y) \otimes \Q < \gamma_L >,
\]
where $\gamma_L := \gamma_{l_1} \cdots \gamma_{l_\beta} \in \Sym  \Q< \gamma_1,\dots, \gamma_d >$ for $L=\{ l_1, \dots, l_{\beta} \}$.
We define a boundary operator
\[
\partial:CFF_j(Y; \underline{\gamma}) \longrightarrow CFF_{j-1}(Y; \underline{\gamma})
\]
as follows.
Take two generators
\[
\begin{split}
[\rho] \otimes \gamma_{L_1} 
&\in CF_{j-2\beta_1}(Y) \otimes \Q< \gamma_{L_1}> \subset CFF_j(Y; \ungam), \\
[\sigma] \otimes \gamma_{L_2} 
&\in CF_{j-2\beta_2-1} \otimes \Q <\gamma_{L_2}> \subset CFF_{j-1}(Y; \ungam).
\end{split}
\]
Then we have a moduli space $\M$ of instantons with limits $\rho, \sigma$ and the  dimension is $2(\beta_2-\beta_1)+1$.
We write $\0M$ for the quotient $\M/\R$, where the action of $\R$ is defined by translations.
When $L_1 \subset L_2$, we define
\begin{equation} \label{eq FF boundary}
< \partial ( [\rho] \otimes \gamma_{L_1} ), [\sigma] \otimes \gamma_{L_2}>:=
`` < c_1( \cL_{ l_1 }^{ \otimes 2 } ) \cup \cdots \cup c_1( \cL_{ l_\beta }^{\otimes 2}  ), [ M_{Y \times \R}^0(\rho, \sigma)  ] > ",
\end{equation}
where $\{ l_1, \dots, l_{\beta} \} = L_2 \backslash L_1$ and
$\cL_{l_i}^{\otimes 2}$ are line bundles defined by families of twisted $\bar{\partial}$ operators over $\gamma_{l_i} \times \R$.
Since the moduli spaces are non-compact in general, we must specify the meaning of the pairing, and it will be done later.
When $L_1$ is not included in $L_2$, we define $< \partial ( [\rho] \otimes \gamma_{L_1} ), [\sigma] \otimes \gamma_{L_2}>$ to be zero. 
The matrix elements $< \partial ( [\rho] \otimes \gamma_{L_1} ), [\sigma] \otimes \gamma_{L_2}>$ give the boundary operator $\partial$.
We will show that $\partial \circ \partial$ is identically zero.
The Fukaya-Floer homology group $HFF_*(Y;\underline{\gamma})$ is defined to be the homology group of $(CFF_*(Y;\underline{\gamma}), \partial)$.

\vspace{2mm}

We give the precise definition of (\ref{eq FF boundary}). To do this, we introduce some spaces of connections on $Y \times \R$ and $\gamma_l \times \R$.  
For flat connections $\rho, \sigma$ on $Q$, take a smooth connection $A_0 = A_0(\rho, \sigma)$ on $\pi^* Q$ as before.
Let $\tau>0$ be a small positive number and we set
\[
\cA_{Y \times \R}^*(\rho,\sigma) :=
\{ \ A_0 + a \ | \ a \in L^{2, \tau}_{4} ( \Lambda_{Y \times \R}^1 \otimes \pi^* \fg_ Q ) \ \}.
\]
The wighted Sobolev space $L^{2, \tau}_4$ is defined as follows.
Take a function $W_{\tau}$ on $Y \times \R$ such that
\[
\begin{array}{ll}
W_{\tau}(y,t) > 0 & \text{ for $^{\forall} (y, t) \in Y \times \R$, } \\
W_{\tau}(y,t) = e^{ \tau | t | } & \text{ for $|t|>1$. }
\end{array}
\]
For a smooth, compact supported section $f$, we define the weighted $L^2_{4}$-norm by
\[
\| f \|_{ L_{4}^{2, \tau}}^2 = 
\sum_{k=0}^4 \| \nabla_{A_0}^k ( W_{\tau} f) \|_{L^2}^2.
\]
The weighted Sobolev space $L^{2, \tau}_4$ is the completion of the space of smooth, compact supported sections.
Note that since $\rho, \sigma$ are irreducible, every connection in $\cA_{Y \times \R}^*(\rho, \sigma)$ is irreducible. 
By Hypothesis \ref{ass flat conn}, we need not to introduce the weighted Sobolev space for the construction of the moduli space. However we need the weighted Sobolev space to define the determinant line bundles as explained below. 

We introduce a gauge group acting on the space of connections. We put
\[
\cG_{Y \times \R} :=
\{ \ 
g \in L^2_{5, loc}( \Aut \pi^* Q ) \ | \ d_{A_0}g \cdot g^{-1} \in L_4^{2,\tau}
\ \}.
\]
Then $\cG_{Y \times \R}$ acts on $\cA_{Y \times \R}^*(\rho, \sigma)$ by the gauge transformations. We denote the quotient $\cA^*_{Y \times \R}/ \cG_{Y \times \R}$ by $\cB_{Y \times \R}^* (\rho, \sigma)$ . 
For $g \in \cG_{Y \times \R}$, we have $g(y,t) \rightarrow 1$ or $g(y,t) \rightarrow -1$ as $t \rightarrow \pm \infty$.  ( See Proposition 4.7 in \cite{D Floer}. ) We set
\[
\cG_{Y \times \R}^0 := 
\{ \ 
g \in \cG_{Y \times \R} \ | \ \lim_{t \rightarrow \pm \infty} g(y, t) = 1
\ \}.
\]
We write $\tilde{\cB}^*_{Y \times \R}(\rho, \sigma)$ for $\cA_{Y \times \R}^* (\rho, \sigma) / \cG_{Y \times \R}^0$. Since $\cG_{Y \times \R}/ \cG_{Y \times \R}^0$ is isomorphic to $\Z_2 \times \Z_2$, we have a natural action of $\Z_2 \times \Z_2$ on $\tilde{\cB}_{Y \times \R}^*(\rho, \sigma)$, and $\cB_{Y \times \R}^*(\rho, \sigma)$ is identified with $\tilde{\cB}^*_{ Y \times \R }(\rho, \sigma) / \Z_2 \times \Z_2$.

Let $M_{Y \times \R}(\rho, \sigma) \subset \cB_{Y \times \R}^*(\rho, \sigma)$ be the moduli space of instantons with limits $\rho, \sigma$. We can perturb the instanton equation such that the operators $D_A$ are surjective for all $[A] \in \0M$ (\cite{Fl, D Floer}). For simplicity, we always assume the following.

\begin{hypothesis} \label{hypo D}
The operators $D_A$ defined by (\ref{eq D}) are surjective for all $[A] \in \0M$.
\end{hypothesis}

Under this hypothesis, the moduli space $\0M$ is smooth of expected dimension.

Next we introduce spaces of connections and gauge groups over $\Gamma_l = \gamma_l \times \R$. Let $A_{l} = A_{l}(\rho, \sigma)$ be the restriction of the fixed connection $A_0(\rho, \sigma)$ to $\Gamma_l$. We set
\[
\begin{split}
\cA_{\Gamma_l}(\rho, \sigma) 
&:= 
\{ \ 
A_{l} + a \ | \ a \in L_{3}^{2, \tau}( \Lambda_{\Gamma_l}^{1} \otimes \pi^* \fg_{Q}|_{\Gamma_l} )
\ \}, \\
\cG_{\Gamma_l}^0 
&:=
\{ \ 
g \in L^{2}_{3, loc}( \Aut \pi^* Q |_{\Gamma_l} ) \ | \ 
d_{A_{l}} g \cdot g^{-1} \in L^{2, \tau}_{3},
\lim_{t \rightarrow \pm \infty} g = 1
\ \}.
\end{split}
\]
We denote the quotient space $\cA_{\Gamma_l}(\rho, \sigma) / \cG_{ \Gamma_l }^0$ by $\tilde{\cB}_{\Gamma_l}(\rho, \sigma)$.
Note that the restrictions of $\rho, \sigma$ to $\gamma_l$ may be reducible, and hence some connections in $\cA_{\Gamma_l}(\rho, \sigma)$ are reducible.

We define the determinant line bundle over $\tilde{\cB}_{\Gamma_l}(\rho, \sigma)$. 
We need a spin structure on $\Gamma_l$.
Since $H^1(\Gamma_l;\Z_2)$ is isomorphic to $\Z_2$, there are two spin structures on each $\Gamma_l$ (up to isomorphism).
We fix a spin structure on $\gamma_l$ which represent the trivial class in the 1-dimensional spin bordism group. 
This spin structure induces a spin structure on $\Gamma_l$.
We use this spin structure.
(We will explain the reason why we take this spin structure in Remark \ref{rm spin} below.)
The spin structure induces a square root $K_{\Gamma_l}^{ \frac{1}{2} }$ of the canonical line bundle $K_{\Gamma_l}$.
For connections $A \in \cA_{\Gamma_l}(\rho, \sigma)$ we have the twisted $\bar{\partial}$ operators
\[
\bar{\partial}_A:
L^{2,(-\tau, \tau)}_3 ( K_{ \Gamma_l }^{ \frac{1}{2} } \otimes \pi^* E|_{\Gamma_l} )
\longrightarrow
L^{2,(-\tau, \tau)}_2 ( \Lambda_{ \Gamma_l }^{0,1} \otimes K_{ \Gamma_l }^{ \frac{1}{2} } \otimes \pi^* E|_{\Gamma_l} ).
\]
Here $L_{3}^{2, (-\tau, \tau)}, L_{2}^{2, (-\tau, \tau)}$ are the weighted Sobolev spaces with weight function $W'_{\tau} > 0$ such that
\[
W'_{\tau}(y,t) = e^{ \tau t } 
\quad
\text{ for $|t| > 1$}.
\]
We do not take the absolute value of $t$ in the exponent this time.
The operators are Fredholm operators for small $\tau > 0$. 
Since we have the universal bundles
\[
\tilde{\bE}_{ \Gamma_l } := \cA_{ \Gamma_l }(\rho, \sigma) \times_{\cG_{\Gamma_l}^0} ( \pi^* E|_{ \Gamma_l } )
\longrightarrow
\tilde{\cB}_{\Gamma_l} \times \Gamma_l
,
\]
we obtain complex line bundles
\[
\tilde{\cL}_{l}(\rho, \sigma) 
=
\big(
\det \Ind \{ \bar{\partial}_A \}_{ [A] }
\big)^{*}
\longrightarrow
\tilde{\cB}_{\Gamma_l}(\rho, \sigma).
\]
Let $\tilde{r}_l$ be the map from $\tilde{ \cB }^*_{Y \times \R}(\rho, \sigma)$ to $\tilde{\cB}_{\Gamma_l}(\rho, \sigma)$ defined by restricting connections to $\Gamma_l$. 
Then we have a natural action of $\Z_2 \times \Z_2$ on the pull-back $\tilde{r}_l^* \tilde{\cL}_{\Gamma_l}^{\otimes 2}$ and the action of the diagonal $\Z_2=\{ \pm (1, 1) \}$ is trivial. Hence we get the line bundle
\[
\cL_l^{\otimes 2}(\rho, \sigma) := 
\tilde{r}_{l}^* \tilde{\cL}_{l}^{ \otimes 2 } (\rho, \sigma)/ \Z_2 \times \Z_2
\longrightarrow
\cB_{Y \times \R}^*(\rho, \sigma)
\]
for each $l$.

\begin{remark} \label{rm spin}
We explain the reason why we choose the spin structure on $\Gamma_l$ induced by the spin structure on $\gamma_l$ representing the trivial class of 1-dimensional spin bordism group $\Omega_1^{spin} = \Z_2$.
When we prove the gluing formula for Donaldson invariants, we consider $\Gamma_l$ as a neck of a closed surface in a closed 4-manifold.  We need the restriction of a spin structure on the closed surface, which is used to define Donaldson invariants, to $\Gamma_l$.  This spin structure is induced by the spin structure on $\gamma_l$ which represent the trivial element in $\Omega_1^{spin}$.
\end{remark}

To define (\ref{eq FF boundary}) we need sections $s_l=s_l(\rho, \sigma)$ of $\cL_{l}^{\otimes 2}(\rho, \sigma)$ which behave nicely on the ends of the moduli spaces and satisfy suitable transversality conditions. 
We briefly recall some basic definitions and facts which are relevant to the end of the moduli spaces. (See \cite{D Floer} for details.)

For a real number $T$ let $c_T$ be the translation
\[
\begin{array}{ccc}
Y \times \R & \longrightarrow & Y \times \R \\
(y,t) & \longmapsto & (y, t+T).
\end{array}
\]
We call a sequence $\underline{T}$ of real numbers
\[
T_1 < \cdots < T_{r-1}
\]
a translation vector.

\begin{definition}
Let $\{ [A^{\alpha}] \}_{ \alpha =1 }^{\infty}$ be a sequence in $M_{Y \times \R}(\rho, \sigma)$.
We say that $\{ [A^{\alpha}] \}_{\alpha}$ is weakly convergent to 
\[
(([A_1],Z_1), \dots, \cdots, ([A_r], Z_r))
\]
for some $([A_i], Z_i) \in \big( M_{Y \times \R}(\rho(i-1), \rho(i)) \times \Sym^{s_i}(Y \times \R) \big) / \R$ if there is a sequence $\{ \underline{T}^{\alpha} \}_{\alpha}$ of translation vectors  with
\[
T^{\alpha}_i - T^{\alpha}_{i-1} \longrightarrow \infty
\]
as $\alpha \rightarrow \infty$ such that for each $i$ the translates $c^*_{T^{\alpha}_i}([A^{\alpha}])$ converge to $[A_i^{\infty}]$ over any compact sets of $(Y \times \R) \backslash Z_i$ and
$|c_{T^{\alpha}_i}^* (F_{A^{\alpha} }) |^2 d \mu_{Y \times \R}$
 weakly converge to
$| F_{ A_i^{\infty} } | ^2 d \mu_{Y \times \R} + \sum_{l = 1}^{s_i} \delta_{z_l}$.
Here $d\mu_{Y \times \R}$ is the volume form on $Y \times \R$, $Z_i = [z_1,\dots, z_{s_i}]$ and $\delta_{z_l}$ are the delta functions.

\end{definition}

\begin{proposition} \label{prop 1}
Any sequence in $M_{Y \times \R}^0(\rho, \sigma)$ has a weakly convergent subsequence.
\end{proposition}

Let $\{ [A^{\alpha}] \}_{\alpha}$ be a sequence of $M_{Y \times \R}^0(\rho, \sigma)$ which weakly converges to $(([A_1^{\infty}], Z_1), \dots, ([A_r^{\infty}], Z_r))$. It follows from the additivity of the index of the operator $D_A$ that if the dimension of $M_{Y \times \R}^0( \rho, \sigma)$ is less than $8$ then $Z_i$ are empty for all $i$.
In Section \ref{section Fukaya-Floer} and \ref{s torsion}, we only consider the case  when $\dim M_{Y \times \R}^0(\rho, \sigma) < 8$ and hence $Z_i$ are always empty. In Section \ref{s Thm B}, we will analyze the end of moduli spaces under the  situation where $Z_i$ are not empty.

The end of the moduli spaces are described by gluing maps. We consider the case  when $Z_i$ are empty.
Let $U_1, \dots, U_{r}$ be precompact, open sets of $M_{Y \times \R}^0(\rho, \rho(1)), \dots, M_{Y \times \R}^0(\rho(r-1), \sigma)$. Then we have a gluing map
\[
Gl:U_1 \times (T_0, \infty) \times \cdots \times (T_0, \infty) \times U_r 
\longrightarrow
M_{Y \times \R}^0 (\rho, \sigma)
\]
for some $T_0 > 0$. The map $Gl$ is a diffeomorphism onto its image.

\begin{proposition} \label{prop 2}
Let $\{ [A^{\alpha}] \}_{\alpha}$ be a sequence in $M_{Y \times \R}^0(\rho, \sigma)$ converging to some $([A_1^{\infty}],\dots, [A_r^{\infty}]) \in U_1 \times \cdots \times U_r$. Then for large $\alpha$, $[A^{\alpha}]$ are in the image of the gluing map.
\end{proposition}

We also have a gluing map
\[
\widehat{Gl}:
\cL_{l}^{\otimes 2}(\rho, \rho(1)) |_{U_1}
\boxtimes
\cdots
\boxtimes
\cL_{l}^{\otimes 2}(\rho(r-1), \sigma) |_{U_r}
\stackrel{\cong}{\longrightarrow}
\cL_{l}^{\otimes 2}(\rho, \sigma)|_{ \Im Gl }
\]
which covers $Gl$. 
For $\underline{T}=(T_1,\dots, T_{r-1})$ with $T_i > T_0$, we write $\widehat{Gl}_{\underline{T}}$ for the restriction of $\widehat{Gl}$ to $U_1 \times \{ T_1 \} \times \cdots \times \{ T_{r-1} \} \times U_r$.

\vspace{3mm}

Using these definitions and facts, we state the properties of sections of the line bundles which are required to define (\ref{eq FF boundary}). Let $\{ [A^{\alpha}] \}_{\alpha}$ be a sequence in $M_{Y \times \R}^0(\rho, \sigma)$ with limit $([A_1^{\infty}], \dots, [A_r^{\infty}])$. By the above proposition, for large $\alpha$, there are instantons $[A_i^{\alpha}] \in U_i$ and $T_i^{\alpha} > T_0$ such that
\[
[A^{\alpha}] = 
Gl([A_1^{\alpha}], T_1^{\alpha}, \dots, T_{r-1}^{\alpha}, [A_r^{\alpha}]).
\]

\begin{definition} \label{def section covergence}
Let $s_l$ be sections of the line bundles $\cL_{l}^{\otimes 2}(\rho, \sigma)$.
Under the above situation, we say that $s_l([A^{\alpha}])$ converge to $s_l([A_1^{\infty}]) \boxtimes \cdots \boxtimes s_l([A_{r}^{\infty}])$ if
\[
\| 
s_l([A^{\alpha}]) - 
\widehat{Gl}_{ \underline{T}^{\alpha} } \big( s_l([A_1^{\alpha}]) \boxtimes \cdots \boxtimes s_l( [A_r^{\alpha}] ) \big)
\|
\longrightarrow 0
\quad
\]
as $\alpha \rightarrow \infty$.
Here $\| \cdot \|$ is the norm on $\cL_l^{\otimes 2}(\rho, \sigma)$ induced by the $L^{2, (-\tau, \tau)}$-norms on the spaces of sections of $K_{\Gamma_l}^{\frac{1}{2}} \otimes \pi^* E$, $\Lambda^{0,1}_{ \Gamma_l } \otimes K_{\Gamma_l}^{ \frac{1}{2} } \otimes \pi^* E $.
\end{definition}

The following proposition is the key in this paper.

\begin{proposition} \label{prop section}
For flat connections $\rho, \sigma$ on $Q$ with $\dim M_{Y \times \R}^0 (\rho, \sigma) < 8$, we have sections $s_l(\rho, \sigma):M_{Y \times \R}^0(\rho, \sigma) \rightarrow \cL^{\otimes 2}_{l}(\rho,\sigma)$ which have the following properties:

\vspace{2mm}

\begin{enumerate}[(a)]
\item \label{tensor}
For any sequence $\{ [A^{\alpha}] \}_{\alpha}$ in $M_{Y \times \R}^0(\rho, \sigma)$  converging to some $([A_1^{\infty}], \dots, [A_{r}^{\infty}])$, 
\[
s_l([A^{\alpha}]) \longrightarrow 
s_l([A_1^{\infty}]) \boxtimes \cdots \boxtimes s_l([A_{r}^{\infty}])
\]
in the sense of Definition \ref{def section covergence}.

\vspace{2mm}

\item \label{transversality I}
Let $V_l = V_l(\rho, \sigma)$ be the zero locus of $s_l(\rho, \sigma)$. 
For $L \subset \{ 1, \dots, d \}$ with $\dim M^0_{Y \times \R}(\rho, \sigma) < 2 |L|$, the intersection
\[
M_{Y \times \R}^0(\rho, \sigma; L) := M_{Y \times \R}^0(\rho, \sigma) \cap \bigcap_{l \in L} V_l
\]
is empty.

\vspace{2mm}

\item \label{transversality II}
If $\dim M_{Y \times }^0(\rho, \sigma) = 2|L|$, the intersection $M_{Y \times \R}^0(\rho, \sigma;L)$ is transverse and compact. Hence the intersection is a finite set.
\end{enumerate}
\end{proposition}

Here we introduce the following definition.

\begin{definition}
If sections $s_{l}(\rho, \sigma)$ of $\cL^{\otimes 2}_{l}(\rho,\sigma)$ have the properties in Proposition \ref{prop section}, we call them admissible.
\end{definition}

The proof of Proposition \ref{prop section} will be given in the following two subsections.  
In this subsection, we assume that we have admissible sections $s_l(\rho, \sigma)$ and define the boundary operator for the Fukaya-Floer homology groups.

\vspace{2mm}

Let $d$ be an integer with $1 \leq d \leq 3$.
For subsets $L_1 \subset L_2 \subset \{ 1, \dots, d \}$ and flat connections $\rho, \sigma$ with $\dim M_{Y \times \R}^{0}(\rho, \sigma) = 2 | L_2 \backslash L_1|$, the intersections $M_{Y \times \R}^0(\rho, \sigma; L_2 \backslash L_1)$ are finite by the property (\ref{transversality II}) in Proposition \ref{prop section}. Hence we can count the number of points in the intersections. They numbers give the definition of (\ref{eq FF boundary}). More precisely we need to attach a sign $\pm 1$ to each point.
However we mention nothing about signs.
(The main purpose of this paper is to construct variants of Floer homology groups for 2-torsion instanton invariants.  They are defined over $\Z_2$ and we do not need signs for the construction. )

\begin{definition}
Let $\underline{\gamma} = \{ \gamma_{l} \}_{l=1}^d$ be a set of loops in $Y$, where $d$ is an integer with $1 \leq d \leq 3$.
For integers $\beta_1, \beta_2$ with $0 \leq \beta_1 \leq \beta_2 \leq d$, take generators $[ \rho ] \in CF_{j-2\beta_1}(Y)$, $[\sigma] \in CF_{j-2\beta_2-1}(Y)$ and choose subsets $L_1, L_2$ of $\{ 1,\dots, d \}$ with $|L_1|=\beta_1$, $|L_2|=\beta_2$. Then we put
\[
< \partial( [\rho] \otimes \gamma_{L_1} ), [\sigma] \otimes \gamma_{L_2} >
:=
\left\{
\begin{array}{cl}
\# M_{Y \times \R}^0 (\rho, \sigma; L_2 \backslash L_1) & 
\text{ if $L_1 \subset L_2$ } \\
0 & \text{otherwise.}
\end{array}
\right.
\]
We define $\partial:CFF_{j}(Y;\underline{\gamma}) \rightarrow CFF_{j-1}(Y;\underline{\gamma})$ by
\[
\partial([\rho] \otimes \gamma_{L_1}) :=
\sum_{\beta_2} \sum_{[\sigma]} \sum_{L_2}
< \partial ( [\rho] \otimes \gamma_{L_1} ), [\sigma] \otimes \gamma_{L_2} > 
[\sigma] \otimes \gamma_{L_2}.
\]

\end{definition}

We prove the following.

\begin{lemma} \label{lem complex}
$\partial \circ \partial = 0$.
\end{lemma}

This is given by counting the number of the ends of 1-dimensional moduli spaces. Let $\beta_1, \beta_2$ be integers with $0 \leq \beta_1 \leq \beta_2 \leq d$. Choose generators $[\rho] \in CF_{j-2\beta_1}(Y)$, $[\sigma] \in CF_{j-2\beta_2 -2}(Y)$ and $L_1 \subset L_2 \subset \{ 1, \dots, d \}$ with $|L_1|= \beta_1$, $|L_2|=\beta_2$. We have a reduced moduli space $\0M$ of dimension $2(\beta_2 - \beta_1) + 1 \ (< 8)$. Suppose that we have a sequence  $\{ [A^{\alpha}] \}_{ \alpha }$ in the intersection $M_{Y \times \R}^0 (\rho, \sigma; L_2 \backslash L_1)$ of formal dimension 1 which converges to some $([A_1^{\infty}],\dots, [A_r^{\infty}])$ with $r>1$.
(Note that Proposition \ref{prop section} does not assure that  $M_{Y \times \R}^0 (\rho, \sigma; L_2 \backslash L_1)$ is transverse, since  $\dim M_{Y \times \R}^0(\rho, \sigma) > 2| L_2 \backslash L_1 |$.)
First we show that $r=2$. Put
\[
L(i) = \{ \ l \in L_2 \backslash L_1 \ | \ s_l([A_i^{\infty}]) = 0 \ \}.
\]
For $l \in L_2 \backslash L_1$ and all $\alpha$, $s_l([A^{\alpha}]) = 0$, and $s_l([A^{\alpha}])$ converges to $s_l([A_1^{\infty}]) \boxtimes \cdots \boxtimes s_l([A_r^{\infty}])$. Hence there is a number $i(l) \in \{ 1, \dots , r \}$ such that
\[
s_l([A^{\infty}_{i(l)}]) = 0
\]
for each $l \in L_2 \backslash L_1$. This means that $l \in L_2 \backslash L_1$ lies in $L(i(l))$. Therefore we have
\[
\beta_2 - \beta_1 = | L_2 \backslash L_1 | \leq \sum_{i=1}^{r} | L(i) |.
\]
Since $[A_i^{\infty}]$ are included in $M_{Y \times \R}^0(\rho(i-1), \rho(i); L(i))$, the intersection are not empty. The transversality condition (\ref{transversality I}) in Proposition \ref{prop section} implies that
\[
2|L(i)| + 1 \leq \dim M_{Y \times \R} (\rho(i-1), \rho(i))
\]
for each $i$.
From the additivity of the index, we have
\[
\begin{split}
2(\beta_2 - \beta_1) + 2
&=
\dim M_{Y \times \R} (\rho, \sigma) \\
&=
\sum_{i=1}^r \dim M_{Y \times \R}(\rho(i-1), \rho(i)) \\
&\geq
\sum_{i=1}^r ( 2|L(i)| + 1) \\
&\geq 2(\beta_2 - \beta_1) + r.
\end{split}
\]
Therefore we get $r \leq 2$. We assumed that $r > 1$, so $r=2$. 

Put $L := L(1) \coprod L_1$. Using (\ref{transversality I}), (\ref{transversality II}) in Proposition \ref{prop section}, we can easily see that 
\[
[A_1^{\infty}] \in M_{Y \times \R}^0(\rho, \rho(1); L \backslash L_1), 
\quad
[A_2^{\infty}] \in M_{Y \times \R}^0(\rho(1), \sigma; L_2 \backslash L)
\]
and that 
\[
\dim M_{Y \times \R}^0(\rho, \rho(1)) = 2|L \backslash L_1|, 
\quad
\dim M_{Y \times \R}^0(\rho(1), \sigma; L_2 \backslash L) = 2|L_2 \backslash L|.
\]

Conversely for each $[ \underline{A} ]= ([A_1], [A_2]) \in M_{Y \times \R}^0(\rho, \rho(1); L \backslash L_1) \times M_{Y \times \R}^0(\rho(1),\sigma; L_2 \backslash L)$, we have gluing maps
\[
\begin{split}
Gl_{[ \underline{A}] } &:
U_{[A_1]} \times (T_0, \infty) \times U_{[A_2]} 
\longrightarrow 
\0M, \\
\widehat{Gl}_{[ \underline{A}] } &:
\cL^{ \otimes 2 }_{l}(\rho, \rho(1))|_{U_{[A_1]}} \boxtimes \cL_{l}^{\otimes 2}(\rho(1), \sigma)|_{U_{[A_2]}}
\longrightarrow
\cL_{l}^{\otimes 2}(\rho, \sigma)|_{ \Im Gl_{[ \underline{A}] } }
\end{split}
\]
for some precompact open neighborhoods $U_{[A_1]}, U_{[A_2]}$ of $[A_1]$, $[A_2]$ and positive number $T_0 > 0$. When $U_{  [A_1] }$, $U_{[A_2]}$ are sufficiently small, the transversality conditions in Proposition \ref{prop section} imply that intersections
\[
U_{[A_1]} \cap \bigcap_{ l \in L \backslash L_1 } V_l, \quad
U_{ [A_2] } \cap \bigcap_{ l \in L_2 \backslash L } V_l
\]
are transverse and that
\begin{equation*} \label{eq intersections}
\begin{split}
&U_{[A_1]} \cap \bigcap_{ l \in L \backslash L_1 } V_l = \{ [A_1] \}, 
\quad
\partial \bar{U}_{[A_1]} \cap \bigcap_{ l \in L \backslash L_1 } V_l = \emptyset, 
\quad
\bar{U}_1 \cap V_l = \emptyset  \quad ( l \in L_2 \backslash L), \\
&U_{[A_2]} \cap \bigcap_{ l \in L_2 \backslash L } V_{l} = \{ [A_2] \}, \quad
\partial \bar{U}_{[A_2]} \cap \bigcap_{ l \in L_2 \backslash L } V_l = \emptyset, 
\quad
\bar{U}_2 \cap V_l = \emptyset  \quad ( l \in L \backslash L_1).
\end{split}
\end{equation*}
Put 
\[
s_l':= \widehat{Gl} \big( s_l(\rho, \tau ) \boxtimes s_l(\tau, \sigma) \big):
U_{[A_1]} \times (T_0, \infty) \times U_{[A_2]} \longrightarrow
\cL_{l}^{\otimes 2} (\rho, \sigma)|_{ \Im Gl_{\unA} } 
\]
and let $V_l'$ be the zero locus of $s_l'$. Then for each $T_1 > T_0$ we have
\[
\begin{split}
( U_{[A_1]} \times \{ T_1 \} \times U_{[A_2]} ) \cap
\bigcap_{ l \in L_2 \backslash L_1  } V_l' 
&=
\{ ( [A_1], T_1, [A_2]) \}, \\
\partial \overline{ ( U_{[A_1]} \times \{ T_1 \} \times U_{[A_2]} )} 
\cap
\bigcap_{ l \in L_2 \backslash L_1  } V_l' 
&= \emptyset
\end{split}
\]
and the first intersection is transverse. The condition (\ref{tensor}) in Proposition \ref{prop section} means that $s_l(\rho, \sigma)$ and $s_l'$ are close to each other on the end of the moduli space. Hence for $T_1$ large enough, the intersection
\[
( U_{ [A_1] } \times \{ T_1 \} \times U_{[A_2]}) \cap
\bigcap_{ l \in L_2 \backslash L_1 } V_{l}
\]
 is transverse and consists of a single point which is close to $( [A_1], T_1, [A_2] )$. We consider a subset
\[
M' = 
M_{Y \times \R}^0 (\rho, \sigma; L_2 \backslash L_1) 
\ \backslash \ 
\bigcup_{ ([A_1], [A_2]) }
E ([A_1], [A_2],  T_1),
\]
of the moduli space. Here
\[
E ([A_1], [A_2],  T_1) =
Gl_{ [A_1], [A_2] }(U_{[A_1]} \times (T_1, \infty) \times U_{ [A_2] }).
\]
Then $M'$ is compact and if we perturb sections $s_l(\rho, \sigma)$ outside of neighborhoods of the triples $([A_1], T_1, [A_2])$, $M'$ becomes a smooth manifold of dimension $1$. Moreover there is a natural identification
\begin{equation} \label{eq M' end}
\partial M'
\cong
\bigcup_{
\begin{subarray}{c}
\beta \\
\beta_1 \leq \beta \leq \beta_2 
\end{subarray}
} 
\bigcup_{
\begin{subarray}{c}
[\tau] \\
\delta_{Y} ([\tau]) \equiv \\
j - 2 \beta - 1 \mod 8
\end{subarray}
} 
\bigcup_{ 
\begin{subarray}{c}
L \\
L_1 \subset L  \subset L_2 \\
| L | = \beta
\end{subarray}
}
M_{Y \times \R}^0 (\rho, \tau; L \backslash L_1) \times
M_{Y \times \R}^0 (\tau, \sigma; L_2 \backslash L).
\end{equation}
By counting the number of $\partial M'$ with signs, we get
\[
< \partial \partial (\rho \otimes \gamma_{L_1}), \sigma \otimes \gamma_{L_2} > = 0
\]
and this gives 
\[
\partial \circ \partial = 0
\]
as required.

\begin{definition}
$HFF_*(Y; \ungam) := H_*( CFF_*(Y; \ungam), \partial)$.
\end{definition}

\subsection{Construction of $s_l(\rho, \sigma)$}
\label{ss sections}

In this section, we will prove Proposition \ref{prop section}.

First, for flat connections $\rho, \sigma$ with $\dim \0M < 8$, we take locally finite open covers $\{ U_{\lambda} \}_{\lambda \in \Lambda(\rho, \sigma)}$ of $M_{Y \times \R}^0(\rho, \sigma)$ with $U_{\lambda}$ precompact as follows.
We will use a partition of unity associated with the open cover to construct admissible sections.

If $\dim M^0_{Y \times \R}(\rho, \sigma) = 0$, $M_{Y \times \R}^0(\rho, \sigma)$ is compact. Put $U_0 = M_{Y \times \R}^0(\rho, \sigma)$. Then $\{ U_0 \}$ is the required open cover.
Fix an integer $m$ with $1 \leq m \leq 7$ and suppose that we have open covers $\{ U_{\lambda} \}_{\lambda \in \Lambda(\rho, \sigma)}$ for flat connections $\rho, \sigma$ with $\dim M_{Y \times \R}^0(\rho, \sigma) \leq m-1$. We consider flat connections $\rho, \sigma$ with $\dim \0M = m$. For each $\underline{\lambda} = ( \lambda_1, \dots, \lambda_{r} ) \in \Lambda(\rho, \rho(1)) \times \cdots \times \Lambda(\rho(r-1), \sigma)$, we have a gluing map
\[
Gl_{ \underline{\lambda} }:
U_{\lambda_1} \times (T_{ \underline{\lambda} }, \infty) \times \cdots \times (T_{ \underline{\lambda} }, \infty) \times U_{\lambda_r} 
\longrightarrow
\0M
\]
for some $T_{\unlam} > 0$. 
We can suppose that $Gl_{\unlam}$ extends to the closures $\overline{U}_{\lam_i}$ of $U_{\lam_i}$.
For $\unn =(n_1,\dots, n_{r-1}) \in ( \Z_{ \geq 0 } )^{r-1}$, we put
\[
U_{\unlam, \unn} :=
Gl_{\unlam}
\left(
U_{\lam_1} \times \big( T_{ \unlam } + n_1, T_{\unlam} + n_1 + \frac{3}{2} \big) \times \cdots \times \big( T_{\unlam} + n_{r-1}, T_{\unlam} + n_{r-1} + \frac{3}{2} \big) \times U_{\lam_r}
\right).
\]
This is a precompact, open set of $\0M$. 
We write $N$ for the complement of the union of the sets $U_{\unlam, \unn}$:
\[
N := \0M \ \backslash \ \bigcup_{ \unlam, \unn } U_{\unlam, \unn}.
\]

\begin{lemma}
Under the above notations and assumptions, $N$ is compact.
\end{lemma}

\begin{proof}
Let $\{ [A^{\alpha}] \}_{\alpha}$ be a sequence in $N$. By Proposition \ref{prop 1}, there is a subsequence $\{ [A^{\alpha'}] \}_{\alpha'}$ such that 
\[
[A^{\alpha'}] \longrightarrow 
( [A_1^{\infty}], \dots, [A_{r}^{\infty}] ).
\]
Note that $Z_i$ are empty since the dimension of the reduced moduli space is less than $8$.
 If $r>1$, it follows from Proposition \ref{prop 2} that $[A^{\alpha'}]$ lie in the image of $Gl_{\unlam}$ for large $\alpha'$ and  some $\unlam$.
This is a contradiction to the fact that $[A^{\alpha'}]$ are in the complement $N$ of the images of the gluing maps.

\end{proof}

Take a small neighborhood $U_0$ of $N$ in $M_{Y \times \R}^0 (\rho, \sigma)$, and put 
\[
\{ U_{\lam } \}_{ \lam \in \Lambda(\rho, \sigma) } := 
\{ U_0 \} \cup
\bigcup_{r=1}^m \bigcup_{\rho(1),\dots, \rho(r-1)}
\bigcup_{\unlam } \ \bigcup_{\unn} \ 
\{ U_{\unlam, \unn} \}.
\]
Then $\{ U_{\lam} \}_{\lam \in \Lam}$ is a locally finite open cover of $\0M$ with $U_{\lam}$ precompact.
We impose a condition on $T_{\unlam}$, where $\unlam = ( \lambda_1, \dots, \lambda_r)$.
If $\lam_i$ has the form $\lam_i=(\lam_1',\dots, \lam_{r'}'; n_1', \dots, n_{r'-1}')$, we suppose that $T_{\unlam}$ satisfies the inequality
\begin{equation} \label{eq condition on T}
T_{\unlam} > \sum_{k} ( T_{\unlam'} + n_k' )
\end{equation}
where $\unlam'=(\lam_1',\dots, \lam_{r'}')$. 
This condition will be used for the proof of Proposition \ref{prop section}.

\vspace{2mm}

Next we construct admissible sections $s_l(\rho, \sigma)$ using a suitable partition of unity associated with the open cover $\{ U_{\lambda} \}_{\lambda}$.
We again do it by induction on $\dim M_{Y \times \R}^0(\rho, \sigma)$.
Let $\rho$ and $\sigma$ be flat connections with $\dim \0M = 0$.
For generic sections $\sl$ on $\0M$, the zero locus are empty, and $\sl$ have the properties in Proposition \ref{prop section}. 
Fix an integer $m$ with $1 \leq m \leq 7$. Assume that for $\rho, \sigma$ with $\dim \0M \leq m-1$ we have sections $\sl$ on $\0M$ having the properties in Proposition \ref{prop section}. 
We need to construct admissible sections $s_l(\rho, \sigma)$ for $\rho, \sigma$ with $\dim \0M = m$.
We need some notations and lemmas.

Take subsets $I=\{ i_1, \dots, i_s \}$, $J=\{ j_1, \dots, j_t \}$ of $I_{r-1} := \{ 1, \dots, r-1 \}$ with $I \cap J = \emptyset$, $I \cup J = I_{r-1}$. Here $i_1 < \dots < i_{s}$, $j_1 < \cdots < j_t$. For $\unn = (n_1, \dots, n_{r-1}) \in (\Z_{ \geq 0 })^{r-1}$ we write $\unn_{I} = (n_{i_1}, \dots, n_{i_s})$, $\unn_{J}= (n_{j_1}, \dots, n_{j_t})$, and for $\unlam=(\lambda_1,\dots, \lam_{r})$ we write $| \unlam |=r$.

\begin{lemma} \label{lem partition}
Let $I, J$ be non-empty subsets of $I_{r-1}$ as above.
There is a partition of unity $\{ f_{\lam} \}_{\lam \in \Lam}$ satisfying the following condition. 
For $(\unlam, \unn) \in \Lam$ with $| \unlam |= r$, there exists a positive integer $N = N(\unlam, \unn_I) > 0$ depending on $\unlam$ and $\unn_I$ such that if $n_j > N$ for all $j \in J$ then $f_{\unlam, \unn}$ is identically zero.
\end{lemma}

Using the partition of unity, we define sections $\sl$ on $\0M$ by
\begin{equation} \label{eq section}
\sl = 
\sum_{ r=2 }^m \sum_{ \rho(1), \dots, \rho(r-1) } 
\sum_{\unlam, \unn} 
f_{\unlam, \unn} \ \widehat{Gl}_{\unlam} \big( s_l(\rho, \rho(1)) \boxtimes \cdots \boxtimes s_l(\rho(r-1), \sigma) \big).
\end{equation}
Here $s_l(\rho, \rho(1)), \dots, s_l(\rho(r-1), \sigma)$ are admissible sections.
We will show that we can perturb $\sl$ on a compact set in $M_{Y \times \R}^0(\rho, \sigma)$ such that $\sl$ have the properties in Proposition \ref{prop section}.

\vspace{2mm}

To prove Lemma \ref{lem partition}, we need to remove extra open sets from the open cover $\{ U_{\lam} \}_{\lam \in \Lam}$.
We consider the condition
\begin{equation} \label{eq open set}
U_{\unlam, \unn} \subset 
\bigcup_{ 
\begin{subarray}{c}
(\unlam', \unn') \in \Lam \\
| \unlam' | < | \unlam |
\end{subarray}
}
U_{ \unlam', \unn' }.
\end{equation}
Put
\[
\begin{split}
\Lambda' = \Lambda'(\rho, \sigma) &:=
\{ \ (\unlam, \unn) \in \Lam  \ | \ 
\text{ $U_{\unlam, \unn}$ satisfies (\ref{eq open set})} \  
\} \\
\Lambda''= \Lambda'' (\rho, \sigma) &:= 
\Lam \ \backslash \ \Lambda'.
\end{split}
\]
By definition, $\{ U_{\lambda} \}_{\lam \in \Lambda''}$ is still an open cover of $\0M$. Let $\{ f''_{\lam} \}_{\lam \in \Lambda''}$ be a partition of unity associated with $\{ U_{\lam} \}_{\lam \in \Lambda''}$. For $\lambda \in \Lam$, we define $f_{\lam}$  by
\[
f_{\lam} := 
\left\{
\begin{array}{cc}
f''_{\lam} & \text{ if $\lam \in \Lambda''$ } \\
0 & \text{otherwise.}
\end{array}
\right.
\]
Then $\{ f_{\lambda} \}_{\lambda \in \Lambda(\rho, \sigma)}$ is a partition of unity associated with $\{ U_{\lambda} \}_{ \Lam }$.
We show that this partition of unity has the property of Lemma \ref{lem partition}.
It is sufficient to show the following:

\begin{lemma} \label{lem open set}
Let $I, J$, $( \unlam, \unn) \in \Lam$ be as in Lemma \ref{lem partition}.
Then there is a positive integer $N=N(\unlam, \unn_I)$ such that $U_{\unlam, \unn}$ satisfies (\ref{eq open set}) if $n_j > N$ for all $j \in J$.

\end{lemma}

\begin{proof}
We give the proof in the case $\unlam = (\lam_1, \lam_2, \lam_3)$, $I=\{ 1 \}$, $J = \{ 2 \}$. 
(The proof in the general case is similar.)

For $[A] \in U_{\unlam, \unn}$, we can write
\[
[A] = Gl_{\unlam} ([A_1], T_1, [A_2], T_2, [A_3]).
\]
Here  $[A_i] \in U_{\lam_i}$, $T_{\unlam} + n_i < T_i < T_{\unlam} + n_i + \frac{3}{2}$. As $T_2 \rightarrow \infty$, $[A]$ converges to $([A_{12}], [A_3])$ for some $[A_{12}]$. Hence there is a positive real number $T_2^0 = T_2^{0} ([A_1],[A_2], [A_3], T_1)$ such that if $T_2 > T_2^0$ then $[A]$ lies in the image of a gluing map. Therefore $[A]$ is included in $U_{\unlam', \unn'}$ for some $( \unlam', \unn') \in \Lam$ with $| \unlam' | = 2 \ (< | \underline{\lambda} |)$. Since $U_{\lam_1}, U_{\lam_2}, U_{\lam_3},  (T_{\unlam} + n_1, T_{\unlam} + n_1 + \frac{3}{2})$ are precompact, we can take $T_2^0$ uniformly with respect to $([A_1], [A_2], [A_3], T_1) $. Therefore we obtain the statement.
\end{proof}

\vspace{5mm}

Define sections $s_l(\rho, \sigma)$ by (\ref{eq section}).
Then we have:

\begin{lemma} \label{lem tensor}
The sections $s_l(\rho, \sigma)$ satisfy (\ref{tensor}) in Proposition \ref{prop section}.
\end{lemma}

The proof will be given in Subsection \ref{ss proof of lem}.
Here we prove that we can perturb the sections $\sl$ such that the sections become admissible, assuming Lemma \ref{lem tensor}. We show the next lemma to do this.

\begin{lemma} \label{lem compactness}
For each $L \subset \{ 1,\dots, d \}$ with $\dim \0M \leq 2|L|$, the intersection $\ML$ is compact.
\end{lemma}

\begin{proof}
If $\ML$ is not compact, there is a sequence $\{ [A^{\alpha}] \}_{\alpha}$ in $\ML$ converging to some $([A_1^{\infty}], \dots, [A_{r}^{\infty}])$ with $r>1$. 
From Lemma \ref{lem tensor}, we have
\[
s_l([A^{\alpha}]) \longrightarrow 
s_l ([A_1^{\infty}]) \boxtimes \cdots \boxtimes s_l ([A_{r}^{\infty}])
\]
as $\alpha \rightarrow \infty$.
Since $s_l([A^{\alpha}])=0$ for $l \in L$, we have
\[
s_l([A_{1}^{\infty}]) \boxtimes \cdots \boxtimes s_{l}([A_r^{\infty}]) = 0.
\]
This means that for each $l \in L$, there is some $i(l)$ such that $s_l([A_{ i(l) }^{\infty}]) = 0$. Put
\[
L(i) := \{ \ l \in L \ | \ s_l([A_i^{\infty}]) = 0 \ \},
\]
then we obtain
\[
L = \bigcup_{i=1}^r L(i).
\]
Hence we have
\begin{equation} \label{eq L}
|L| \leq \sum_{i=1}^r | L(i) |.
\end{equation}
Using this inequality, we show
\begin{equation} \label{eq dim M, L}
\dim M_{Y \times \R}^0 (\rho(i_0-1), \rho(i_0)) < 2 |L(i_0)|
\end{equation}
for some $i_0$.
If not, we have
\[
\dim \iM \geq 2 |L(i)| + 1 \quad 
\]
for all $i$.
From (\ref{eq L}),
\[
\dim \M = 
\sum_{i=1}^r \dim \iM  \geq 
\sum_{i=1}^r (2 |L(i)| + 1) \geq 
2 |L| + r > 2|L| + 1.
\]
This is a contradiction since we assumed that $\dim \0M \leq 2|L|$. We have obtained (\ref{eq dim M, L}). 

By the hypothesis of induction, $s_{l}(\rho(i_0-1), \rho(i_0))$ satisfy (\ref{transversality I}) in Proposition \ref{prop section}. Hence the inequality (\ref{eq dim M, L}) means that
\[
M_{Y \times \R}^0 (\rho(i_0 - 1), \rho (i_0); L(i_0)) = \emptyset.
\]
On the other hand, $[ A_{i_0}^{\infty} ]$ lies in $M_{Y \times \R}^0 (\rho (i_0-1), \rho(i_0); L(i_0))$ by the definition of $L(i_0)$. We have obtained a contradiction. Therefore $\ML$ is compact.

\end{proof}

\noindent
{\it Proof of Proposition \ref{prop section}.}

When $\dim M_{Y \times \R}^0(\rho, \sigma) = 0$, generic sections $s_{l}(\rho, \sigma)$ are admissible.
Let $m$ be an integer with $1 \leq m \leq 7$ and suppose that we have admissible sections $s_l(\rho, \sigma)$ when $\dim M_{Y \times \R}^0(\rho, \sigma) \leq m-1$.
Let $\rho, \sigma$ be flat connections with $\dim M_{Y \times \R}^0(\rho, \sigma) = m$.
It follows from Lemma \ref{lem tensor} that the sections $\sl$ defined by (\ref{eq section}) satisfy (\ref{tensor}) in Proposition \ref{prop section}.
From Lemma \ref{lem compactness}, the intersections $\ML$ are compact if $\dim \0M \leq 2|L|$. Hence perturbing $\sl$ on a compact set of $\0M$, the transversality conditions (\ref{transversality I}), (\ref{transversality II}) in Proposition \ref{prop section} are satisfied. Since the region where $\sl$ are perturbed is compact, the perturbed sections also satisfy (\ref{tensor}). Therefore the perturbed sections are admissible.

\subsection{Proof of Lemma \ref{lem tensor}}
\label{ss proof of lem}

It remains to show Lemma \ref{lem tensor}. For simplicity of notations, we give a proof of the case when $[A^{\alpha}]$ converges to $([A_1^{\infty}], [A_2^{\infty}])$. 
To do this, we show three lemmas.

\vspace{2mm}

Let $[A^{\alpha}]$ converge to $([A_1^{\infty}], [A_2^{\infty}]) \in M_{Y \times \R}^0(\rho, \rho(1)) \times M_{Y \times \R}^0(\rho(1), \sigma)$. In the first lemma,  we consider the situation where the connections $[A^{\alpha}]$ can be written as gluing of three instantons.
For example, such a situation occurs if $[A^{\alpha}]$ are given by 
\[
[A^{\alpha}] = Gl_{\unlam}([A_1], T_1, [A_2], T_2^{\alpha}, [A_3]).
\]
Here $\lambda$, $[A_1], [A_2], [A_3]$ and $T_1$ are independent of $\alpha$, and $T_2^{\alpha} \rightarrow \infty$.

Assume that $[A^{\alpha}] \in Gl_{\unlam^{\alpha}}$ with $| \unlam^{\alpha} | = 3$. We can write
\[
[A^{\alpha}] = Gl_{\unlam^{\alpha}} ([B_1^{\alpha}], S_1^{\alpha}, [B_2^{\alpha}], S_2^{\alpha}, [B_3^{\alpha}]).
\]
for some $[B_i^{\alpha}] \in U_{\lam_i^{\alpha}}$  ($i=1, 2, 3$) and $S_j^{\alpha} > T_{\unlam^{\alpha} }$ ($j=1, 2$).
The first lemma is the following.

\begin{lemma} \label{lem 3 instantons}
There is a subsequence $\{ [A^{\alpha'}] \}_{\alpha'}$ which satisfies the following.

\begin{enumerate}[(1)]

\item
$\unlam^{\alpha'}$ are independent of $\alpha'$. We denote $\unlam^{\alpha'}$ by $\unlam = ( \lam_1, \lam_2, \lam_3 )$.

\vspace{3mm}

\item
For all $i$, 
$[B_i^{\alpha'}] \in U_{ \lam_i}$ converges to some $[B_i^{ \infty}] \in \overline{U}_{ \lam_i }$.

\vspace{3mm}

\item
We have either $S_1^{\alpha'} \rightarrow S_1^{\infty} < \infty$, $S_2^{\alpha'} \rightarrow \infty$ or $S_1^{\alpha'} \rightarrow \infty$, $S_{2}^{\alpha'} \rightarrow S_{2}^{\infty} < \infty$.

\end{enumerate}

\end{lemma}

We can show similar statements to this lemma when we assume $| \unlam^{\alpha} | = \ell$ with $\ell \geq 4$.
As we will see later, these statements imply that when $[A^{\alpha}]$ splits into two instantons there are no terms with $| \unlam | \geq 3$ in (\ref{eq section}) for large $\alpha$ .

\vspace{2mm}

In the second lemma, we consider the difference of two gluing maps.
Let $\{ [A^{\alpha}] \}_{\alpha}$ be a sequence in $\0M$ converging to some $([A_1^{\infty}], [A_2^{\infty}]) \in M_{Y \times \R}^0(\rho, \rho(1)) \times M_{Y \times \R}^0(\rho(1), \sigma)$.
Let $U_i$, $U_{i}'$ be precompact open neighborhoods of $[A_i^{\infty}]$ in $\i0M$ for $i=1, 2$. 
(Here $\rho(0) = \rho, \rho(2) = \sigma$.)
Then we have two gluing maps
\[
\begin{split}
Gl  &: U_1 \times (T_0, \infty) \times U_2 \longrightarrow \0M, \\
Gl' &: U_1' \times (T_0, \infty) \times U_2' \longrightarrow \0M.
\end{split}
\]
When $\alpha$ is sufficiently large, we can write
\[
[A^{\alpha}] 
= Gl ([A_1^{\alpha}], T^{\alpha}, [A_2^{\alpha}])
= Gl' ([A_1^{' \alpha}], T^{' \alpha}, [A_2^{' \alpha}])
\]
for some $[A_i^{\alpha}] \in U_i$, $[A_i^{' \alpha}] \in U_i'$, $T^{\alpha}, T^{' \alpha} > T_0$.

\begin{lemma} \label{lem 2 instantons I}
Under the above notations, we have
\[
\lim_{ \alpha \rightarrow \infty }
\left\| 
Gl_{T^{\alpha}} \big( s_l([A_1^{\alpha}] ) \boxtimes s_l ([A_2^{\alpha}]) \big) - 
Gl'_{ T^{' \alpha} } \big( s_l ( [A_1^{' \alpha} ] ) \boxtimes s_l ([ A_2^{' \alpha} ]) \big)
\right\|
= 0.
\]
\end{lemma}

\vspace{3mm}

As before, let $\{ [A^{\alpha}] \}_{\alpha}$ a sequence of instantons which converges to a pair $([A_0^{\infty}], [A_1^{\infty}]) \in M_{Y \times \R}^0(\rho, \rho(1)) \times M_{Y \times \R}^0(\rho(1), \sigma)$.
For $\unlam$ with $[A^{\alpha}] \in \Im Gl_{\unlam}$ and with $| \unlam | = 2$, we have $[A_i^{\alpha}(\unlam)] \in U_{\lam_i}$, $T^{\alpha}(\unlam) > T_{\unlam}$ such that
\[
[A^{\alpha}] = Gl_{\unlam} ([A_1^{\alpha}(\unlam)], T^{\alpha}(\unlam), [A_2^{\alpha}(\unlam)]).
\]
In the third lemma, we consider the behavior of $[A_i^{\alpha}(\unlam)]$.

\begin{lemma} \label{lem 2 instantons II}
$\phantom{a}$

\begin{enumerate}[(1)]

\item
There is an integer $\alpha_0 > 0$ such that if $\alpha > \alpha_0$, then $[A_i^{\alpha}(\unlam)]$ have the same flat limits as $[A_i^{\infty}]$. That is, $[A_i^{\alpha}(\unlam)] \in \i0M$ for $\alpha > \alpha_0$.
Here $\rho(0) = \rho$, $\rho(2) = \sigma$.

\vspace{2mm}

\item
For any positive number $\delta > 0$, there is some $\alpha_{\delta} > \alpha_0$ independent of $\unlam$ such that if $\alpha > \alpha_{\delta}$ we have
\[
d([A_i^{\alpha}(\unlam)], [A_i^{\infty}]) < \delta 
\]
for all $\unlam$ with $| \unlam |=2$, $[A^{\alpha}] \in \Im Gl_{\unlam}$ and $i=1, 2$. Here $d( \cdot, \cdot )$ is the metric on $\i0M$ induced by the $L^{2,\tau}_4$-norm.

\end{enumerate}

\end{lemma}

Before we prove these lemmas, we show Lemma \ref{lem tensor} using the lemmas.

\vspace{3mm}

\noindent
{\it Proof of Lemma \ref{lem tensor} assuming Lemma \ref{lem 3 instantons}, \ref{lem 2 instantons I} and \ref{lem 2 instantons II} }

Let $\{ A^{\alpha}] \}_{\alpha}$ be a sequence in $M_{Y \times \R}^0(\rho, \sigma)$ with limit $([A_1^{\infty}], [A_2^{\infty}])$.
Let $U_i$ be precompact open neighborhoods of $[A_i^{\infty}]$ for $i=1, 2$. Then we have a gluing map $Gl$ from $U_1 \times (T_0, \infty) \times U_2$ to $\0M$. 
For large $\alpha$, we have 
\[
[A^{\alpha}] = Gl([A_1^{\alpha}], T^{\alpha}, [A_2^{\alpha}])
\]
for some $[A_i^{\alpha}] \in U_i$, $T^{\alpha} > T_0$.
We need to show
\[
\lim_{ \alpha \rightarrow \infty}
\left\| 
s_l([A^{\alpha}]) - \widehat{Gl}_{T^{\alpha} } \big( s_l([A_1^{\alpha}]) \boxtimes s_l ([A_2^{\alpha}]) \big) 
\right\| 
= 0.
\]
First we prove the following claim.

\begin{claim}
Let $\{ f_{\lambda} \}_{\lambda}$ be the partition of unity as in Lemma \ref{lem partition}. Then there is a positive integer $\alpha_1 > 0$ such that if $\alpha > \alpha_1$, then $f_{\unlam, \unn} ([A^{\alpha}]) = 0$ for all $(\unlam, \unn) \in \Lam$ with $| \unlam | \geq 3$. 
\end{claim}

\begin{proof}
Otherwise we have a subsequence $\{ [A^{\alpha'}] \}_{\alpha'}$ such that $f_{ \unlam^{\alpha'}, \unn^{\alpha'} } ([A^{\alpha'}]) \not= 0$ for some $(\unlam^{\alpha'}, \unn^{\alpha'}) \in \Lam$ with $| \unlam^{\alpha'} | \geq 3$. From the additivity of index, we have
\[
| \unlam^{\alpha'} | \leq \dim \M.
\]
Hence we may assume that $| \unlam^{\alpha'} |$ are independent of $\alpha'$. For simplicity, we suppose $| \unlam^{\alpha'} | = 3$.

By Lemma \ref{lem 3 instantons}, we can suppose that $\unlam^{\alpha'}$ are independent of $\alpha'$ and we denote it by $\unlam = (\lam_1, \lam_2, \lam_3)$. 
Since the supports of  $f_{\unlam, \unn^{\alpha'}}$ are included in $U_{\unlam, \unn^{\alpha'}}$ and $f_{\unlam, \unn^{\alpha'}}([ A^{\alpha'}]) \not=0$, $[A^{\alpha'}]$ lie in $U_{\unlam, \unn^{\alpha'}}$.  Hence there are $[ B_i^{\alpha'}] \in U_{\lam_i}$ for $i=1, 2, 3$, and $S_j^{\alpha'} > T_{\unlam}$ for $j=1, 2$ such that
\[
[A^{\alpha'}] = Gl_{\unlam}([B_1^{ \alpha'}], S_1^{ \alpha'}, [B_2^{\alpha'}], S_2^{ \alpha'},[B_3^{\alpha'}]).
\]
We can assume that $[B_i^{\alpha'}]$ converge to $[B_i^{ \infty}]$ for $i=1, 2, 3$ and that $T_1^{\alpha'} \rightarrow T_1^{\infty} < \infty$, $T_2^{\alpha'} \rightarrow  \infty$ by Lemma \ref{lem 3 instantons}. The $\unn^{\alpha'}$ are  pairs $(n_1^{\alpha'}, n_2^{\alpha'})$ with
\[
T_{\unlam} + n_j^{\alpha'} < S_j^{\alpha'} < T_{\unlam} + n_j^{\alpha'} + \frac{3}{2}.
\]
Hence $n_1^{\alpha'} \rightarrow n_1^{\infty} < \infty$, $n_{2}^{\alpha'} \rightarrow \infty$. Put $I = \{ 1 \}$, $J= \{ 2 \}$, then we have
\[
n_2^{\alpha'} > N (\unlam, \unn_I^{\alpha'})
\]
for large $\alpha'$. 
Here $N (\unlam, \unn_I^{\alpha'})$ is the integer which appeared in Lemma \ref{lem partition}. Hence $f_{ \unlam, \unn^{\alpha'} }$ are identically zero for large $\alpha'$. This is a contradiction.
\end{proof}

It follows from this claim and (\ref{eq section}) that
\begin{equation} \label{eq section alpha large}
s_l([A^{\alpha}]) = 
\sum_{ 
\begin{subarray}{c}
\unlam = ( \lam_1, \lam_2) \\
[A^{\alpha}] \in \Im Gl_{\unlam}
\end{subarray}
}
f_{\unlam, n^{\alpha}(\unlam)} \ 
\widehat{Gl}_{\unlam, T^{\alpha}(\unlam)}
\big(
s_l([A_1^{\alpha}(\unlam)]) \boxtimes s_l ([A_2^{\alpha}(\unlam)])
\big)
\end{equation}
for $\alpha > \alpha_1$.
For $i=1, 2$ and $\delta > 0$, we put
\[
U_{i, \delta} := 
\{ \ [A] \in \i0M \ | \ d([A], [A_i^{\infty}]) < \delta \ \}.
\]
The open covers $\{ U_{\lam} \}_{\lam \in \Lambda(\rho(i-1), \rho(i))}$ are locally finite, hence if $\delta$ is sufficiently small then the numbers of open sets $U_{\lam}$ with $U_{\lam} \cap U_{i, \delta}$ non-empty are finite for all $i$. We fix such $\delta>0$. 
By Lemma \ref{lem 2 instantons II}, there is a positive integer $\alpha_{\delta}$ such that if $\alpha > \alpha_{\delta}$, we have $[A_i^{\alpha}(\unlam)] \in U_{i,\delta}$ for all $\unlam$.
Let $\{ U_{\lam_i(p)} \}_{p=1}^{q_i}$ be the sets of open sets which intersect $U_{i,\delta}$ for each $i$, and put $\alpha_{\delta}' := \max \{ \alpha_1, \alpha_{\delta} \}$. Then $\unlam$ in (\ref{eq section alpha large}) take the form
\[
\unlam(p_1, p_2) = (\lam_1(p_1), \lam_2(p_2))
\quad
(1 \leq p_1 \leq q_1, 1 \leq p_2 \leq q_2)
\]
for $\alpha > \alpha_{\delta}'$.

By Lemma \ref{lem 2 instantons I}, for any $\varepsilon > 0$, there exists $\alpha(p_1, p_2; \varepsilon)$ such that if $\alpha > \alpha(p_1, p_2; \varepsilon)$, we have
\[
\left\|
\widehat{Gl}_{T^{\alpha} }
\big(
s_l([A_1^{\alpha}]) \boxtimes s_l([A_2^{\alpha}])
\big)
-
\widehat{Gl}_{\unlam(p_1, p_2), T^{\alpha}(p_1, p_2)}
\big(
s_l( [A_1^{\alpha}(p_1, p_2) ] ) \boxtimes
s_l( [A_2^{\alpha}(p_1, p_2)])
\big)
\right\|
<
\varepsilon.
\]
Here $[A_i^{\alpha}(p_1, p_2)]=[A_i^{\alpha}(\unlam(p_1, p_2))]$, $T^{\alpha}(p_1, p_2) = T^{\alpha}(\unlam(p_1, p_2))$. We put
\[
\alpha(\varepsilon) :=
\max \{ \ 
\alpha_{\delta}', \alpha(p_1, p_2; \varepsilon) \ | \ 
1 \leq p_ 1 \leq q_1, 1 \leq p_2 \leq q_2 \ 
\}.
\]
Then for $\alpha > \alpha(\varepsilon)$
\[
\left\| 
s_l ([A^{\alpha}]) - 
Gl_{ T^{\alpha} } \big(
s_l ([A_1^{\alpha}]) \boxtimes s_l ([A_2^{\alpha}])
\big)
\right\|
<
\sum_{
\begin{subarray}{c}
\unlam = (\lam_1, \lam_2) \\
[A^{\alpha}] \in \Im Gl_{ \unlam }
\end{subarray}
}
f_{\lam, n^{\alpha}(\unlam)}([A^{\alpha}]) \ \varepsilon
= 
\varepsilon.
\]
Therefore $s_l([ A^{\alpha}])$ converges to $s_l([A_1^{\infty}]) \boxtimes s_l([A_2^{\infty}])$.

\qed

\vspace{3mm}

Here we prove Lemma \ref{lem 3 instantons} and \ref{lem 2 instantons II}. The proof of Lemma \ref{lem 2 instantons I} will be given in the next subsection.

\vspace{3mm}

\noindent
{\it Proof of Lemma \ref{lem 3 instantons} }

Let $[A^{\alpha}]$, $[B_i^{\alpha}]$ and $S_j^{\alpha}$ be as in Lemma \ref{lem 3 instantons}.
There is a subsequence $\{ [A^{\alpha'}] \}_{\alpha'}$ such that $[B_i^{\alpha'}]$ converge to $([B_{i,1}^{\infty}], \dots, [B_{i,r_i}^{\infty}])$ for all $i$. We will show that $r_i=1$ for all $i$. 
This implies (1) and  (2) since the oven covers $\{ U_{\lam} \}_{\lam}$ are locally finite.

If not, $r_{i_0} > 1$ for some $i_0$. For simplicity, we assume $r_1 = 2$, $r_2 = r_3 = 1$.
Since $[B_2^{\alpha'}]$, $[B_3^{\alpha'}]$ converge and $\{ U_{\lambda} \}_{\lambda}$ are locally finite, we may suppose that $\lam_2^{\alpha'}$, $\lam_3^{\alpha'}$ are constant. We write $\lam_2, \lam_3$ for $\lam_2^{\alpha'}$, $\lam_3^{\alpha'}$. On the other hand, the sequence $\{ [B_1^{\alpha'}] \}_{\alpha'}$ has the limit $([B_1^{\infty}], [B_2^{\infty}])$. We may suppose that $\lam_1^{\alpha'}$ take the form $\lam_1^{\alpha'}=(\lam_1', \lam_2'; n^{\alpha'})$ with $n^{\alpha'} \rightarrow \infty$. 
Put $\unlam' := (\lam_1', \lam_2')$.
We assumed that $T_{\unlam}$ satisfy the inequality (\ref{eq condition on T}) for all $\unlam$. Hence we have
\[
 T_{\unlam'} + n^{\alpha'}  < T_{\unlam^{\alpha}} < S_j^{\alpha}
\]
for $j=1, 2$.
In particular, $S_j^{\alpha'} \rightarrow \infty$ for $j=1, 2$. This means that $[A^{\alpha}]$ converge to some
\[
([B_{1,1}^{\infty}], [B_{1, 2}^{\infty}], [B_{3}^{\infty}], [B_4^{\infty}]).
\]
This is a contradiction since $[A^{\alpha}] \rightarrow ([A_1^{\infty}], [A_2^{\infty}])$. Therefore we obtain $r_1 = r_2 = r_3 = 1$.

Lastly we prove (3). If (3) does not hold, then we may suppose that $S_j^{\alpha'} \rightarrow S_j^{\infty} < \infty$ for $j= 1, 2$ or $S_j^{\alpha'} \rightarrow \infty$ for $j=1, 2$. In the first case, $[A^{\alpha'}]$ does not split into  instantons as $\alpha' \rightarrow \infty$. In the second case, $[A^{\alpha'}]$ splits into three instantons as $\alpha' \rightarrow \infty$. Both cases contradict the fact that $[A^{\alpha}]$ split into two instantons.

\qed

\vspace{3mm}

\noindent
{\it  Proof of Lemma \ref{lem 2 instantons II} }

Let $[A^{\alpha}]$, $[A^{\alpha}_i]$ and $T^{\alpha}(\unlam)$ be as above.
Suppose that (1) does not hold. Then we have a subsequence $\{ [A^{\alpha'}] \}_{\alpha'}$ and $\unlam^{\alpha'} \in \Lambda(\rho, \rho^{\alpha'}(1)) \times \Lambda(\rho^{\alpha'}(1), \sigma)$ with $[A^{\alpha}] \in \Im Gl_{\unlam^{\alpha'}}$ and $\rho^{\alpha'}(1) \not= \rho(1)$. By Hypothesis \ref{ass flat conn}, the set of gauge equivalence classes of flat connections is finite. Hence we can suppose $\rho^{\alpha'}(1)$ are independent of $\alpha'$. We write $\rho'(1)$ for $\rho^{\alpha'}(1)$.
As in the proof of Lemma \ref{lem 3 instantons}, we can see that there is a subsequence (still denoted by $[A^{\alpha'}]$) such that $[A_i^{\alpha'}(\unlam^{\alpha'})]$ converge to some $[A_i^{' \infty}]$,  $\unlam^{\alpha'}$ are independent of $\alpha'$ and that $T^{\alpha'}(\unlam^{\alpha'})$ diverges to $\infty$ . This means that
\[
[A^{\alpha}] 
\longrightarrow 
([A_1^{' \infty}], [A_2^{' \infty}])
\in M_{Y \times \R}^0 (\rho, \rho'(1)) \times M_{Y \times \R}^0(\rho'(1), \sigma)
\quad
\rho'(1) \not= \rho(1).
\]
This is a contradiction to the fact that $[A^{\alpha}]$ converge to $([A_1^{\infty}], [A_2^{\infty}]) \in M_{Y \times \R}^0(\rho, \rho(1)) \times M_{Y \times \R}^0(\rho(1), \sigma)$ and we have shown that (1) holds.

If (2) does not hold, there is a subsequence $\{ [A^{\alpha'}] \}_{\alpha'}$ and $\unlam^{\alpha'} \in \Lambda(\rho, \rho(1)) \times \Lambda(\rho(1), \sigma)$ with  
\begin{equation} \label{eq dist}
d([A_i^{\alpha'}( \unlam^{\alpha'} )], [A_i^{\infty}]) \geq \delta.
\end{equation}
As before, we can deduce that $[A_i^{\alpha'}(\unlam^{\alpha'})]$ converge to some $[A_i^{' \infty}]$ for $i=1, 2$, $\unlam^{\alpha'}$ are independent of $\alpha'$, and that $T^{\alpha'}(\unlam^{\alpha'})$ diverges to $\infty$.
Hence $[A^{\alpha'}]$ converge to $([A_1^{' \infty}], [A_2^{' \infty}])$.
 From (\ref{eq dist}), we have $d([A_i^{\infty}], [A_i^{' \infty}]) \geq \delta$. In particular, $[A_i^{\infty}] \not= [A_{i}^{' \infty}]$. This is a contradiction.

\subsection{Evaluation of difference of two gluing maps}
\label{ss eva}

In this subsection, we prove Lemma \ref{lem 2 instantons I}. We need to compare  two gluing maps defined on different regions intersecting each other. To do this, we must explicitly see the construction of gluing maps.  We give outline of the construction following \cite{D Floer} before we prove Lemma \ref{lem 2 instantons I}.

First note that elements in $\0M$ can be considered as elements in $\M$ with center of mass $0$. Here the center of mass of $[A] \in \M$ is defined to be
\[
\int_{Y \times \R} t | F_A |^2 \ d \mu_{Y \times \R}.
\]
Let $X_1, X_2$ be two copies of $Y \times \R$ and for $T > 0$ put $X_1(T) = Y \times ( -\infty, T)$, $X_2(Y) = Y \times ( -T, \infty)$. For each $T$, we have an identification
\[
\begin{array}{cccc}
\varphi : & Y \times (T, 2T) & \longrightarrow & Y \times (-2T, -T) \\
                     & (y, t) & \longmapsto & (y, t-3T).
\end{array}
\]
Gluing $X_1(2T)$ and $X_2(2T)$ through $\varphi$, we obtain a manifold $X^{\#(T)} = X_1(2T) \cup_{\varphi} X_2(2T)$. We have a natural identification between $X^{\# (T)}$ and $Y \times \R$ such that $(y, 3T/2) \in X_1(2T) \subset X^{\#(T)}$ corresponds to $(y, 0) \in Y \times \R$.

Let $\rho=\rho(0)$, $\rho(1)$, $\sigma = \rho(2)$ be flat connections and choose $[A_1] \in M_{X_1}^0(\rho, \rho(1))$, $[A_2] \in M_{X_2}^0(\rho(1), \sigma)$. Here $A_1, A_2$ are instantons on $X_1, X_2$ with center of mass $0$. We write $\rho(0)$, $\rho(2)$ for $\rho$, $\sigma$ respectively.
Then we can write 
\[
A_i = B_i + a_i,
\]
where $B_i$ are connections with
\[
B_i = \left\{
\begin{array}{cl}
\rho(i-1) & \text{ on $Y \times (-\infty, -1)$ }, \\
\rho(i) & \text{ on $Y \times (1, \infty)$ },
\end{array}
\right.
\]
and $a_i$ are $\pi_i^* \fg_{Q}$-valued 1-forms on $X_i$ with
\[
| \nabla^{k} a_i(y,t) | \leq C_k e^{-\delta t} 
\quad ( ^{\forall} k \in \Z_{\geq 0})
\]
for some $C_k > 0$, $\delta > 0$.
Fix a smooth cut-off function $\chi:Y \times \R \rightarrow [0,1]$ such that
\[
\chi(y,t) = \left\{
\begin{array}{cl}
1 & \text{ on $Y \times (-\infty, 0)$, } \\
0 & \text{ on $Y \times (1, \infty)$. }
\end{array}
\right. 
\]
We define $A_1'$, $A_2'$ by
\[
A_1' = B_1 + \chi(t-T) a_1, 
\quad
A_2' = B_2 + \chi(-t-T) a_2.
\]
These connections give a connection $A'$ on $X^{\#(T)} \cong Y \times \R$ such that
\begin{equation} \label{eq A'}
A' = 
\left\{
\begin{array}{cl}
c_{-3T/2}^*(A_1') & \text{ on $ Y \times (-\infty, -\frac{T}{2}+1] $ } \\
\rho(1) & \text{ on $ Y \times (-\frac{T}{2} + 1, \frac{T}{2} -1) $ } \\
c_{3T/2}^*(A_2') & \text{ on $ Y \times [\frac{T}{2}-1, \infty) $ }.
\end{array}
\right.
\end{equation}
Let $F_{A'}^+$ be the self-dual part of the curvature of the connection on $\pi^* \fg_{Q}$ induced by $A'$. Then we have
\[
\| F_{A'}^+ \|_{ L_4^2(Y \times \R) }
\leq
\const e^{-\delta T}.
\]
That is, $A'$ is an almost instanton. To obtain a genuine instanton, we consider the equation 
\[
F_{A' + a}^+ = 0
\]
for a small $\pi^* \fg_{Q}$-valued 1-form $a$. The equation can be rewritten as
\begin{equation} \label{eq perturb}
d_{A'}^+ a + ( a \wedge a )^+ = -F_{A'}^+.
\end{equation}
By Hypothesis \ref{hypo D}, $D_{A_i} = d_{A_i}^* + d_{A_i}^+$ have right inverses $P_i$. When $T$ is large enough, we can construct a right inverse $P = P_T$ of $D_{A'}$ by using $P_1, P_2$ and some cut-off functions. 
We can see the operator norm $\| P \|$ of $P$ is bounded by $2( \| P_1 \| + \| P_2 \|)$.
If we substitute  $a = P \phi$  into (\ref{eq perturb}), we get the equation
\[
\phi + ( P \phi \wedge P \phi)^ + = - F_{A'}^+
\]
for $\phi \in L_3^2( \Lambda^0_{Y \times \R} \otimes \pi^* \fg_Q)$.
We can see that for large $T > 0$ this equation has a unique solution $\phi = \phi_T$ with $\phi_{T} \rightarrow 0$ as $T \rightarrow \infty$. Hence for large $T$ we get an instanton of the form
\begin{equation} \label{eq A}
A^{\#(T)} = A' + a
\end{equation}
with
\begin{equation} \label{eq a}
\lim_{T \rightarrow \infty} a = 0.
\end{equation}

For $U_i \subset M_{X_i}^0(\rho(i-1), \rho(i))$ precompact open sets, the construction can be applied to all $([A_1], [A_2]) \in U_1 \times U_2$ and we obtain the gluing map
\[
Gl:U_1 \times (T_0, \infty) \times U_2 \longrightarrow \0M
\]
for large $T_0$.
(We can take $T_0$ uniformly since $U_1, U_2$ are precompact.)

More precisely, we must slightly translate $A^{\#(T)}$ to make the center of mass be $0$.
Let $m_{T}$ be the center of mass of $A^{\#(T)}$.
Then the translate $c_{-m_T}^*(A^{\#(T)})$ is an instanton with center of mass $0$. The precise definition of the gluing map is
\[
Gl([A_1], T, [A_2]) = [c_{-m_T}^*(A^{\#(T)})].
\]
We can easily show
\[
\lim_{T \rightarrow \infty} m_T = 0
\]
and it does not matter even if we assume that $m_T = 0$ in the proof of Lemma \ref{lem 2 instantons I}. Therefore we will drop $c_{-m_T}^*$ from notations in the proof.

Let $\gamma \cong S^1$ be a loop in $Y$ and put $\Gamma = \gamma \times \R$.
We see outline of the construction of the gluing map 
\[
\widehat{Gl}:
\cL^{\otimes 2}_{\Gamma} (\rho, \rho(1)) |_{U_1} \boxtimes 
\cL^{\otimes 2}_{\Gamma} (\rho(1), \sigma)|_{U_2}
\stackrel{\cong}{\longrightarrow}
\cL^{\otimes 2}_{\Gamma}(\rho, \sigma)|_{\Im Gl}
\]
covering $Gl$. Let $A_1, A_2$ be instantons on $X_1, X_2$. For simplicity of notations, we suppose that $\bar{\partial}_{A_i}$ are surjective for $i=1, 2$. 
Let $\Gamma_1$, $\Gamma_2$ be copies of $\Gamma$ and fix cut-off functions $\gamma_i = \gamma_i^{T}$ on $\Gamma_i$ with
\[
\| d \gamma_i \|_{C^0} = O \left( \frac{1}{T} \right), 
\quad
\supp (\gamma_i) \subset \Gamma_i(2T),
\quad
(\gamma_1)^2 + ( \gamma_2 )^2 = 1 \ \text{on $\Gamma^{\#(T)}$}.
\]
Here $\supp (\gamma_i)$ is the support of $\gamma_i$, and $\Gamma_i(2T), \Gamma^{\#(T)}$ are defined as before.
In the third equation, we consider $\gamma_{i}$ as functions on $\Gamma^{\#(T)}$ in the natural way.
Take $f_i \in \ker \bar{\partial}_{A_i}$ for $i=1, 2$ and put
\begin{equation} \label{eq f'}
f' = c_{-3T/2} (\gamma_1 f_1) + c_{3T/2}(\gamma_2 f_2).
\end{equation}
Then we can show that
\begin{gather*}
\| \bar{ \partial }_{ A^{ \# (T) } } f' \|_{ L^{ 2, (-\tau, \tau) }_2 }
\leq 
\varepsilon (T) 
\big ( \| f_1 \|_{L^{ 2, (-\tau, \tau) }_{3} } + \| f_2 \|_{ L_3^{2,(-\tau, \tau)} } \big), \\
\varepsilon (T) \longrightarrow 0 \ 
\text{as  $T \rightarrow \infty $}.
\end{gather*}

We may construct a right inverse $Q_{T}$ from right inverses $Q_i$ of $\bar{\partial}_{A_i}$ for large $T$. The difference
\begin{equation} \label{eq gl line 1}
f^{\#(T)} = f' - Q_{T} \bar{\partial}_{A^{\#(T)}} f'
\end{equation}
lies in $\ker \bar{\partial}_{ A^{ \#(T) } }$. The operator norm $\| Q_{ T } \|$ of $Q_{T}$ is bounded by $2( \| Q_1 \| + \| Q_2 \|)$, hence we have
\begin{equation} \label{eq gl line 2}
Q_{T} \bar{\partial}_{A^{\#(T)}} f' \longrightarrow 0
\end{equation}
as $T \rightarrow \infty$. It can be shown that for large $T$ the map $(f_1, f_2) \mapsto f^{\#(T)}$ is an isomorphism from $\ker \bar{\partial}_{A_1} \oplus \ker \bar{\partial}_{A_2}$ to $\ker \bar{\partial}_{ A^{ \#(T) } }$  and induces the isomorphism from $\cL_{\Gamma_l}(\rho, \rho(1))_{[A_1]} \otimes  \cL_{ \Gamma_l }(\rho(1), \sigma)_{ [A_2] }$ to $ \cL_{\Gamma_l}( \rho, \sigma)_{[ A^{\# (T)} ]} $. Applying this construction to all connections in $U_1 \times U_2$, we get the map $\widehat{Gl}$.

\vspace{3mm}

When $\bar{\partial}_{A_i}$ are surjective, we choose maps
\[
U_i : 
\R^{n_i} \longrightarrow 
\Omega_{ \Gamma_l }^{0,1}( E \otimes K_{\Gamma_l}^{ \frac{1}{2} }  )
\]
such that $\bar{\partial}_{A_i} \oplus U_i$ are surjective.
Applying the above method to these operators, we obtain the gluing maps.

\noindent
{\it Proof of Lemma \ref{lem 2 instantons I}}

Assume that we have a sequence $\{ [A^{\alpha}] \}_{\alpha}$ in $\0M$ which converges to some $([A_1^{\infty}], [A_2^{\infty}]) \in M_{Y \times \R}^0(\rho, \rho(1)) \times M_{Y \times \R}^0(\rho(1), \sigma)$. Take precompact, open neighborhoods $U_i, U_i'$ of $[A_i^{\infty}]$. Then we have two gluing maps
\[
\begin{split}
Gl &: U_1 \times (T_0, \infty) \times U_2 \longrightarrow \0M \\
Gl' &: U_1' \times (T_0, \infty) \times U_2' \longrightarrow \0M.
\end{split}
\]
For large $\alpha$, we have two different expressions of $[A^{\alpha}]$:
\[
[A^{\alpha}] 
= Gl([A_1^{\alpha}], T^{\alpha}, [A_2^{\alpha}])
= Gl'([A_1^{' \alpha}], T^{' \alpha}, [A_2^{' \alpha}]).
\]
We can take representations $A_i^{\alpha}$, $A_{i}^{' \alpha}$, $A_i^{\infty}$ of $[A_i^{\alpha}]$, $[A_i^{' \alpha}]$, $[A^{\infty}_i]$ such that $A_i^{\alpha}$, $A_i^{' \alpha}$ converge to $A_i^{\infty}$.
For simplicity, we assume that $\bar{\partial}_{A_i}$ are surjective for all $[A_i] \in U_i$.
Then what we must show is that
\[
\lim_{\alpha \rightarrow \infty} 
\|
Gl_{T^{\alpha}} (f_1^{\alpha}, f_2^{\alpha}) - Gl'_{T^{' \alpha}} (f_1^{' \alpha}, f_2^{' \alpha})
\|_{ L^{2, (-\tau, \tau)} }
=0
\]
for sequences $\{ f_i^{\alpha} \}_{\alpha}$, $\{ f_i^{ ' \alpha} \}_{\alpha}$ of $\ker \bar{\partial}_{A_i^{\alpha}}$, $\ker \bar{\partial}_{A_i^{' \alpha} }$ which converge to some $f_i^{\infty} \in \ker \bar{\partial}_{A_i^{\infty}}$ in the $L^{2, (-\tau, \tau)}_{3}$-norms.
From (\ref{eq f'}), (\ref{eq gl line 1}) and (\ref{eq gl line 2}), we have
\begin{equation} \label{eq gl}
\begin{split}
Gl_{T^{\alpha}} (f_1^{\alpha}, f_2^{\alpha})
&= c_{ -\frac{ 3T^{\alpha} }{ 2 } } (\gamma_1 f_1^{\alpha}) + 
c_{ \frac{ 3T^{\alpha} }{ 2 } } (\gamma_2 f_2^{\alpha}) +
g^{\alpha} \\
Gl'_{ T^{' \alpha} }(f_1^{' \alpha}, f_2^{' \alpha})
&= c_{ -\frac{ 3T^{' \alpha} }{ 2 } } (\gamma_1 f_1^{' \alpha}) + 
c_{ \frac{ 3T^{' \alpha} }{ 2 } } (\gamma_2 f_2^{' \alpha}) +
g^{' \alpha}
\end{split}
\end{equation}
for some sections $g^{\alpha}$, $g^{' \alpha}$ with $g^{\alpha}, g^{' \alpha} \rightarrow 0$. Here we show that the difference $T^{\alpha} - T^{' \alpha}$ goes to zero as $\alpha \rightarrow \infty$.

\begin{lemma} \label{lem T T'}
$\lim_{\alpha \rightarrow \infty} | T^{\alpha} - T^{' \alpha} | = 0$.
\end{lemma}

\begin{proof}
Since $[A^{\alpha}]$ converges to $([A_1^{\infty}], [A_2^{\infty}])$, we have real numbers $S_1^{\alpha}, S_2^{\alpha}$ with $S_2^{\alpha} - S_1^{\alpha} \rightarrow \infty$ such that for any compact sets $K$ in $Y \times \R$, 
\begin{equation} \label{eq A translated by S}
c_{S_i^{\alpha}} ([A^{\alpha}|_K]) \longrightarrow [A_i^{\infty}|_{K}]
\end{equation}
in $L^{2, \tau}_{4}$. 

On the other hand, $[A^{\alpha}]$ can be written as $Gl([A_1^{\alpha}], T^{\alpha}, [A_2^{\alpha}])$.
It follows from (\ref{eq A'}), (\ref{eq A}) and (\ref{eq a}) that
\begin{equation} \label{eq A translated by T}
c_{ \frac{3T^{\alpha}}{2} }([A^{\alpha}|_{K}])
\longrightarrow
[ A_1^{\infty}|_{K} ].
\end{equation}
Comparing (\ref{eq A translated by S}) with (\ref{eq A translated by T}), we have
\[
\lim_{\alpha \rightarrow \infty}
| S_1^{\alpha} - (3T^{\alpha} / 2)|
= 0.
\]
Similarly we have
\[
\lim_{\alpha \rightarrow \infty}
| S_1^{\alpha} - (3T^{' \alpha} / 2)|
= 0.
\]
Hence we obtain the required result.

\end{proof}

From (\ref{eq gl}), we get
\[
\begin{split}
& \| Gl_{T^{\alpha}}(f_1^{\alpha}, f_2^{\alpha}) -
Gl_{T^{' \alpha}}(f_1^{' \alpha}, f_2^{' \alpha})
\| \\
& \quad \leq
\| c_{ -\frac{ 3T^{\alpha} }{2} }( \gamma_1 f_1^{\alpha}) -
c_{ -\frac{ 3T^{\alpha} }{2} }( \gamma_1 f_1^{' \alpha})
\| +
\|
c_{ -\frac{ 3T^{\alpha} }{2} }( \gamma_1 f_1^{' \alpha}) -
c_{ -\frac{ 3T^{' \alpha} }{2} }( \gamma_1 f_1^{' \alpha})
\| + \\
& \qquad \qquad
\|
c_{ \frac{ 3T^{\alpha} }{2} }( \gamma_2 f_2^{\alpha}) -
c_{ \frac{ 3T^{\alpha} }{2} }( \gamma_2 f_2^{' \alpha})
\| +
\|
c_{ \frac{ 3T^{\alpha} }{2} }( \gamma_2 f_2^{'\alpha}) -
c_{ \frac{ 3T^{' \alpha} }{2} }( \gamma_2 f_2^{' \alpha})
\| + 
\| g^{\alpha} \| + \| g^{' \alpha} \|.
\end{split}
\]
Here $\| \cdot \|$ is the $L^{2, (-\tau, \tau)}$ norm.
The first term on the right hand side is equal to
\[
\| \gamma_1 (f_1^{\alpha} - f_1^{' \alpha}) \|.
\]
This is bounded by $\| f_1^{\alpha} - f_1^{ ' \alpha} \|$, and $\| f_1^{\alpha} - f_1^{ ' \alpha} \|$ converges to zero since both of $f_1^{\alpha}$, $f_1^{' \alpha}$ converge to $f_1^{\infty}$.
Hence the first term goes to zero. 

\vspace{2mm}

The second term is equal to
\[
\| \gamma_1 f_1^{' \alpha} - c_{ \delta^{ \alpha } } ( \gamma_1 f_1^{ ' \alpha } ) \|,
\]
where $\delta^{\alpha} = 3(T^{\alpha}-T^{' \alpha}) / 2$. It follows from Lemma \ref{lem T T'} that $\delta^{\alpha}$ goes to zero.
Since  $\| f_1^{\infty} \|$ is finite and $\{ \delta^{\alpha} \}_{\alpha}$ is bounded, for any $\varepsilon > 0$ there exist $T(\varepsilon) > 0$ independent of $\alpha$ such that
\begin{gather*}
\| f_1^{\infty} |_{Y \times (-\infty, \ -T(\varepsilon) )} \| 
< \frac{\varepsilon}{8},
\quad
\| f_1^{\infty} |_{Y \times (T(\varepsilon), \ \infty )} \| 
< \frac{\varepsilon}{8} \\
\| c_{ \delta^{ \alpha } } (f_1^{\infty}) |_{Y \times (-\infty, \ -T(\varepsilon) )} \| 
< \frac{\varepsilon}{8},
\quad
\| c_{ \delta^{ \alpha } } (f_1^{\infty}) |_{Y \times (T(\varepsilon), \ \infty )} \| 
< \frac{\varepsilon}{8}.
\end{gather*}
Since $f_{1}^{' \alpha}$ converge to $f_1^{\infty}$, we have $\alpha(\varepsilon) > 0$ such that if $\alpha > \alpha(\varepsilon)$ then
\[
\| f_1^{' \alpha} - f_1^{\infty} \| < \frac{ \varepsilon }{8}.
\]
Thus we have
\[
\begin{split}
& \left\| 
\left.
\big( \gamma_1 f_1^{' \alpha} - c_{ \delta^{ \alpha } } ( \gamma_1 f_1^{ ' \alpha } ) \big) 
\right|_{ Y \times (-\infty, \ -T(\varepsilon) ) } 
\right\| \\
&\quad  \leq
\| \gamma_1 f_1^{' \alpha} |_{ Y \times (-\infty, \ -T(\varepsilon) ) }  \| + 
\| c_{ \delta^{ \alpha } } ( \gamma_1 f_1^{ ' \alpha } )|_{ Y \times (-\infty, \ -T(\varepsilon) ) } \| \\
&\quad \leq 
\| f_1^{' \alpha} - f_1^{\infty} \| + 
\|  f_1^{\infty}|_{ Y \times (-\infty, \ -T(\varepsilon) ) }  \| +
\| c_{ \delta^{\alpha} } ( f_1^{' \alpha} - f_1^{\infty} ) \| + 
\| c_{ \delta^{ \alpha } }  (f_1^{\infty}) |_{ Y \times (-\infty, \ -T(\varepsilon) ) }  \| \\
&\quad <
\frac{ \varepsilon }{2}.
\end{split}
\]
Similarly
\[
\left\| 
\left.
\big( \gamma_1 f_1^{' \alpha} - c_{ \delta^{ \alpha } } ( \gamma_1 f_1^{ ' \alpha } ) \big) 
\right|_{ Y \times (T(\varepsilon), \ \infty ) } 
\right\|
< \frac{ \varepsilon }{ 2 }.
\]
Since $\gamma_1 f_1^{' \infty}$ is uniformly continuous on $Y \times [-T(\varepsilon), T(\varepsilon)]$ and $\delta^{\alpha}$ converge to zero, we have
\[
\begin{split}
&
\left\| \left.
\big( \gamma_1 f_1^{\infty} - c_{\delta^{\alpha}} (\gamma_1 f_1^{\infty}) \big)  \right|_{ Y \times [ -T(\varepsilon), T(\varepsilon) ] }
\right\|_{ L^{2, (-\tau, \tau)} } \\
& \qquad
\leq 
2 T(\varepsilon) M(\varepsilon)
\max_{ Y \times [-T(\varepsilon), T(\varepsilon)] } 
| \gamma_1 f_1^{\infty} -  c_{\delta^{\alpha}} (\gamma_1 f_1^{\infty}) |
\longrightarrow 0
\end{split}
\]
as $\alpha \rightarrow \infty$.
Here $M(\varepsilon)$ is the maximum of the weight function $W_{\tau}'$ over $Y \times [-T(\varepsilon), T(\varepsilon)]$.
Hence we get
\begin{equation*}
\begin{split}
&
\left\| \left.
\big(
\gamma_1 f_1^{' \alpha} - c_{\delta^{\alpha}} (\gamma_1 f_1^{' \alpha}) 
\big)
\right|_{ Y \times [-T(\varepsilon), T(\varepsilon)] }
\right\| \\
& \quad \leq
\left\| \left. \big( \gamma_1 f_1^{' \alpha} - \gamma_1 f_1^{\infty} \big) \right|_{ Y \times [ -T(\varepsilon), T(\varepsilon) ] } \right\| +
\left\| \left. \big( \gamma_1 f_1^{\infty} - c_{\delta^{\alpha}} (\gamma_1 f_1^{\infty}) \big) \right|_{ Y \times [-T(\varepsilon), T(\varepsilon) ] } \right\| \\
& \quad
\longrightarrow 0 
\end{split}
\end{equation*}
as $\alpha \rightarrow \infty$.
Therefore
\[
\begin{split}
& \limsup_{\alpha \rightarrow \infty} 
\|
\gamma_1 f_1^{' \alpha} - c_{ \delta^{ \alpha } } ( \gamma_1 f_1^{ ' \alpha } )
\| \\
&\quad \leq
\limsup_{\alpha \rightarrow \infty}
\left(
\left\| 
\left.
\big( \gamma_1 f_1^{' \alpha} - c_{ \delta^{ \alpha } } ( \gamma_1 f_1^{ ' \alpha } ) \big) 
\right|_{ Y \times (-\infty, \ -T(\varepsilon) ) } 
\right\|
+ \right. \\
& \quad \quad
\left.
\left\| 
\left.
\big(
\gamma_1 f_1^{' \alpha} - c_{ \delta^{ \alpha } } ( \gamma_1 f_1^{ ' \alpha } )
\big) 
\right|_{ Y \times ( -T(\varepsilon), \ T(\varepsilon )) }
\right\|
+
\left\| 
\left.
\big( \gamma_1 f_1^{' \alpha} - c_{ \delta^{ \alpha } } ( \gamma_1 f_1^{ ' \alpha } ) \big) 
\right|_{ Y \times (T(\varepsilon), \ \infty ) } 
\right\|
\right) \\
&\quad< 
\varepsilon.
\end{split}
\]
Since $\varepsilon$ is arbitrary, the norm of difference $\|
\gamma_1 f_1^{' \alpha} - c_{ \delta^{ \alpha } } ( \gamma_1 f_1^{ ' \alpha } )
\|$ goes to zero. Thus the limit of the second term is zero.
Similarly the limits of the third and fourth terms are zero.
Since $g^{\alpha}, g^{' \alpha} \rightarrow 0$, we have
\[
\lim_{\alpha \rightarrow \infty}
\| Gl_{T^{\alpha}}(f_1^{\alpha}, f_2^{\alpha}) -
Gl_{T^{' \alpha}}(f_1^{' \alpha}, f_2^{' \alpha})
\|
= 0
\]
as required.

\subsection{Well-definedness}
\label{ss FFH well-defined}

A priori, Fukaya-Floer homology groups seem to depends on the choices of Riemannian metric on $Y$ and sections of $\cL_l^{\otimes 2}(\rho, \sigma)$. But we will prove that Fukaya-Floer homology does not depend on these choices up to canonical isomorphism. In the proof of well-definedness of usual Floer homology groups, we need the functorial property of Floer homology groups with respect to Riemannian bordisms.  We shall generalize the functorial property to Fukaya-Floer homology groups and prove the well-definedness by using this property.

The main statement of this subsection is

\begin{proposition} \label{prop well-defined}
Take two Riemannian metrics $g_0, g_1$ on $Y$ and two sets $\{ s_l(\rho, \sigma) \}_{\rho, \sigma, l}$, $\{ s_l'(\rho, \sigma) \}_{\rho, \sigma, l}$ of admissible sections. We write $HFF_*(Y; \ungam)$, $HFF_*' (Y; \ungam)$ for Fukaya-Floer homology groups  associated with the metrics and sections. Then we have a canonical isomorphism
 between $HFF_*(Y; \ungam)$ and $HFF_*'(Y; \ungam)$.
\end{proposition}

Assume that we have the following data:

\vspace{3mm}

\begin{itemize}
\item
two oriented, closed Riemannian 3-manifolds $(Y_0, g_0)$, $(Y_1, g_1)$,

\vspace{3mm}

\item
$U(2)$-bundles $Q_0, Q_1$ over $Y_0, Y_1$ satisfying Hypothesis \ref{ass flat conn},

\vspace{3mm}

\item
sets of loops 
$\underline{\gamma} = \{ \gamma_l \}_{l=1}^d$,
$\underline{\gamma}' = \{ \gamma_l' \}_{l=1}^{d}$,

\vspace{3mm}

\item
a Riemannian bordism $(X, G)$ between $(Y_0, g_0), (Y_1, g_1)$,

\vspace{3mm}

\item
a $U(2)$-bundle $P$ over $X$ with $P|_{Y_0}=Q_0$, $P|_{Y_1}=Q_1$,

\vspace{3mm}

\item
oriented, compact surfaces $\Sigma_l$ embedded in $X$ with boundary $\gamma_l \coprod \gamma_l'$.

\end{itemize}

\vspace{3mm}

We introduce the following notations:

\vspace{3mm}

\begin{itemize}

\item
$\hat{X}:= X \cup (Y_0 \times \R_{ \geq 0 }) \cup ( Y_1 \times \R_{ \geq 0}) $,

\vspace{3mm}

\item
the extension $\hat{P}$ of $P$ to $\hat{X}$,

\vspace{3mm}

\item
the extension $\hat{G}$ of $G$ to $\hat{X}$,

\vspace{3mm}

\item
$\hat{\Sigma}_l := \Sigma_l \cup (\gamma_l \times \R_{ \geq 0}) \cup (\gamma_l' \times \R_{\geq 0})$.

\end{itemize}

Let $\rho, \sigma$ be flat connections on $Q_0, Q_1$ respectively. We denote the moduli space of instantons on $\hat{P}$ with limits $\rho, \sigma$ by $M_{ \hat{X} }(\rho, \sigma)$. Using families of twisted $\bar{\partial}$ operators on $\hat{\Sigma}_l$, we get the line bundles $\cL^{\otimes 2}_{\hat{\Sigma}_l}(\rho, \sigma)$ over $M_{ \hat{X} }(\rho, \sigma)$. As the proof of Proposition \ref{prop section}, we can show that for $\rho, \sigma$ with $\dim M_{ \hat{X} }(\rho, \sigma) < 8$ there are sections $\hat{s}_l(\rho, \sigma)$ of $\cL^{\otimes 2}_{\hat{ \Sigma }_l}(\rho, \sigma)$ satisfying the following conditions:

\begin{itemize}

\item
Let $\{ [A^{\alpha}] \}_{\alpha}$ be a sequence in $M_{ \hat{X} }(\rho, \sigma)$  converging to 
\[
([A_1^{\infty}], \dots, [A_{r}^{\infty}], [A^{\infty}], [A_1^{' \infty}], \dots, [A_{r'}^{' \infty}]),
\]
where 
\[
\begin{split}
[A_i^{\infty}] & \in M_{Y_0 \times \R}^0(\rho(i-1), \rho(i)), \\
[A^{\infty}] & \in M_{ \hat{X} }(\rho(r), \rho'(0)), \\
[A_i^{' \infty}] & \in M_{ Y_1 \times \R }^0 (\rho'(i-1), \rho'(i)).
\end{split}
\]
Then we have
\[
\hat{s}_l ([A^{\alpha}]) \longrightarrow
s_l([A_1^{\infty}]) \boxtimes \cdots \boxtimes s_{l}([A_r^{\infty}]) \boxtimes
\hat{s}_l ([A^{\infty}]) \boxtimes
s_l'([A_1^{' \infty}]) \boxtimes \cdots \boxtimes s_l'([A_{r'}^{' \infty}]).
\]

\vspace{3mm}

\item
For $L \subset \{ 1, \dots, d \}$ with $\dim M_{ \hat{X} }(\rho, \sigma) < 2 |L|$, the intersection $M_{ \hat{X} }(\rho, \sigma; L) = M_{ \hat{X} }( \rho, \sigma ) \cap \cap_{l \in L} V_{l}$ is empty.

\vspace{3mm}

\item
For $L$ with $\dim M_{ \hat{X} }(\rho, \sigma) = 2|L|$, the intersection $M_{ \hat{X} }(\rho, \sigma; L)$ is transverse and compact. Hence the intersection is finite.

\end{itemize}

\begin{lemma} \label{lem chain map}
Using the data $(\hat{X}, \hat{G}, \hat{\Sigma}_l, \hat{s}_l)$, we can define a homomorphism
\[
\zeta:CFF_*(Y_0; \underline{\gamma}) \longrightarrow CFF_*(Y_1; \underline{\gamma}')
\]
with the following property:

\begin{itemize}
\item
The homomorphism $\zeta$ is a chain map. That is,
$\partial_1 \circ \zeta = \zeta \circ \partial_0$.

\vspace{2mm}

\item
The induced homomorphism $\zeta_*:HFF_*(Y_0; \ungam) \rightarrow HFF_*(Y_1; \ungam')$ is independent of the metric $G$ and the sections $\hat{s}_l(\rho, \sigma)$.

\vspace{2mm}

\item
Assume that we have a bordism $(X', G', P', \{ \Sigma_l' \}_{l=1}^d)$ from $(Y_1, g_1, Q_1, \ungam')$ to other 4-tuple $(Y_2, g_2, Q_2, \ungam'')$. Let $\zeta'$ be the map from $CFF(Y_1;\ungam')$ to $CFF(Y_2;\ungam'')$.
For $T > 0$ we define a bordism $X^{\# (T)}$ between $X_0, X_2$ to be
\[
X^{(T)} := 
X \cup (Y_1 \times [0, T]) \cup X'.
\]
The metrics $G, G'$ naturally induce a metric $G^{(T)}$ on $X^{(T)}$ and we have surfaces $\Sigma_l^{(T)} = \Sigma_l \cup \gamma' \times [0, T] \cup \Sigma_l'$ with boundary $\gamma_l \coprod \gamma_l''$. So we have the map $\zeta^{(T)}$ from $CFF_*(Y_0; \ungam)$ to $CFF_*(Y_2;\ungam'')$. Then for large $T > 0$,
\[
\zeta_* \circ \zeta'_* = \zeta^{(T)}_*.
\]

\end{itemize}

\end{lemma}

We can derive Proposition \ref{prop well-defined} from Lemma \ref{lem chain map}.
Take two Riemannian metric $g_0, g_1$ on $Y$ and two sets $\{ s_l(\rho, \sigma) \}_{l}$, $\{ s_l'(\rho, \sigma) \}_l$ of sections of $\cL_{l}^{\otimes 2}(\rho, \sigma)$. Put $X=Y \times [0, 1]$, $\Sigma_l = \gamma_l \times [0, 1]$, $P=\pi^* Q$. Here $\pi$ is the projection from $X=Y \times [0,1]$ to $Y$. Choose a Riemannian metric $G$ on $X$ with $G|_{Y_0}=g_0$, $G|_{Y_1} = g_1$. These induce the maps
\[
\begin{split}
\zeta_*:HFF_*(Y; \ungam) &\longrightarrow HFF_*'(Y; \ungam), \\
\zeta'_* : HFF_*'(Y; \ungam) &\longrightarrow HFF_* (Y;\ungam).
\end{split}
\]

On the other hand, the metric $G_0 = g_0 + dt^2$ and the pull-backs $p^*(s_l(\rho, \sigma))$ of $s_l(\rho, \sigma)$ by $p:\M \rightarrow \0M$ induce an endomorphism $\zeta''_*$ of $HFF_*(Y;\ungam)$. It follows from the construction of $\zeta$ we will see below that $\zeta''_*$ is the identity map. 
The third part of Lemma \ref{lem chain map} implies 
\[
\zeta'_* \circ \zeta_* = \zeta''_* =  id.
\]
Similarly the composition $\zeta_* \circ \zeta'_*$ is the identity map of $HFF_*'(Y; \ungam)$. Thus $\zeta_*$ is an isomorphism from $HFF_*(Y;\ungam)$ to $HFF_*'(Y; \ungam)$.

\vspace{2mm}

We begin the proof of Lemma \ref{lem chain map}. Recall that the degree $\delta_{Y_0}( [\rho] )$, $\delta_{Y_1}( [\sigma])$ are defined by
\[
\begin{split}
\delta_{Y_0} ( [ \rho ] ) &\equiv
\ind D_{ A(\rho, \rho_0) } \mod 8, \\
\delta_{Y_1} ([ \sigma ] ) &\equiv
\ind D_{ A( \sigma, \rho_1 ) } \mod 8
\end{split}
\]
for some fixed flat connections $\rho_0, \rho_1$ over $Y_0, Y_1$.
For simplicity of notations, we suppose that 
\[
\ind D_{ A(\rho_0, \rho_1) } \equiv 0 \mod 8.
\]
Here $A(\rho, \sigma)$ is a connection on $\hat{X}$ with limits $\rho_0, \rho_1$. Then the dimension of $M_{ \hat{X} }(\rho, \sigma)$ is equal to $\delta_{Y_1}(\sigma) - \delta_{Y_0}(\rho)$ modulo 8. 

For $0 \leq \beta_0 \leq \beta_1 \leq d$, $L_0 \subset L_1 \subset \{ 1, \dots, d \}$ with $| L_0 | = \beta_0$, $| L_1 | = \beta_1$ and generators $[\rho] \in CF_{j-2\beta_0}(Y_0)$, $[\sigma] \in CF_{j-2\beta_1}(Y_1)$, the intersection $M_{ \hat{X} }(\rho, \sigma; L_1 \backslash L_0)$ is finite.
We define  $< \zeta ([\rho] \otimes \gamma_{L_1}), [\sigma] \otimes \gamma_{L_2} >$ by
\[
< \zeta ([\rho] \otimes \gamma_{L_0}), [\sigma] \otimes \gamma_{L_1}' >=
\# M_{ \hat{X} }(\rho, \sigma; L_1 \backslash L_0).
\]
Then the matrix elements give the map $\zeta:CFF_*(Y_0; \ungam) \rightarrow CFF_*(Y_1; \ungam')$. That is,
\[
\zeta([\rho] \otimes \gamma_{L_0}) =
\sum_{\beta_2} \sum_{ [\sigma] } \sum_{ L_1 }
< \zeta ([\rho] \otimes \gamma_{L_0}), [\sigma] \otimes \gamma_{L_1}' >
[\sigma] \otimes \gamma_{L_1}'.
\]
We give outline of the proof that $\zeta$ has the properties stated in Lemma \ref{lem chain map}.

\vspace{2mm}

The first part follows from a similar discussion to that in the proof that $\partial \circ \partial = 0$. Take a generator $[\tau] \in CF_{j-2\beta_1-1}(Y_1)$. Then we have a cut-down moduli space
$
M_{\hat{X}}(\rho, \tau; L_1 \backslash L_0)
$
with dimension $1$. Counting the number of the end of this moduli space, we get
\[
\partial_1 \circ \zeta = \zeta \circ \partial_0.
\]

\vspace{2mm}

To prove the second part, take other metric $\hat{G}'$ and sections $\hat{s}_l' (\rho, \sigma)$ over $\hat{X}$. Then we have another map $\zeta':CFF_*(Y_0; \ungam) \rightarrow CFF_*(Y_1,;\ungam')$. We will construct a chain homotopy
\[
H:CFF_*(Y_0; \ungam) \longrightarrow CFF_{*+1}(Y_1; \ungam')
\]
such that
\begin{equation} \label{eq chain homotopy}
\zeta' - \zeta = H \circ \partial_0 + \partial_1 \circ H.
\end{equation}
Take a path $\{ G^{s} \}_{ 0 \leq s \leq 1 }$ of metrics on $X$ from $G$ to $G'$, and put
\[
\overline{M}(\rho, \sigma) := 
\bigcup_{0 \leq s \leq 1} 
M_{ \hat{X}. \hat{G}^{s} }(\rho, \sigma) \times \{ s \}.
\]
Then $\overline{M}(\rho, \sigma)$ is smooth for generic paths of metrics. We can define line bundles $\overline{\cL}^{\otimes 2}_{l}(\rho, \sigma)$ over $\overline{M}(\rho, \sigma)$ as usual. 
As in Subsection \ref{ss sections},
for $\rho, \sigma$ with $\dim \overline{M}(\rho, \sigma) < 8$, we can construct sections $\overline{s}_l(\rho, \sigma)$ such that $\overline{s}(\rho, \sigma)$ are compatible with gluing maps, $\{ \overline{s}_l(\rho, \sigma) \}_l$ satisfy the transversality conditions as in Proposition \ref{prop section}, and 
\[
\overline{s}_l(\rho, \sigma)|_{M_{ \hat{X}, \hat{G}^0 }(\rho, \sigma) \times \{ 0 \}  } =s_l(\rho, \sigma),
\quad
\overline{s}_l(\rho, \sigma)|_{M_{ \hat{X}, \hat{G}^1 }(\rho, \sigma) \times \{ 1 \}  } =s_l'(\rho, \sigma).
\]
For $0 \leq \beta_0 \leq \beta_1 \leq d$, $L_0 \subset L_1 \subset \{ 1, \dots, d \}$ with $| L_0 | = \beta_0$, $| L_1 | = \beta_1$ and generators $[\rho] \in CF_{j-2\beta_0}(Y_0)$, $[\sigma] \in CF_{j-2\beta_1+1}(Y_1)$, we have a moduli space $\overline{M}(\rho, \sigma)$ with dimension $2(\beta_1-\beta_0)$. Cutting down $\overline{M}(\rho, \sigma)$ by the sections, we have the finite set $\overline{M}(\rho, \sigma;L_1 \backslash L_0)=\overline{M}(\rho, \sigma) \cap \bigcap_{l \in L_1 \backslash L_0} \overline{V}_l$. Here $\overline{V}_l$ is the zero locus of $\overline{s}_l(\rho, \sigma)$.
We put
\[
< H( [\rho] \otimes \gamma_{ L_0 }  ), [\sigma] \otimes \gamma_{L_1} > :=
\# \overline{M}(\rho, \sigma; L_1 \backslash L_0).
\]
These matrix elements give $H:CFF_{j}(Y_0;\ungam) \rightarrow CFF_{j+1}(Y_1;\ungam')$ as usual.
We can show the chain homotopy condition (\ref{eq chain homotopy}) by counting the number of the ends of the cut-down moduli spaces $\overline{M}(\rho, \tau;L_1 \backslash L_0)$ for generators $[\tau] \in CF_{j-2\beta_1}(Y_1)$.
Thus we obtain the second part of Lemma \ref{lem chain map}.

\vspace{2mm}

The proof of the third part is essentially same as the proof of the gluing formula for Donaldson invariants which will be  given in the next subsection, and we omit the proof here.

\vspace{3mm}

We give a remark on the dependence of $HFF_*(Y;\ungam)$ on $\ungam$.
It seems that $HFF_*(Y; \ungam)$ depends only on the homology classes $[\gamma_l] \in H_1(Y;\Z)$.
If $\gamma_l$ and $\gamma_l'$ are homologous, then we have oriented, compact, surfaces $\Sigma_l$ in $Y \times [0, 1]$ with boundary $\gamma_l \coprod \gamma_l'$. As above we can construct a linear map 
\[
\zeta_{\underline{\Sigma}}:CFF_*(Y;\ungam) \rightarrow CFF_*(Y; \ungam')
\] 
where $\underline{\Sigma} = (\Sigma_1, \dots, \Sigma_{d})$.  
To prove that $\zeta_{\underline{\Sigma} }$ induces an isomorphism from $HFF_*(Y; \ungam)$ to $HFF_*(Y; \ungam')$, it is sufficient to show that the map $\zeta_{ \underline{\Sigma} }$ has the properties as in Lemma \ref{lem chain map}. 
The most difficult part of the proof is to prove that the induced map between the Fukaya-Floer homology groups is independent of $\underline{\Sigma}$.

Suppose that we have another bordism $\Sigma_l'$ from $\gamma_l$ to $\gamma_l'$. 
We would like to show that the two maps between the Fukaya-Floer homology groups induced by $\zeta_{\underline{\Sigma}}$ and $\zeta_{\underline{\Sigma}'}$ are the same.
A natural way to prove this is to use an isomorphism between $\cL_{ \hat{\Sigma}_1 }(\rho, \sigma)$ and $\cL_{ \hat{\Sigma}_1' }(\rho, \sigma)$ which is compatible with the gluing maps in an appropriate sense, if the isomorphism exists.
Hence the problem reduces to the existence of the isomorphism.
This can be regarded as a generalization of the fact that the index of a family of elliptic differential operators on a closed manifold depends only on the bordism class.

We can take a compact, oriented manifold $W$ with corner, whose boundary is 
\begin{equation} \label{eq boundary}
\Sigma_l \cup \Sigma_l' \cup (\gamma_l \times [0, 1]) \cup (\gamma_l' \times [0,1])
\end{equation}
as follows.
Since $[\gamma_l]$ is equal to $[\gamma_l']$ in $H_1(Y;\Z)$, there is a complex line bundle $L$ over $Y$ and sections $s$, $s'$ of $L$ with $s^{-1}(0)=\gamma_l$, $s^{' -1}(0) = \gamma_l'$. Moreover there are sections $\tilde{s}$, $\tilde{s}'$ of the complex line bundle $L \times [0,1]$ over $Y \times [0, 1]$ with \begin{gather*}
\tilde{s}|_{ Y \times \{ 0 \} } = s, \quad
\tilde{s}|_{ Y \times \{ 1 \} } = s', \quad
\tilde{s}^{-1}(0) = \Sigma_l,  \\
\tilde{s}'|_{ Y \times \{ 0 \} } = s, \quad
\tilde{s}'|_{ Y \times \{ 1 \} } = s', \quad
\tilde{s}^{'-1}(0) = \Sigma_l'.
\end{gather*}

Take a section $\hat{s}$ of the line bundle $L \times [0, 1] \times [0, 1]$ over $Y \times [0, 1] \times [0, 1]$ with
\begin{gather*}
\hat{s}|_{ Y \times [0, 1] \times \{ 0 \} } = \tilde{s}, \quad
\hat{s}|_{ Y \times [0, 1] \times \{ 1 \} } = \tilde{s}',  \quad
\hat{s}|_{ Y \times \{ 0 \} \times \{ u \} } = s, \quad
\hat{s}|_{ Y \times \{ 1 \} \times \{ u \} } = s' \quad
(u \in [0, 1]).
\end{gather*}
For a generic $\hat{s}$, the zero locus $W$ of $\hat{s}$ is a compact, oriented manifold whose boundary is (\ref{eq boundary}).
Attaching $\gamma_l \times \R_{ \geq 0}  \times [0, 1]$, $\gamma_l' \times \R_{ \geq 0 } \times [0, 1]$ to $W$, we get a ``bordism'' $\hat{W}$ from $\hat{\Sigma}_l$ to  $\hat{\Sigma}_l'$, i.e. $\hat{W}$ is an oriented manifold with boundary $\hat{\Sigma}_l \coprod \hat{\Sigma}_l'$.

If we mimic the proof of the invariance of the index of operators over closed manifolds with respect to bordism, it may be possible to prove that there is an isomorphism from $\cL_{\hat{\Sigma}_l}(\rho,\sigma)$ to $\cL_{\hat{\Sigma}_l'}(\rho, \sigma)$ which is compatible with the gluing maps.
However, we do not discuss this in this paper and we just state it as a conjecture.

\begin{conjecture} \label{conj invariance}
In the above notations, the bordism $\hat{W}$ between $\hat{\Sigma}_l$ and $\hat{\Sigma}_l'$ gives a natural isomorphism between $\cL_{\hat{\Sigma}_l}(\rho, \sigma)$ and $\cL_{ \hat{\Sigma}_l' }(\rho, \sigma)$ which are compatible with the gluing maps.
\end{conjecture}

If this is true, $HFF_*(Y; \ungam)$ depends only on the homology classes $[\gamma_l] \in H_1(Y;\Z)$.

\subsection{Relative Donaldson invariants}

In  this subsection, we will see that
the Fukaya-Floer homology groups allow us to extend the Donaldson invariants for closed 4-manifolds to 4-manifolds with boundary. 

Let $Y$ be a closed, oriented 3-manifold and $\ungam = \{ \gamma_l \}_{l=1}^{d}$ be a set of loops in $Y$, where $0 \leq d \leq 3$. For compact, oriented 4-manifolds $X_0$ with boundary $Y$, we write $\hat{X}_0$ for $X_0 \cup ( Y \times \R_{ \geq 0})$. We also denote $\Sigma_{l} \cup ( \gamma_l \times \R_{\geq 0})$ by $\hat{\Sigma}_l$ for surfaces $\Sigma_l$ embedded in $X_0$ with boundary $\gamma_l$.
Assume that $X_0$ is simply connected and $b^+ >1$. Moreover suppose that we have a $U(2)$-bundle $P_0$ on $X_0$ such that the restriction $Q = P_0 |_{Y}$ satisfies Hypothesis \ref{ass flat conn}. 
Take a metric $\hat{g}_0$ on $\hat{X_0}$ such that the restriction of $\hat{g}_0$ to $Y \times \R_{\geq 0}$ is equal to $g_Y + dt^2$ for some metric $g_Y$ on $Y$.
Using families of twisted $\bar{\partial}$ operators over $\hat{\Sigma}_l$, we can define complex line bundles $\cL_{l}^{\otimes 2}(\rho)$ over the moduli space $M_{\hat{X}_0}(\rho)$ of instantons on the extension $\hat{P}_0$ of $P_0$ to $\hat{X}_0$ with limit $\rho$. As before, we can show that for $\rho$ with $\dim M_{\hat{X}_0}(\rho)<8$ there are sections $s_l(\rho)$ of the line bundles with properties similar to those in Proposition \ref{prop section}.

Let $\rho_0$ be the fixed flat connection used to define $\delta_{Y}$. For simplicity of notations, we suppose that
\[
\ind D_{A(\rho_0)} \equiv 0 \mod 8,
\]
where $A(\rho_0)$ is a connection on $\hat{P}_0$ with limit $\rho_0$. Then we have
\[
\dim M_{ \hat{X}_0 }(\rho) \equiv -\delta_Y ( [\rho] ) \mod 8.
\]
Under these hypotheses,  we define an element $\psi_{X_0} = \psi_{ \hat{X}_0 } \in CFF_0(Y; \ungam)$ as follows.
Let $\beta$ be an integer with $0 \leq \beta \leq d$ and take $L \subset \{ 1, \dots, d \}$ with $| L | = \beta$ and a generator $[\rho] \in CF_{-2\beta}(Y)$. Then we get a number
\[
< \psi_{X_0}, [\rho] \otimes \gamma_L > := \# M_{ \hat{X}_0 }(\rho; L).
\]
Here $M_{ \hat{X}_0}(\rho; L)$ is the 0-dimensional cut-down moduli space by the sections $\{ s_l(\rho) \}_{l \in L }$.
We define $\psi_{X_0}$ by
\[
\psi_{X_0} := 
\sum_{\beta} \sum_{L} \sum_{[\rho]}
< \psi_{X_0}, [\rho] \otimes \gamma_L > [ \rho ] \otimes \gamma_L
\in
CFF_0(Y; \ungam).
\]

The following proposition follows from a combination of a discussion of the case when $Y$ is a homology 3-sphere and techniques we have developed in this paper.

\begin{proposition}
The chain $\psi_{X_0}$ is a cycle, and the class $\Psi_{X_0}=\Psi_{X_0}(\Sigma_1, \dots, \Sigma_{d}) \in HFF_0(Y; \ungam)$ represented by $\psi_{X_0}$ is independent of the metric and the sections.
\end{proposition}

\begin{remark}
It seems that $\Psi_{X_0}(\Sigma_1, \dots, \Sigma_d)$ depend only on the homology classes $[\Sigma_1], \dots, [\Sigma_d] \in H_2(X_0,Y; \Z)$.
This is true if an analogy of Conjecture \ref{conj invariance} holds for the line bundles $\cL^{ \otimes 2 }_l(\rho)$.
\end{remark}

\subsection{Gluing formula}
\label{ss gluing D inv}

We consider a situation where a closed, oriented 4-manifold $X$ is cut into two parts $X_0$, $X_1$ along a closed 3-manifold $Y$. Let $P$ be a $U(2)$-bundle  over $X$ such that the restriction $Q=P|_{Y}$ satisfies Hypothesis \ref{ass flat conn} and $\dim M_P = 2d$ for some integer $d$ with $0 \leq d \leq 3$.
We suppose that $X_0, X_1$ are simply connected and $b^+$ of $X_0, X_1$ are larger than $1$.
Let $\Sigma_1, \dots, \Sigma_d$ be embedded surfaces in $X$ such that the intersections $\gamma_l := Y \cap \Sigma_l$ are diffeomorphic to $S^1$.
We denote $X_0 \cap \Sigma_l$, $X_1 \cap \Sigma_l$ by $\Sigma_l', \Sigma_l''$.
Then we have relative invariants
\[
\Psi_{X_0}(\Sigma_1', \dots, \Sigma_{d}') \in HFF_0(Y; \ungam), \quad
\Psi_{X_1}(\Sigma_1'', \dots, \Sigma_{d}'') \in HFF_{2d-3}( \bar{Y}; \ungam).
\]
We will express the invariant $\Psi_{X}([\Sigma_1],\dots, [\Sigma_d])$ of the closed manifold $X$ in terms of the relative invariants $\Psi_{X_0}$, $\Psi_{X_1}$. To do this, we define a pairing 
\[
< \ , \ >:HFF_{j}(Y; \ungam) \otimes HFF_{-j-3+2d}( \bar{Y}; \ungam) 
\longrightarrow
\Q
\]
as follows. 

For flat connections $\rho$, we have
\[
\delta_{ \bar{Y} }([\rho]) \equiv -\delta_{Y}([\rho]) - 3 \mod 8.
\]
Hence $CF_{j}(Y)$ and $CF_{-j-3}( \bar{Y})$ are dual to each other.
We define a pairing 
\[
< \ , \ > : CFF_{j}(Y;\ungam) \otimes CFF_{-j-3+2 d}( \bar{Y}; \ungam) \longrightarrow \Q
\]
by setting $[\rho] \otimes \gamma_{L^c} \in CF_{-j-3+2\beta}( \bar{Y}) \otimes \Q< \gamma_{L^c}>$ as the dual element of $[\rho] \otimes \gamma_{L} \in CF_{j-2\beta}(Y) \otimes \Q < \gamma_L >$.
Here $L$ is a subset of $\{ 1, \dots, d \}$ with $| L | = \beta$ and $L^c$ is the complement of $L$. It is easy to see the following.

\begin{lemma}
$< \partial ( [\rho] \otimes \gamma_{L}), [\sigma] \otimes \gamma_{L'} > = 
\pm < [\rho] \otimes \gamma_L, \partial( [\sigma] \otimes \gamma_{L'}) >.
$
\end{lemma}

Therefore we get the pairing
\[
< \, \ >:HFF_{j}(Y; \ungam) \otimes HFF_{-j-3+2d}(Y;\ungam) \longrightarrow \Q
\]
on Fukaya-Floer homology groups.

\begin{theorem} \label{thm gluing D inv}
In the above situation, we have
\[
\Psi_{X}([\Sigma_1], \dots, [\Sigma_d] ) =
\frac{1}{2^{d}}
< \Psi_{X_0}( \Sigma_1', \dots, \Sigma_d'), \Psi_{X_1}(\Sigma_1'', \dots, \Sigma_{d}'')>.
\]
\end{theorem}

To prove the gluing formula, we need sections of $\cL_{\Sigma_l}^{\otimes 2}$ which are compatible with gluing maps as usual.
Fix a Riemannian metric $g_Y$ on $Y$ and choose Riemannian metrics $g_0, g_1$ on $X_0, X_1$ such that the restrictions to $Y$ are equal to $g_Y$.
For each $T>0$, we define a manifold $X^{ \#(T) }$ by $X^{\#(T)}=X_0 \cup Y \times [0, T] \cup X_1$. Then $g_0, g_1$ naturally induce a Riemannian metric $g^{\#(T)}$ on $X^{\#(T)}$.
Fix a sequence $T^{\alpha} \rightarrow \infty$ and we write $X^{\alpha}$ for $X^{ \#(T^{\alpha}) }$. As in Subsection \ref{ss sections}, we can construct sections $s^{\alpha}_{l}$ of $\cL_{\Sigma_l}^{\otimes 2} \rightarrow M_{X^{\alpha}}$ such that
 for each sequence $[A^{\alpha}] \in M_{X^{\alpha}}$ converging to $([A_0^{\infty}], \dots, [A_r^{\infty}])$ the values $s_l([A^{\alpha}])$ at $[A^{\alpha}]$ converge to
$s_{l}([A_0^{\infty}]) \boxtimes \cdots \boxtimes s_l([A_r^{\infty}])$ in an appropriate sense.
By dimension counting, we can show that if $[A^{\alpha}]$ lie in the intersection
\[
M_{ X^{ \alpha } } \cap \bigcap_{l=1}^d V_l
\]
then $r=1$ and $([A_0^{\infty}], [A_1^{\infty}]) \in M_{ \hat{X}_0 }(\rho; L) \times M_{ \hat{X_1} }(\rho;L^c)$. Here $[\rho] \in CF_{-2\beta}(Y)$ with $0 \leq \beta \leq d$ and $L$ is a subset of $\{ 1, \dots, d \}$ with $| L |= \beta$. The transversality conditions for the sections imply that $M_{ \hat{X}_0 }(\rho; L)$ and $M_{ \hat{X_1} }(\rho;L^c)$ are finite sets. From an argument like that in the proof of (\ref{eq M' end}), we can show there is a natural identification
\[
M_{X^{\alpha}} \cap \bigcap_{l=1}^d V_l
\cong
\bigcup_{ \beta = 0}^d
\bigcup_{ 
\begin{subarray}{c}
\delta_{Y}( [\rho] ) \equiv \\
-2\beta \mod 8
\end{subarray}
}
\bigcup_{ | L | = \beta }
M_{ \hat{X}_0 }(\rho;L) \times M_{ \hat{X}_1 }(\rho; L^c)
\]
for large $\alpha$.
Counting the number of elements of both sides with signs, we obtain the gluing formula.

\section{Floer homology for 2-torsion instanton invariants}
\label{s torsion}

In this section, we consider a variant $\Psi_{X}^{u_1}$ of Donaldson invariants for non-spin 4-manifolds $X$, which is a linear function
\[
\Psi_{X}^{u_1}:A'(X) \longrightarrow \Z_2.
\]
Here $A'(X)$ is the subspace of $\oplus_{d \geq 0 } H_2(X;\Z)^{ \otimes d}$ generated by the set
\[
\{ \ [\Sigma_1] \otimes \cdots \otimes [\Sigma_d] \ | 
\ [\Sigma_l] \in H_2(X;\Z), \ 
\ [\Sigma_l] \cdot [\Sigma_l] \equiv 0 \mod 2 \ \}.
\]
This invariant is defined by a 2-torsion cohomology class $u_1$ of the moduli spaces.
Originally this was defined for spin manifolds by Fintushel-Stern \cite{FS}, and the author extended to non-spin 4-manifolds in \cite{S}.
In this section, we construct a variant of Floer homology group and prove a gluing formula for $\Psi_X^{u_1}$.

\vspace{2mm}

In \cite{S}, it was shown that $\Psi_{X}^{u_1}$ is non-trivial for $X = \CP^2_0 \# \CP^2_1  \# \barCP2$, where $\CP^2_0$ and $\CP^2_1$ are copies of $\CP^2$.
More precisely
\begin{equation} \label{eq non vanishing}
\Psi_{X}^{u_1} ( -H_0 + E, H_1 - E) \equiv 1 \mod 2.
\end{equation}
Here $H_0, H_1$ and $E$ are the generators of $H_2(\CP^2_0;\Z)$, $H_2(\CP^2_1;\Z)$ and $H_2(\barCP2;\Z)$ respectively.
We can consider $X=\CP^2_0 \# \CP^2_1 \# \barCP2$ as a connected sum of $X_0 = \CP^2_0$ and $X_1=\CP^2_1 \# \barCP2$.
Since the only flat connection over $S^3$ is the trivial one, we can deduce from (\ref{eq non vanishing}) that the trivial flat connection has an important role in the gluing formula for $\Psi_{X}^{u_1}$ in contrast to the case of usual Donaldson invariants.
Note also that the homology class $-H_1 + E$ in (\ref{eq non vanishing}) is a sum of homology classes $-H_1$ and $E$ of $X_0$ and $X_1$ with self-intersection numbers odd.

On the other hand, for closed, simply connected, non-spin 4-manifolds $X_0, X_1$ with $b^+$ positive, we can show the following vanishing theorem.
For homology classes $[\Sigma_{l}] \in H_2(X_0;\Z)$, $[\Sigma_{k}'] \in H_2(X_1;\Z)$ with self-intersection numbers even, 
\begin{equation} \label{eq vanishing}
\Psi_{X_0 \# X_1}^{u_1}([\Sigma_1], \dots, [\Sigma_{d_0}], [\Sigma_{1}'], \dots, [\Sigma_{d_1}']) \equiv 0 \mod 2.
\end{equation}
This follows from a standard dimension-counting argument.
The formula (\ref{eq vanishing}) implies that the key of the non-vanishing result (\ref{eq non vanishing}) is that the homology class $-H_1 + E$ is split into two homology classes of $X_0, X_1$ with self-intersection numbers odd.

In this section, we consider a situation where $X$ has a decomposition $X = X_0 \cup_{Y} X_1$ for a homology 3-sphere $Y$ and a closed, oriented surface $\Sigma$ in $X$ is split into two surfaces with self-intersection numbers odd along $Y$.
We will give a gluing formula for $\Psi_{X}^{u_1}([\Sigma])$ in terms of differential-topological date about $X_0, X_1$.
This situation is similar to that of Section \ref{section Fukaya-Floer} and we can apply the method developed in Section \ref{section Fukaya-Floer}.
The main difference is that a $U(2)$-bundle over a homology 3-sphere has the trivial flat connection, which is reducible, and we need to pay attention to the effect of the trivial flat connection on the gluing formula.

\subsection{2-torsion instanton invariants for closed manifolds}

We will summarize the construction of 2-torsion instanton invariants for closed  non-spin 4-manifolds which is given in \cite{S}.
See also \cite{FS}.

Let $X$ be a closed, oriented, simply connected, non-spin 4-manifold with $b^+ > 1$.
Take a $U(2)$-bundle $P$ over $X$ with $w_2( P ) = w_2(X)$.
For $[\Sigma] \in H_2(X; \Z)$, we have the complex line bundle
\[
\tilde{\cL}_{\Sigma} \stackrel{\C}{\longrightarrow} \tilde{\cB}_{ \Sigma }.
\]
We can show that if $[\Sigma] \cdot [\Sigma] \equiv 0 \mod 2$ then the center $\{ \pm 1 \}$ of $SU(2)$ acts trivially on the line bundle. Hence we get a line bundle
\[
\cL_{\Sigma} = \tilde{\cL}_{\Sigma} / SO(3) \longrightarrow \cB_{\Sigma}^*.
\]

Let $c$ be a spin-c structure of $X$. For each connections $A$ on $P$, we have the twisted Dirac operator
\[
\Dir_A : \Gamma(S^+ \otimes E) \longrightarrow \Gamma (S^- \otimes E).
\]
Here $E$ is the rank two complex vector bundle associated with $P$ and $S^{\pm}$ are the spinor bundles. If $c_1( \det c ) = - c_1(P)$, we have a ``real part" of the Dirac operators:
\[
(\Dir_A)_{\R}:\Gamma ( (S^+ \otimes E)_{\R} ) \longrightarrow \Gamma ( (S^- \otimes E)_{\R}).
\]
(See \cite{AMR}.) The family of real operators $\{ (\Dir_A)_{\R} \}_{[A] \in \tilde{\cB}_{P} }$ gives a real line bundle
\[
\tilde{\Lambda} \stackrel{\R}{\longrightarrow} \tilde{B}_{P}^*.
\]
We also suppose that $c_2(P) \equiv 0 \mod 2$. Then we can see that the center $\{ \pm 1 \}$ of $SU(2)$ acts trivially on $\tilde{\Lambda}$, and we obtain a real line bundle
\[
\Lambda = \tilde{\Lambda}/SO(3) \stackrel{\R}{\longrightarrow} \cB_{P}^*.
\]
We define $u_1 \in  H^1(\cB_P^*; \Z_2)$ to be $w_1(\Lambda)$.

When $b^+(X)$ is even, the virtual dimension of the moduli space $M_P$ is odd. We can write $\dim M_P = 2d + 1$ for some integer $d$. Assume $d \geq 0$. Then we define
\[
\Psi_{X}^{u_1} ([\Sigma_1], \dots, [\Sigma_{d}]) =
``< u_1 \cup c_1(\cL_{\Sigma_1}) \cup \cdots \cup c_1(\cL_{\Sigma_d}), [M_P] >" \in \Z_2.
\]
Here $[\Sigma_1], \dots, [\Sigma_d] \in H_2(X; \Z)$ with $[\Sigma_i] \cdot [\Sigma_i] \equiv 0 \mod 2$. In general $M_P$ is not compact, however we can define the pairing as follows. 
Let $s_{\Sigma_i}:\cB_{\Sigma_i}^* \rightarrow \cL_{\Sigma_i}$ be sections and $V_{\Sigma_i}$ be the zero loci. We can prove the following lemma by a standard dimension counting argument.

\begin{lemma} \label{lem compact}
Assume that the dimension of $M_P$ is $2d + r$ for some $d \geq 0$ and $0 \leq r \leq 3$. Then the intersection $M_P \cap V_{\Sigma_1} \cap \cdots \cap V_{\Sigma_{d}}$ is compact and smooth for generic sections $s_{\Sigma_i}$.
\end{lemma}

When $\dim M_P$ is $2d + 1$, the intersection
\[
M_P \cap V_{\Sigma_1} \cap \cdots \cap V_{\Sigma_d}
\]
is a compact, smooth manifold of dimension $1$.
Precisely the invariant $\Psi_{X}^{u_1}([\Sigma_1], \dots, [\Sigma_d])$ is defined to be
\[
< u_1, [M_P \cap V_{\Sigma_1} \cap \cdots \cap V_{\Sigma_d}] > \in \Z_2.
\]
We can show that $\Psi_{X}^{u_1}([\Sigma_1], \dots, [\Sigma_d])$ is independent of the metric on $X$ and sections of $\cL_{\Sigma_i}$, and hence it is a differential-topological invariant of $X$.

\begin{remark}
When we define the invariant $\Psi_{X}^{u_1}([\Sigma_1], \dots, [\Sigma_d])$, we use the pull-back $r_{\Sigma_i}^*(s_{\Sigma_i})$ of the section $s_{\Sigma_i}$ over $\cB_{\Sigma_i}^*$.
Here $r_{\Sigma_i}$ is the map $M_P \rightarrow \cB_{\Sigma_i}^*$ defined by restricting connections to $\Sigma_i$.
However Lemma \ref{lem compact} shows that when the dimension of $M_P$ is $2d + 1$ with $d \geq 0$ $M_P \cap V_{\Sigma_2} \cap \cdots \cap V_{\Sigma_{d}}$ is a compact, smooth manifold of dimension $3$. Hence we need not to use the pull-back $r_{\Sigma_1}^* (s_{\Sigma_1})$. We can use any section of $r_{\Sigma_1}^* (\cL_{\Sigma_1})$ which is transverse to the zero section.
\end{remark}

\subsection{Line bundles}

In this subsections, we introduce line bundles over $\cB_{Y \times \R}^*(\rho, \sigma)$ which are defined by families of Dirac operators over $Y \times \R$ and twisted $\bar{\partial}$ operators over surfaces in $Y \times \R$.
These are used to define a variant of Floer homology groups for 2-torsion instanton invariants.

\vspace{2mm}

Let $Y$ be a homology 3-sphere and $Q = Y \times U(2)$ be the trivial $U(2)$-bundle over $Y$.
Fix a connection $a_{\det}$ on $Q_{\det} = Y \times U(1)$. As before, connections on $Q$ and $\pi^* Q$ are compatible with $a_{\det}$, unless explicitly stated otherwise.

We assume the following.

\begin{hypothesis}
All flat connections on $Q$ are non-degenerate.
\end{hypothesis}

For irreducible flat connections $\rho, \sigma$, we define $\cB^*_{Y \times \R}(\rho, \sigma)$, $\tilde{\cB}_{Y \times \R}^* (\rho, \sigma)$ as before.
We have just one gauge equivalence class of projectively flat connections which are not irreducible. It is represented by a connection $\theta_0$ which induce the trivial connection on the adjoint bundle of $Q$.  
Fix a smooth map $g_1:Y \rightarrow SU(2)$ with degree $1$ and put $\theta_{a}:= g_1^a(\theta_0)$ for $a \in \Z$.
We define $\cB_{Y \times \R}^*(\theta_a, \sigma)$, $\tilde{\cB}_{Y \times \R}(\rho, \theta_a)$ as usual. Let $\Gamma_{\rho}$ be the stabilizer of $\rho$ in $\cG_{Q}$. Then we have a natural action of $\Gamma_{\rho} \times \Gamma_{\sigma}$ on $\tB$, and $\B = \tB / \Gamma_{\rho} \times \Gamma_{\sigma}$. Note that the action of the subgroup $\{ \pm (1, 1) \} \subset \Gamma_{\rho} \times \Gamma_{\sigma}$ on $\tB$ is trivial.

For flat connections $\rho, \sigma$, we have real line bundles
\[
\tilde{\Lambda}(\rho, \sigma) 
\stackrel{\R}{\longrightarrow}
\tilde{\cB}^*_{Y \times \R}(\rho, \sigma)
\]
induced by families of the real part of Dirac operators.
Since the action of $\{ \pm (1, 1) \} \subset \Gamma_{\rho} \times \Gamma_{\sigma}$ on the line bundle is not trivial in general, the line bundle may not descend to $\B$. To avoid this problem, we will introduce real line bundles $\underline{\R}(\theta_a, \rho)$ and $\underline{\R}(\sigma, \theta_b)$ over $\tB$ for irreducible flat connections $\rho, \sigma$ and $a, b \in \Z$. Choose connections $A(\theta_a, \rho)$, $A(\sigma, \theta_b)$ on $\pi^* Q \rightarrow Y \times \R$ such that
\[
\begin{split}
A(\theta_a, \rho) &=
\left\{
\begin{array}{cl}
\theta_a & \text{ on $Y \times (-\infty, -1)$ }, \\
\rho & \text{ on $Y \times (1, \infty)$ }.
\end{array}
\right. \\
A(\sigma, \theta_b) &=
\left\{
\begin{array}{cl}
\sigma & \text{ on $Y \times (-\infty, -1)$ }, \\
\theta_b & \text{ on $Y \times (1, \infty)$ }.
\end{array}
\right. 
\end{split}
\]
Then we have the Dirac operators
\[
\begin{split}
\Dir_{A(\theta_a, \rho)} &:
L_{4}^{2,(-\tau,\tau)} (S^+ \otimes E) \longrightarrow
L_3^{2, (-\tau, \tau)}(S^- \otimes E), \\
\Dir_{A(\sigma, \theta_b)} &:
L_{4}^{2,(-\tau,\tau)} (S^+ \otimes E) \longrightarrow
L_3^{2, (-\tau, \tau)}(S^- \otimes E).
\end{split}
\]
We denote the numerical index of these operators by $\Ind^{- + }\Dir_{A(\theta_a, \rho)}$, $\Ind^{-+}\Dir_{A(\sigma, \theta_b)}$.
Define $\underline{\R}(\theta_a, \rho)$ to be trivial line bundle with action of $\Gamma_{\rho} = \Z_2$ of wight $\aInd$. Similarly $\underline{\R}(\sigma, \theta_b)$ is the trivial real line bundle with $\Gamma_{\sigma}$ action of wight $\bInd$. Here we put
\[
\tilde{\Lambda}_{ab}(\rho, \sigma) :=
\aR \otimes \tilde{\Lambda}(\rho, \sigma) \otimes \bR.
\]
The wight of the action of $\{ \pm (1, 1) \} \subset \Gamma_{\rho} \times \Gamma_{\sigma}$ is 
\[
\aInd + \Ind^{-+} \Dir_{A} + \bInd = \abInd.
\]
Here $[A] \in \B$.
If $\abInd$ is even, the line bundle $\tilde{\Lambda}_{ab}(\rho, \sigma)$ descends to $\B$.

\begin{lemma} \label{lem ind Dir}
$\abInd = -(b-a)$.
\end{lemma}

\begin{proof}
Let $P_{b-a}$ be a $U(2)$-bundle over $Y \times S^1$ with $c_1(P_{b-a})=0$,  $c_2(P_{b-a})=b-a$.
It follows from the additivity of index and the index formula that
\[
\abInd = ch (P_{b-a}) \hat{A}(Y \times S^1)/[Y \times S^1].
\]
Since $ch(P_{b-a}) = 2 - c_2(P_{b-a})$ and $\hat{A}(Y \times S^1)=1$, we have
\[
\abInd = -<c_2(P_{b-a}), [Y \times S^1] > = -(b-a).
\]
\end{proof}

For $a, b \in \Z$ with $b - a \equiv 0 \mod 2$, we get a real line bundle
\[
\Lambda_{ab}(\rho, \sigma):= \tilde{\Lambda}_{ab}(\rho, \sigma)/SO(3)
\stackrel{\R}{\longrightarrow}
\B.
\]

Let $a, b, a', b'$ be integers with $a \equiv b \equiv a' \equiv b' \mod 2$. 
Then we have a natural isomorphism from $\underline{\R}(\theta_{a'},\theta_{a}) \otimes \tilde{\Lambda}_{ab}(\rho, \sigma) \otimes \underline{\R}(\theta_{b},\theta_{b'})$ to $\tilde{\Lambda}_{a'b'}(\rho,\sigma)$ and we can see that this isomorphism induces an isomorphism from $\Lambda_{ab}(\rho, \sigma)$ to $\Lambda_{a'b'}(\rho, \sigma)$. Therefore we obtain:

\begin{lemma} 
If $a \equiv b \equiv a' \equiv b' \mod 2$, $\Lambda_{ab}(\rho,\sigma)$ and $\Lambda_{a' b'}(\rho, \sigma)$ are isomorphic to each other.
\end{lemma}

\begin{definition}
For irreducible flat connections $\rho, \sigma$ and $a \in \{ 0, 1 \}$, we write $\Lambda^{(a)}(\rho, \sigma)$ for $\Lambda_{aa}(\rho, \sigma)$. Moreover we define $u_1^{(a)}=u_1^{(a)}(\rho, \sigma)$ by
\[
u_1^{(a)} := w_1( \aLam) \in H^1(\B; \Z_2).
\]
\end{definition}

Let $\gamma \cong S^1$ be a loop in $Y$. Then we have the determinant line bundle
\[
\tilde{\cL}_{\Gamma} \stackrel{\C}{\longrightarrow} \tB.
\]
Here $\Gamma = \gamma \times \R$.
This line bundle does not descend to $\B$ in general,  we can however apply the technique used to define $\Lambda_{ab}(\rho, \sigma)$.

Let $\theta_{\gamma, 0}$ be the restriction of $\theta_0$ to $\gamma$.
Choose a smooth map $g_{\gamma}:\gamma \rightarrow U(2)$ such that the homotopy class $[g_{\gamma} ]$  is the generator of $\pi_1(U(2)) \cong \Z$.
Then for $a \in \Z$ we put
\[
\theta_{\gamma, a} := g_{\gamma}^a (\theta_{\gamma,0}).
\]
Note that these connections are not compatible with the fixed connection $a_{\det}$ on $Q_{\det}$.
For $a, b \in \Z$ and irreducible flat connections $\rho, \sigma$, we take connections $A(\theta_{\gamma,a}, \rho)$, $A( \sigma, \theta_{\gamma, b} )$ on $\pi^* Q$ such that
\[
\begin{split}
A(\theta_{ \gamma, a }, \rho) &=
\left\{
\begin{array}{cl}
\theta_{\gamma, a} & \text{on $Y \times (-\infty, -1)$,} \\
\rho & \text{on $Y \times (1, \infty)$,}
\end{array}
\right. \\
A(\sigma, \theta_{\gamma, b}) &=
\left\{
\begin{array}{cl}
\sigma & \text{on $Y \times (-\infty, -1)$,} \\
\theta_{\gamma, b} & \text{ on $Y \times (1, \infty)$.  }
\end{array}
\right.
\end{split}
\]
(These connections are not compatible $a_{\det}$ either.)
Let $\aC$ and $\bC$ be the trivial complex line bundles over $\tB$ with $\Gamma_{\rho}$-action of weight $\gamaInd$ and $\Gamma_{\sigma}$-action of wight $\gambInd$ respectively. Put
\[
\tilde{\cL}_{\Gamma, ab}(\rho, \sigma) := 
\aC \otimes \tilde{\cL}_{\Gamma}(\rho, \sigma) \otimes \bC
\longrightarrow 
\tB.
\]
Then $\Gamma_{\rho} \times \Gamma_{\sigma}$ naturally acts on $\tilde{\cL}_{\Gamma, ab}(\rho, \sigma)$ and the weight of $\{ \pm (1, 1) \}$ is 
\[
\gamaInd + \Ind^{-+} \bar{\partial}_{A(\rho, \sigma)} + \gambInd =
\gamabInd.
\]
If this number is even, the line bundle $\tilde{\cL}_{\Gamma, ab}(\rho, \sigma)$ descends to $\B$.

\begin{lemma} \label{lem ind d-bar}
$\gamabInd = b-a$.
\end{lemma}

\begin{proof}
Let $P_{\gamma, b-a}$ be a $U(2)$ bundle over $\gamma \times S^1$ with $c_1(P_{\gamma, b-a})=b-a$. From the additivity of index and the index formula, we have
\[
\begin{split}
\gamabInd 
&=ch(P_{\gamma, b-a})\hat{A}(\gamma \times S^1)/[\gamma \times S^1] \\
&=(2 + c_1(P_{\gamma, b-a}))/[\gamma \times S^1] \\
&=b-a.
\end{split}
\]
\end{proof}

This lemma implies that for $a, b \in \Z$ with $a \equiv b \mod 2$ we can define
\[
\cL_{\Gamma,ab}(\rho,\sigma) :=
\tilde{\cL}_{\Gamma, ab}(\rho,\sigma) / \Gamma_{\rho} \times \Gamma_{\sigma}
\stackrel{\C}{\longrightarrow}
\B.
\]
We can also show that for $a \equiv b \equiv a' \equiv b' \mod 2$ we have a natural isomorphism between $\cL_{\Gamma,ab}(\rho,\sigma)$ and $\cL_{\Gamma,a'b'}(\rho,\sigma)$. We get two bundles $\cL_{\Gamma}^{(0)}(\rho,\sigma):=\cL_{\Gamma,00}(\rho,\sigma)$, $\1L:=\cL_{\Gamma,11}(\rho,\sigma)$. We will use only $\1L$ to define a variant of Floer homology.

\subsection{The complex}

Using the line bundles introduced in the previous subsection, we construct homology groups which allow us to extend 2-torsion instanton invariants to 4-manifolds with boundary.
The idea of the construction is fundamentally the same as Fukaya-Floer homology groups in Section \ref{section Fukaya-Floer}. The main difference is the existence of a reducible flat connection.

\vspace{2mm}

Fix $a \in \{ 0, 1 \}$ and take sections
\[
\begin{split}
s_{\Lambda}=s_{\Lambda}(\rho, \sigma) &: \0M \longrightarrow \aLam \\
s_{\cL}=s_{\cL}(\rho, \sigma) &: \0M \longrightarrow \1L.
\end{split}
\]
We denote the zero sets by $V_{\Lambda}$, $V_{\Lambda}$.
For sets
\[
L_0 = \emptyset, \quad 
L_1=\{ \Lambda \},\quad
L_2=\{ \cL \}, \quad
L_3=\{ \Lambda, \cL \},
\]
we define
\begin{gather*}
M_{Y \times \R}^0(\rho, \sigma;L_0):= M_{Y \times \R}^0(\rho, \sigma), \quad
M_{Y \times \R}^0(\rho, \sigma;L_1):= M_{Y \times \R}^0(\rho, \sigma) \cap V_{\Lambda}, \\
M_{Y \times \R}^0(\rho, \sigma;L_2):= M_{Y \times \R}^0(\rho, \sigma) \cap V_{\cL}, \ 
M_{Y \times \R}^0(\rho, \sigma;L_3) := M_{Y \times \R}^0(\rho, \sigma) \cap V_{\Lambda} \cap V_{\cL}.
\end{gather*}

\begin{lemma} \label{lem rho irr}
Let $\rho, \sigma$ be (possibly reducible) flat connections on a $U(2)$-bundle $Q$ over a homology 3-sphere $Y$ with $\dim \0M \leq 3$. Take a sequence $\{ [A^{\alpha} ] \}_{\alpha}$ in $\0M$ converging to $([A_1^{\infty}], \dots, [A_{r}^{\infty}])$ with $r>1$. Here $[A_i^{\infty}] \in \i0M$. Then $\rho(i)$ are irreducible for $i=1,\dots, r-1$.
\end{lemma}

\begin{proof}
The dimension of $\dim \M$ is given by
\[
\dim \M = \Ind^{++} D_{A^{\alpha}} = \Ind^{-+}D_{A^{\alpha}} - \dim H_{\rho}^0,
\]
where $H_{\rho}^0$ is the Lie algebra of $\Gamma_{\rho}$.
The additivity of index implies that
\[
\Ind^{- +} D_{A^{\alpha}} 
= \sum_{i=1}^r \Ind^{-+} D_{A_i^{\infty}}
= \sum_{i=1}^r \big( \dim \iM + \dim H_{\rho(i-1)}^0 \big).
\]
Hence we have
\[
\dim \M = \sum_{i=1}^r \dim \iM + \sum_{i=1}^{r-1} \dim H_{\rho(i)}^0.
\]
Assume that $\rho(i_0)$ is reducible for some $i_0$ with $1 \leq i_0 \leq r-1$.
Then $\rho(i_0)$ is the trivial flat connection and the dimension of $H_{\rho(i_0)}^0$ is $3$, and $\dim \iM \geq 1$. Hence
\[
4 \geq \dim \M \geq r + 3 >4.
\]
This is a contradiction.
\end{proof}

This lemma means that we can apply the method in the previous section to showing that we can take admissible sections of the line bundles $\aLam$, $\1L$ when $\dim \0M \leq 3$.
Thus we obtain the following.

\begin{lemma} \label{lem torsion section}
Let $\rho, \sigma$ be irreducible flat connections on $Q$ with $\dim \0M \leq 3$. Then we can take sections $s_{\Lambda}, s_{\cL}$ of $\aLam$, $\1L$ with the following properties.

\begin{itemize}

\item
Let $\{ [A^{\alpha}] \}_{\alpha}$ be a sequence in $\0M$ with $[A^{\alpha}] \rightarrow ([A_1^{\infty}], \dots, [A_{r}^{\infty}])$. Then we have
\[
\begin{split}
s_{\Lambda}([A^{\alpha}]) &\longrightarrow 
s_{\Lambda}([A_1^{\infty}]) \boxtimes \cdots \boxtimes s_{\Lambda}([A_r^{\infty}]), \\
s_{\cL}([A^{\alpha}]) &\longrightarrow
s_{\cL}([A_1^{\infty}]) \boxtimes \cdots \boxtimes s_{\cL}([A_r^{\infty}]).
\end{split}
\]

\item
Let $\beta$ be an integer with $0 \leq \beta \leq 3$. If $\dim \0M < \beta$, then
\[
M_{Y \times \R}^0 (\rho, \sigma;L_{\beta}) = \emptyset.
\]

\item
If $\dim \0M = \beta$, then the intersection $\bM$ is transverse and compact. Hence $\bM$ is a finite set.

\end{itemize}

\end{lemma}

In this section, we refer to $CF_i(Y)$ as the $\Z_2$-vector space generated by gauge equivalence classes $[\rho]$ of irreducible flat connections with $\delta_{Y}([\rho]) \equiv i \mod 8$.
Using admissible sections of $\Lambda^{(a)}(\rho, \sigma)$, $\cL^{(1)}_{ \Gamma }(\rho, \sigma)$, we define the complex as follows.

\begin{definition}
$\phantom{a}$

\begin{enumerate}
\item
\[
C_j^{(a)}(Y):=
\left\{
\begin{array}{cl}
\left(
\displaystyle \bigoplus_{\beta=0}^3 CF_{8n-\beta}(Y)
\right)
\oplus \Z_{2} < \theta_n > &
\text{ if $j=8n$, $n \equiv a + 1 \mod 2$  }, \\
\displaystyle \bigoplus_{\beta = 0}^3 CF_{j- \beta}(Y) &
\text{otherwise.}
\end{array}
\right.
\]

\item
We define $\partial^{(a)}:C_j^{(a)}(Y) \rightarrow C_{j-1}^{(a)}(Y)$ as follows.
\begin{enumerate}
\item
For integers $\beta_1, \beta_2$ with $0 \leq \beta_1 \leq \beta_2 \leq 3$ and generators $[\rho] \in CF_{j-\beta_1}(Y)$, $[\sigma] \in CF_{j-\beta_2-1}(Y)$, we put
\[
< \partial^{(a)} ([\rho]), [\sigma]> :=
\left\{
\begin{array}{cl}
\# \b12M \mod 2 & \text{ if $ L_{\beta_1} \subset L_{\beta_2}  $,  } \\
0 & \text{ otherwise. }
\end{array}
\right. 
\]

\item
For $[\rho] \in CF_{8n+1}(Y)$ with $n \equiv a+1 \mod 2$, we put
\[
< \partial^{(a)}([\rho]), \theta_n > := \# M_{Y \times \R}^0(\rho, \theta_n) \mod 2.
\]

\item
For $[\sigma] \in CF_{8n-4}(Y)$ with $n \equiv a + 1 \mod 2$, we put
\[
< \partial^{(a)}(\theta_n), [\sigma] > = \# M_{Y \times \R}^0(\theta_n, \sigma) \mod 2.
\]

\end{enumerate}

\end{enumerate}

\end{definition}

We claim that $\partial^{(a)} \circ \partial^{(a)}$ is identically zero.

\begin{lemma} \label{lem torsion complex}
$\partial^{(a)} \circ \partial^{(a)} = 0$.
\end{lemma}

We must show that for generators $[\rho] \in C_j^{(a)}(Y)$, $[\sigma] \in C_{j-2}^{(a)}(Y)$
\begin{equation} \label{eq torsion complex}
< \partial^{(a)} \circ \partial^{(a)}([\rho]), [\sigma] > \equiv 0 \mod 2.
\end{equation}
We split the proof into four cases.

\vspace{2mm}

\begin{enumerate}[(1)]

\item
Flat connections $\rho, \sigma$ are irreducible and $[\rho] \in CF_{j-\beta_1}(Y)$, $[\sigma] \in CF_{j-\beta_2-2}(Y)$ for some integers $\beta_1, \beta_2$ with $0 \leq \beta_2 - \beta_1 \leq 2$.

\vspace{2mm}

\item
$[\rho] \in CF_{8n+2}(Y)$, $\sigma = \theta_{n}$ for some $n \in \Z$ with $n \equiv a+1 \mod 2$.

\vspace{2mm}

\item
$\rho = \theta_n$, $[\sigma] \in CF_{8n-5}(Y)$ for some $n \in \Z$ with $n \equiv a + 1 \mod 2$.

\vspace{2mm}

\item
$[\rho] \in CF_{j}(Y)$, $[\sigma] \in CF_{j-5}(Y)$.

\end{enumerate}

\vspace{3mm}

To prove (\ref{eq torsion complex}), we need to describe the ends of $\0M$ as in the previous section.
In the cases (1), (2), (3), the dimension of the moduli space $\0M$ is less than 4. Lemma \ref{lem rho irr} means that (\ref{eq torsion complex}) follows from the same discussion in the previous section.
Therefore we need to consider only the case (4).

\begin{lemma}
Take generators $[\rho] \in CF_{j}(Y)$, $[\sigma] \in CF_{j-5}(Y)$. We can take sections $s_{\Lambda}$, $s_{\cL}$ of the line bundles $\aLam$, $\1L$ with the following property.
Let $\{ [A^{\alpha}] \}_{\alpha}$ be a sequence in $\0M$ converging to some $([A_1^{\infty}], \dots, [A_r^{\infty}])$. Here $[A_i^{\infty}] \in \i0M$. Suppose that $\rho(i)$ are irreducible for $i=1, \dots, r-1$. Then 
\[
\begin{split}
s_{\Lambda}([A^{\alpha}]) &\longrightarrow s_{\Lambda}([A_1^{\infty}]) \boxtimes \cdots \boxtimes s_{\Lambda}([A_r^{\infty}]), \\
s_{\cL}([A^{\alpha}]) &\longrightarrow s_{\cL}([A_1^{\infty}]) \boxtimes \cdots \boxtimes s_{\cL}([A_r^{\infty}]).
\end{split}
\]
\end{lemma}

We can show this by using a partition of unity on the moduli space as Proposition \ref{prop section}.
The following lemma follows from the usual dimension-counting argument.

\begin{lemma} \label{lem ML3 end}
Take generators $[\rho] \in CF_{j}(Y)$, $[\sigma] \in CF_{j-5}(Y)$.
Let $\{ [A^{\alpha}] \}_{\alpha}$ be a sequence in $\3M$ converging to $([A_1^{\infty}], \dots, [A_r^{\infty}])$ with $r > 1$. Then we have either

\begin{itemize}

\item
$r=2$ and 
\[
[A_1^{\infty}] \in M_{Y \times \R}^0(\rho, \rho(1);L_{\beta}), \quad
[A_2^{\infty}] \in M_{Y \times \R}^0(\rho(1), \sigma; L_3 \backslash L_{\beta})
\]
for some integer $\beta$ with $0 \leq \beta \leq 3$ and an irreducible flat connection $\rho(1)$ with $[\rho(1)] \in CF_{j-\beta - 1}(Y)$, 
or

\vspace{2mm}

\item
$r=2$, $\rho(1)=\theta_n$ for some $n$ and 
\[
\dim M_{Y \times \R}^0(\rho, \theta_n) = \dim M_{Y \times \R}^0(\theta_n, \sigma) = 0.
\]

\end{itemize}
\end{lemma}

This lemma means that the reducible flat connections $\theta_n$ do not appear in the description of the end of $\0M$ if $[\rho] \not \in CF_{8n+1}(Y)$ for any $n$. Therefore (\ref{eq torsion complex}) follows from a discussion like that in the previous section in this case.

From now on, we assume that $[\rho] \in CF_{8n+1}(Y)$, $[\sigma] \in CF_{8n-4}(Y)$. 
To describe the end of the cut-down moduli space $\bM$, we put
\[
\begin{split}
N_{irr} &:=
\bigcup_{\beta = 0}^3
\bigcup_{ [\rho(1)] } 
M_{Y \times \R}^0(\rho, \rho(1); L_{\beta}) \times M_{ Y \times \R }^0 (\rho(1), \sigma; L_3 \backslash L_{\beta} ), \\
N_{red} &:=
M_{Y \times \R}^0 (\rho, \theta_n) \times M_{Y \times \R}^0(\theta_n, \sigma), \\
N &:= N_{irr} \coprod N_{red}.
\end{split}
\]
Here $[\rho(1)]$ runs over generators of $CF_{8n+1-\beta}(Y)$.
For each $\unA =([A_1], [A_2]) \in N_{irr}$, we have a gluing map
\[
Gl_{\unA}:U_{[A_1]} \times (T_0, \infty) \times U_{[A_2]} \longrightarrow \0M
\]
for some precompact open neighborhoods  $U_{[A_1]}, U_{[A_2]}$ of $[A_1]$, $[A_2]$ and $T_0 > 0$, and for $\unA=([A_1], [A_2]) \in N_{red}$, we have a gluing map
\[
Gl_{\unA}:(T_0, \infty) \times SO(3) \longrightarrow \0M.
\]
For $T_1 > T_0$, we write $Gl_{\unA, T > T_1}$ for the restriction $Gl_{\unA}$ to the region where $T>T_1$. Here we put
\[
M'=M_{T_1}' := \3M \ \backslash \ \bigcup_{ \unA \in N } \Im Gl_{\unA, T>T_1}.
\]
Lemma \ref{lem ML3 end} means that $M'$ is compact. 
For irreducible flat connections $\rho(1)$, the sections of the line bundles over $M_{Y \times \R}(\rho, \rho(1))$, $M_{Y \times \R}^0(\rho(1), \sigma)$ satisfy the transversality conditions in Lemma \ref{lem torsion section}. Hence from a discussion like that in the proof of (\ref{eq M' end}), we have
\begin{equation} \label{eq M' end irr}
M' \cap \bigcup_{ \unA \in N_{irr} } \Im Gl_{ \unA, T_1-1 < T \leq T_1 }
\cong
N_{irr} \times (0, 1]
\end{equation}
for large $T_1$.
Perturbing the sections $s_{\Lambda}$, $s_{\cL}$ outside the end (\ref{eq M' end irr}) of $M'$, we get a 1-dimensional, smooth, compact manifold with boundary
\[
N_{irr} \ \cup \ \bigcup_{ \unA \in N_{red} } \big( SO(3)_{\unA} \cap V_{\Lambda} \cap V_{\cL} \big).
\]
Here $SO(3)_{\unA}$ is the image of $SO(3) \times \{ T_1 \}$ by $Gl_{\unA}$.
We write $S_{\unA}$ for the intersection $SO(3)_{\unA} \cap V_{\Lambda} \cap V_{\cL}$.
Then we obtain
\[
\# N_{irr} + \sum_{ \unA \in N_{red} } \# S_{\unA} \equiv 0 \mod 2.
\]
We must show
\begin{equation} \label{eq S_A}
\# S_{\unA} =
\left\{
\begin{array}{cl}
1 \mod 2 & \text{ if $n \equiv a + 1 \mod 2$}, \\
0 \mod 2 & \text{ otherwise }
\end{array}
\right.
\end{equation}
to get (\ref{eq torsion complex}). This follows from the following.

\begin{lemma} \label{lem SO(3)}
The restriction $\1L|_{SO(3)_{\unA}}$ is non-trivial, and the restriction $\aLam|_{SO(3)_{\unA} }$ is non-trivial if and only if $n \equiv a+1 \mod 2$.
\end{lemma}

\begin{proof}
Let $X_1, X_2$ be two copies of $Y \times \R$. And we write $X_1(T)$, $X_2(T)$ for $Y \times (-\infty, 2T)$, $Y \times (-2T, \infty)$. Gluing $X_1(T)$, $X_2(T)$ through the identification
\[
\begin{array}{ccc}
Y \times (T, 2T) & \stackrel{\cong}{\longrightarrow} & Y \times (-2T, -T) \\
(y,t) & \longmapsto & (y, t-3T),
\end{array}
\]
we get a manifold $X^{\#(T)}$. 

Let $E$ be the rank-two complex vector bundle over $Y$ associated with $Q$.
Take a trivialization $\phi$ of $E$ such that $\theta_n$ is trivial under $\phi$. 
We write $E_i$ for $\pi_i^*E$ and $\phi_i$ for $\pi^*_i \phi$. 
Here $\pi_i$ are the projections from $X_i(T)$ to $Y$. 
For $\tilde{u} \in SU(2)$ we define a vector bundle $E(\tilde{u})$ by
\[
E(\tilde{u}):= E_1 \bigcup_{ \phi_2^{-1} \circ (\tilde{u}  \phi_1) } E_2
\longrightarrow 
X^{\#(T)}.
\]

Gluing two instantons $([A_1], [A_2]) \in N_{red}$, we get an instanton $[A( \tilde{u})]$ on $E(\tilde{u})$, and we have the determinant line bundle $\{ \cL^{(1)}_{\Gamma, [A(\tilde{u})]} \}_{\tilde{u} \in SU(2)}$ over $SU(2)$. We have an action of $\{ \pm 1 \}$ on this bundle defined by the following diagram:
\[
\begin{CD}
( \cL^{(1)} )_{ [A(\tilde{u})] } @>{ \cong }>> 
\big( \C( \theta_{\gamma, 1}, \rho ) \otimes ( \cL)_{ [A_1] } \big) \otimes 
\big( ( \cL)_{ [A_2] } \otimes \C (\sigma, \theta_{ \gamma, 1 }) \big)  \\
@V{-1}VV        @V{ (-1)^{ m_1 + m_2} \otimes 1 }VV \\
( \cL^{(1)} )_{ [A(  -\tilde{u} )] } @>{ \cong }>> 
\big( \C( \theta_{\gamma, 1}, \rho ) \otimes ( \cL)_{ [A_1] } \big) \otimes 
\big( ( \cL)_{ [A_2] } \otimes \C (\sigma, \theta_{ \gamma, 1 }) \big).
\end{CD}  
\]
Here
\[
m_1 = \Ind^{-,+} \ind \bar{\partial}_{ A(\theta_{\gamma, 1}, \rho) }, \quad
m_2 = \Ind^{-,+} \bar{\partial}_{ A_1 }.
\]
The restriction $\1L|_{ SO(3)_{\unA} }$ is the quotient of $\{ \cL^{(1)}_{\Gamma, [A(\tilde{u})]} \}_{\tilde{u} \in SU(2)}$ by this action.  The necessary and sufficient condition for $\1L|_{ SO(3)_{\unA} }$ to be non-trivial is that $m_1 + m_2$ is odd. From the additivity of the index, we have
\[
m_1 + m_2 = \Ind^{-,+} \bar{\partial}_{ A(\theta_{\gamma, 1}, \theta_{n}|_{\gamma} ) }.
\]
Here $A(\theta_{\gamma, 1}, \theta_{n}|_{\gamma})$ is the connection on $\Gamma = \gamma \times \R$ with limits $\theta_{\gamma, 1}$, $\theta_{n}|_{\gamma}$.
Let $g_1$ be the fixed map from $Y$ to $SU(2)$ with degree 1. The class $[ g_1 |_{\gamma} ]$ represented by the restriction of $g_1$ to $\gamma$ is trivial in $\pi_1(U(2))$. Hence we have
\[
m_1 + m_2 
= \Ind^{-,+} \bar{\partial}_{ A(\theta_{\gamma, 1}, \theta_{\gamma, 0}) }
=-1
\]
by Lemma \ref{lem ind d-bar}. Therefore the restriction $\1L|_{ SO(3)_{\unA} }$ is non-trivial.

\vspace{2mm}

Similarly the necessary and sufficient condition for $\aLam|_{SO(3)_{\unA}}$ to be non-trivial is that
\[
\Ind^{-,+} \Dir_{ A(\theta_a, \theta_{n}) } \equiv 1 \mod 2.
\]
Lemma \ref{lem ind Dir} implies that this condition is equivalent to
\[
n \equiv a + 1 \mod 2.
\]

\end{proof}

\begin{definition}
$I^{(a)}_*(Y; \gamma) := H_*( C_*^{(a)}(Y), \partial^{(a)} )$.
\end{definition}

As in Subsection \ref{ss FFH well-defined}, we can show the following.

\begin{proposition}
$I^{(a)}(Y;\gamma)$ is independent of the metric on $Y$ and the sections of $\aLam$, $\1L$.
\end{proposition}

\begin{remark}
We assumed that $Y$ is a homology 3-sphere. In this case, it is sufficient to consider a very small loop $\gamma$ which represents the trivial class in $\pi_1(Y)$ or one point which is a degenerate loop, when we calculate invariants using our gluing formula. However we have considered a general loop $\gamma$ which may not be trivial in $\pi_1(Y)$. One of the reason is that even if we take a degenerate loop, we can see that the components of the boundary operator are not trivial, and we need to consider the components as in the case where the homotopy class of the loop is non-trivial. Another reason is that we hope our construction is extended to more general 3-manifolds.

\end{remark}

\subsection{Relative invariants}

Let $X$ be a closed, oriented, simply connected, non-spin $4$-manifold with $b^+(X)$ positive and even. 
We consider a situation where we have a decomposition $X= X_0 \cup_{Y} X_1$. Here $Y$ is an oriented homology 3-sphere and $X_0, X_1$ are compact, simply connected, non-spin 4-manifold with boundary $Y$, $\bar{Y}$ and $b^+(X_0), b^+(X_1)>1$. We consider invariants $\Psi_{X,P}^{u_1}(z)$, where $P$ is a $U(2)$-bundle over $X$ with $w_2(P) \equiv w_2(X)$ and $\dim M_P = 3$ and $z \in H_2(X;\Z)$ with $z \cdot z \equiv 0 \mod 2$. 
The purpose is to write $\Psi_{P}^{u_1}(z)$ in terms of data from $X_0, X_1$ as in the case of the usual Donaldson invariants.

We fix a trivialization $\varphi_0$ of $Q:=P|_{Y}$ and we write $\theta_0$ for the trivial connection. Choose a smooth map $g_1$ from $Y$ to $SU(2)$ of degree $1$. For $n \in \Z$, we denote $g_1^n(\theta_0)$ by $\theta_n$ as before.
For a sequence $T^{\alpha} \rightarrow \infty$, put $X^{\alpha}=X_0 \cup (Y \times [0, T^{\alpha}]) \cup X_1$.
We can easily show:

\begin{lemma}
Take $[A^{\alpha}] \in M_{X^{\alpha},P}$ and assume that $[A^{\alpha}]$ converges to $([A_0^{\infty}],\dots, [A_{r}^{\infty}])$. 
Here $[A_i^{\infty}] \in \i0M$.
Then we have either

\begin{itemize}

\item
$\rho(i)$ are irreducible for $i=1, \dots, r-1$, or

\vspace{2mm}

\item
$r=1$ and $\rho(1) = \theta_n$ for some $n \in \Z$, moreover $\dim M_{ \hat{X}_0}(\theta_n)= \dim M_{ \hat{X}_1 }(\theta_n) = 0$.

\end{itemize}

\end{lemma}

This lemma means that the construction of the gluing formula for $\Psi_{X}^{u_1}(z)$ is the same as that in Subsection \ref{ss gluing D inv} if $\dim M_{\hat{X}_0}(\theta_n) \not= 0$ for all $n \in \Z$. Hence we suppose the following.

\begin{hypothesis} \label{ass trivial flat}
There is an integer $n_0$ with $\dim M_{ \hat{X}_0 }(\theta_{n_0})=0$.
\end{hypothesis}

We will define relative invariants for $X_0$ under this assumption. 
Put $P_0 :=P|_{X_0}$ and let  $\gamma \cong S^1$ be a loop in $Y$ and $\Sigma_0$ be an oriented surface embedded in $X_0$ with boundary $\gamma$.
Using Dirac operators associated with a spin-c structure $c_0$ over $X_0$ with $c_1( \det c_0) = -c_1(P_0)$ and twisted $\bar{\partial}$ operators over $\hat{\Sigma}_0$, we have line bundles
\[
\tilde{\Lambda}(\rho) \stackrel{\R}{\longrightarrow} \tX0B, \quad
\tilde{\cL}_{ \hat{\Sigma} _0} \stackrel{\C}{\longrightarrow} \tX0B.
\]
For irreducible flat connections $\rho$ and $a \in \{ 0, 1 \}$ we put
\[
\tilde{\Lambda}^{(a)}:= \tilde{\Lambda}(\rho) \otimes \underline{\R}(\rho, \theta_a), \quad
\tilde{\cL}_{ \hat{\Sigma}_0  }^{(1)}(\rho) :=
\tilde{ \cL }_{ \hat{\Sigma}_0 }(\rho) \otimes \underline{\C}(\rho, \theta_{ \gamma, 1 }).
\]
Here $\theta_{\gamma, 1}$ is defined as the previous subsection.
Then $\Gamma_{\rho}=\Z_2$ naturally acts on these bundles.

\begin{lemma} \label{lem line descend}
Let $\varphi_0$ be the fixed trivialization of $Q$ and put $\varphi_{a}:=g_1^{a} \varphi_0$. When $c_2(P_0, \varphi_a) \equiv 0 \mod 2$, $\tilde{\Lambda}^{(a)}(\rho)$ descends to $\X0B$.
Let $z_0 \in H_2(X_0,Y;\Z) \cong H_2(X_0;\Z)$ be the class represented by $\Sigma_0$. If $z_0 \cdot z_0 \equiv 1 \mod 2$,
the line bundle $\tilde{ \cL }^{(1)}_{ \hat{\Sigma}_0 }(\rho)$ also descends from $\tX0B$ to $\X0B$.
\end{lemma}

\begin{proof}
The wight of the $\Gamma_{\rho}$ action on $\tilde{\Lambda}^{(a)}(\rho)$ is $\Ind^+ \Dir_{A(\theta_{a})}$. Hence it is sufficient to show that the index is even if $c_2(P_0; \varphi_a)$ is even.
Choose a compact, spin 4-manifold $X_0'$ with boundary $\bar{Y}$ and we consider  a manifold $X'=X_0 \cup_{Y} X_0'$. Let $P'$ be a $U(2)$-bundle over $X'$ obtained by gluing $P_0$ and the trivial bundle $P_{X_0'}$ over $X_0'$ through $\varphi_a$.
Let $c_0'$ be a spin-c structure over $X_0'$ induced by a spin structure. Then we have a spin-c structure $c'$ on $X'$ induced by $c_0, c_0'$ and we have the Dirac operator
\[
\Dir_{A'}:\Gamma (S^+_{c'} \otimes E') \longrightarrow \Gamma(S^-_{c'} \otimes E').
\]
Here $A'$ is a connection induced by $A(\theta_a)$ and the trivial connection $\theta_{X_0'}$ on $P_{X_0'}$ and $E'$ is the complex vector bundle of rank two associated with $P'$. Since $\Dir_{\theta_{X_0'} }$ is the direct sum of two copies of a Dirac operator, the index $\Ind^- \cD_{ \theta_{X_0'} }$ is even. Thus we get
\[
\begin{split}
\Ind^+ \Dir_{ A(\theta_a) }
&\equiv \Ind^+ \Dir_{ A(\theta_a) } + \Ind^- \Dir_{ \theta_{X_0'} } \mod 2  \\
&\equiv \Ind \Dir_{A'} \mod 2.
\end{split}
\]
Applying the index formula, we have
\[
\begin{split}
\Ind \Dir_{A'} 
&= ch(E') \hat{A}(X') e^{ \frac{1}{2}c_1(\det c') }/[X'] \\
&=2 \hat{A}(X')e^{ \frac{1}{2} c_1(\det c') }/[X'] +
\frac{1}{2} c_1(E')(c_1(\det c') + c_1(E')) /[X'] - c_2(E')/[X'].
\end{split}
\]
Since $\hat{A}(X')e^{ \frac{1}{2} c_1(\det c') }/[X']$ is the index of a spin-c Dirac operator, it is an integer. Hence the first term is even.
From  $c_1(\det c') = -c_1(E')$, the second term is zero. Therefore we obtain
\[
\Ind^+ \Dir_{ A(\theta_a) } 
\equiv -c_2(E')/[X']
\equiv c_2(P_0; \varphi_a)
\mod 2.
\]
Thus the first part of Lemma follows.

To show the second part, we need to show
\[
\Ind^+ \bar{\partial}_{ A( \theta_{\gamma, 1} ) } \equiv 0 \mod 2.
\]
Here $A( \theta_{\gamma, 1} )$ is a connection on $\hat{X}_0$ with limit $\theta_{\gamma, 1}$.
The additivity of index implies that
\[
\Ind^+ \bar{\partial}_{ A( \theta_{\gamma, 1} ) } =
\Ind^+ \bar{\partial}_{ A( \theta_{\gamma, 0} ) } +
\Ind^{-+} \bar{\partial}_{ A( \theta_{\gamma, 0}, \theta_{\gamma, 1} ) }.
\]
Applying Lemma \ref{lem ind d-bar} to the second term in the right hand side, we get
\[
\Ind^+ \bar{\partial}_{ A( \theta_{\gamma, 1} ) } = 
\Ind^+ \bar{\partial}_{ A( \theta_{\gamma, 0} ) } + 1.
\]
Let $\varphi_{\gamma}$ be the restriction of $\varphi_0$ to $\gamma$ and take an oriented, compact surface $\Sigma_0'$ with boundary $\bar{\gamma}$. Gluing $P_0|_{\Sigma_0}$ and the trivial $U(2)$-bundle $P_{\Sigma_0'}$ over $\Sigma_0'$ through $\varphi_{\gamma}$, we get a $U(2)$-bundle $P_{\Sigma'}$ over $\Sigma'=\Sigma_0 \cup_{\gamma} \Sigma_0'$. Let $\theta'$ be the trivial connection on $P_0'$. The index $\Ind^- \bar{\partial}_{\theta'}$ is even, since $\bar{\partial}_{\theta'}$ is the direct sum of two copies of the $\bar{\partial}$-operator. Thus we have
\[
\begin{split}
\Ind^+ \bar{\partial}_{ A( \theta_{\gamma,0} ) }
& \equiv \Ind^+ \bar{\partial}_{ A( \theta_{\gamma,0} ) } + \Ind^- \bar{\partial}_{\theta'} \mod 2 \\
& \equiv \Ind \bar{\partial}_{A'} \mod 2.
\end{split}
\]
Here $A'$ is the connection over $\Sigma'$ induced by $A( \theta_{\gamma,0} )$ and $\theta'$. From the fact that $c_1(P_0) \equiv w_2(X_0) \mod 2$ and $z_0 \cdot z_0 \equiv 1 \mod 2$, we have
\[
\begin{split}
\Ind \bar{\partial}_{A'} 
&=  ch(P_{\Sigma'})\hat{A}(\Sigma')/[\Sigma'] \\
&= (2 + c_1(P_{\Sigma'}))/[\Sigma'] \\
&= < c_1(P_{\Sigma'}), [\Sigma']> \\
&= < c_1(P_0), z_0> \\
&\equiv 1 \mod 2.
\end{split}
\]
Therefore we obtain
\[
\Ind^+ \bar{\partial}_{ A( \theta_{\gamma, 1} ) } \equiv 0 \mod 2.
\]

\end{proof}

\begin{definition}
Let $a \in \{ 0, 1 \}$ with $c_2(P_0;\varphi_a) \equiv 0 \mod 2$.
For irreducible flat connections $\rho$, we define 
\[
\a0Lam := \tilde{\Lambda}^{(a)}(\rho)/\Gamma_{\rho}  \stackrel{\R}{\longrightarrow} \X0B,
\quad
\S0L:= \tilde{\cL}_{ \hat{\Sigma}_0 }^{(1)}/\Gamma_{\rho} \stackrel{\C}{\longrightarrow} \X0B.
\]
\end{definition}

The following lemma follows from a dimension counting argument like the proof of Lemma \ref{lem rho irr}.

\begin{lemma}
Let $\rho$ be a flat connection on $Q$ with $\dim M_{\hat{X}_0}(\rho) < 4$ and $\{ [A^{\alpha}] \}_{\alpha}$ be a sequence in $M_{\hat{X}_0}(\rho)$ which converges to some $([A_1^{\infty}], \dots, [A_r^{\infty}])$. Here $[A_i^{\infty}] \in \i0M$ and $\rho(r) = \rho$. Then $\rho(i)$ are irreducible for $i=0,\dots, r-1$.
\end{lemma}

This lemma implies that for $\rho$ with $\dim \MX0 < 4$ we can take sections $s_{\Lambda}(\rho):\MX0 \rightarrow \a0Lam$, $s_{\cL}(\rho):\MX0 \rightarrow \S0L$ having properties like those in Proposition \ref{prop section}.
Using admissible sections, we define $\psi_{X_0}^{u_1} \in C_{8n_0}^{(a)}(Y)$ as follows. Let $\beta$ be an integer with $0 \leq \beta \leq 3$ and take a generator $[ \rho ] \in CF_{8n_0 - \beta}(Y)$. Then put
\[
< \psi_{X_0}^{u_1}, [\rho] >:= \# M_{\hat{X}_0}(\rho;L_{\beta}) \mod 2,
\]
and we define
\[
< \psi_{X_0}^{u_1}, \theta_{n_0} > :=
\left\{
\begin{array}{cl}
M_{ \hat{X}_0}(\theta_0) \mod 2 & \text{if $n_0 \equiv a + 1 \mod 2$},  \\
0 & \text{otherwise.}
\end{array}
\right.
\]
These numbers define a chain $\psi_{X_0}^{u_1} \in C^{(a)}_{ 8 n_0 }(Y)$ as usual. As before, we have:

\begin{proposition}
The chain $\psi_{X_0}^{u_1}$ is a cycle. Moreover the class $\Psi_{X_0}^{u_1} = [\psi_{X_0}^{u_1}] \in I_{8n_0}^{(a)}(Y; \gamma)$ is independent of the metric on $X_0$ and the admissible sections of $\a0Lam$, $\S0L$.
\end{proposition}

\begin{remark}
As the case of relative Donaldson invariants $\Psi_{X_0}$, if an analogy of Conjecture\ref{conj invariance} holds then relative 2-torsion instanton invariants $\Psi_{X_0}^{u_1}$ is independent of the surface $\Sigma_0$ representing the homology class $z_0$.
\end{remark}

\subsection{Gluing formula}

We begin with the definition of a pairing $I_{*}^{(a)}(Y) \otimes I_{*}^{(a)}(\bar{Y}) \rightarrow \Z_2$.
Since $\delta_{ \bar{Y} } ([\rho]) \equiv -\delta_{Y}([\rho]) - 3 \mod 8$, we have a natural pairing $CF_{j}(Y) \otimes CF_{-j-3}( \bar{Y} ) \rightarrow \Z_2$.
For $[ \theta_n] \in C^{(a)}_{8n}(Y)$ and $[ \theta_{-n} ] \in C^{(a)}_{-8n}( \bar{Y})$, we define $< \theta_n, \theta_{-n} > \equiv 1 \mod 2$. Then we get a pairing
\[
< \ , \ > : C_j^{(a)}(Y) \otimes C_{-j}^{(a)}( \bar{Y} ) \longrightarrow \Z_2.
\]
We can easily show

\begin{lemma}
For $[\rho] \in C^{(a)}_{j}(Y)$ and $[\sigma] \in CF_{-j+1}^{(a)}( \bar{Y} )$,
\[
< \partial^{(a)} ( [\rho] ), \sigma > = < [\rho], \partial^{(a)} ([\sigma]) >.
\]
\end{lemma}

This lemma means that the pairing induces a pairing
\[
< \ , \ > : I^{(a)}_{j}(Y;\gamma) \otimes I^{(a)}_{-j}( \bar{Y};\gamma ) \longrightarrow \Z_2.
\]

The main statement in this section is the following.

\begin{theorem} \label{thm main}
Let $X$ be a closed, oriented, simply connected 4-manifold with $b^+$ positive and even. Assume that $X$ has a decomposition $X=X_0 \cup_{Y} X_1$, where $Y$ is a homology 3-sphere and $X_0, X_1$ are compact, non-spin, simply connected 4-manifolds with boundary $Y$, $\bar{Y}$.  Take a loop $\gamma$ in $Y$ and compact, oriented surfaces $\Sigma_0$, $\Sigma_1$ in $X_0$, $X_1$ with boundary $\gamma$. If the self-intersection numbers of $z_i := [\Sigma_i] \in H_2(X_i,Y;\Z) = H_2(X_i;\Z)$ are odd, then we have
\[
\Psi_X^{u_1}(z) = < \Psi_{X_0}^{u_1}(\Sigma_0), \Psi_{X_1}^{u_1}(\Sigma_1) >,
\]
where $z = z_0 + z_1 \in H_2(X;\Z)$.

\end{theorem}

We can generalize this formula to the case when the dimension of $M_{X,P}$ is more than 3. Assume that the dimension of $M_{X,P}$ is $2(d_0 + d_1) + 3$ for some $d_0, d_1 \in \Z_{ \geq 0}$.
For $x_1, \dots, x_{d_0} \in H_2(X_0;\Z)$, $y_1, \dots, y_{d_1} \in H_2(X_1;\Z)$ with $x_i \cdot x_i \equiv y_j \cdot y_j \equiv 0 \mod 2$, we can define relative invariants
\[
\Psi_{X_0}^{u_1}( x_1,\dots, x_{d_0}, \Sigma_0 ) \in I_*^{(a)}(Y;\gamma), \quad
\Psi_{X_1}^{u_1}( y_1, \dots, y_{d_1}, \Sigma_1 ) \in I_*^{(a)}( \bar{Y};\gamma).
\]
Here $\Sigma_i$ are surfaces in $X_i$ as in Theorem \ref{thm main}.
Then we have
\[
\Psi_{X}^{u_1}(x_1,\dots, x_{d_1}, y_1,\dots, y_{d_1}, z) =
< \Psi_{X_0}^{u_1}(x_1,\dots, x_{d_0}, \Sigma_0), \Psi_{X_1}^{u_1}(y_1, \dots, y_{d_1}, \Sigma_1) >.
\]

Applying this formula to the case $Y=S^3$, we immediately obtain the following.

\begin{corollary} \label{coro connected sum}
Let $X_0, X_1$ be closed, oriented, non-spin, simply connected 4-manifolds with $b^+$ positive and odd. Let $P_i$ be $U(2)$-bundles over $X_i$ with $w_2(P_i) = w_2(X_i)$, $c_2(P_i) \equiv 1 \mod 2$ and $\dim M_{P_i} = 2d_i$ for some $d_i \geq 0$.
For homology classes $x_1,\dots, x_{d_0}, z_0 \in H_2(X_0;\Z)$, $y_1, \dots, y_{d_1}, z_1 \in H_2(X_1;\Z)$ with $x_i \cdot x_i \equiv y_j \cdot y_j \equiv 0 \mod 2$, $z_0 \cdot z_0 \equiv z_1 \cdot z_1 \equiv 1 \mod 2$, we have
\[
\Psi_{X}^{u_1}(x_1, \dots, x_{d_0}, y_1, \dots, y_{d_1}, z) \equiv
\Psi_{X_0, P_0} (x_1, \dots, x_{d_0}) \cdot 
\Psi_{X_1, P_1}(y_1, \dots, y_{d_1}) \mod 2.
\]
Here $z = z_0 + z_1 \in H_2(X;\Z)$.
\end{corollary}

\begin{remark}
In general, usual Donaldson invariants $\Psi_{P}(x_1, \dots, x_d)$ do not lie in $\Z$ even if $x_i \in H_2(X;\Z)$.
However if $< w_2(P), x_i > \equiv 0 \mod 2$ for all $i$ then the value $\Psi_{P}(x_1,\dots, x_d)$ lies in $\Z$.
Therefore the above formula makes sense.
\end{remark}

Befoer we show Theorem \ref{thm main}, we prove non-vanising results using Corollary \ref{coro connected sum}.

\begin{theorem} \label{thm 2CP2}
For $i=0, 1$, let $\CP^2_i$ and $\barCP2_i$ be copies of $\CP^2$ and $\barCP2$  respectively. Let $H_i \in H_2(\CP^2_i;\Z)$, $E_i \in H_2(\barCP2;\Z)$ be the standard generators. Then we have
\[
\begin{split}
\Psi_{\CP^2_0 \# \CP^2_1}^{u_1}(H_0 + H_1) &\equiv 1 \mod 2, \\
\Psi_{\CP^2_0 \# \CP^2_1 \# \barCP2_0}^{u_1}(H_0 + E_0, H_0 + H_1) &\equiv 1 \mod 2, \\
\Psi_{\CP^2_0 \# \CP^2_1 \# \barCP2_0 \# \barCP2_1}^{u_1}(H_0 + E_0, H_1 + E_1, H_0 + H_1) &\equiv 1 \mod 2.
\end{split}
\]
\end{theorem}

In \cite[Proposition 7.1]{K}, Kotschick showed $\Psi_{\CP^2, P} = -1$ for the $U(2)$-bundle $P$ with $c_1 = H$, $c_2 = 1$. (In this case, the dimension of the moduli space is zero.) Hence the first equality in Theorem \ref{thm 2CP2} follows from Corollary \ref{coro connected sum}. We also have $\Psi_{\CP^2 \# \barCP2, P}(H+E) = 1$ for the $U(2)$-bundle $P$ with $c_1 = H-E$, $c_2 = 1$ by \cite[Proposition 7.1]{K}. (Note that in \cite{K}, $2\mu([\Sigma])$ is used to define the invariant.) Hence we obtain the second and third equations from Corollary \ref{coro connected sum}.

\vspace{3mm}

We begin the proof of Theorem \ref{thm main}.
As mentioned before, if there are no integers $n$ such that $\dim M_{ \hat{X}_0 }(\theta_n) = 0$ then the proof of Theorem \ref{thm main} is quite the same as that of Theorem \ref{thm gluing D inv}. Hence we suppose Hypothesis \ref{ass trivial flat}.
To prove Theorem \ref{thm main}, we fix a sequence $T^{\alpha} \rightarrow \infty$ and consider manifolds $X^{\alpha} = X_0 \cup ( Y \times [0, T^{\alpha}] ) \cup X_1$ as usual. As in Subsection \ref{ss sections}, we can take sections $s^{\alpha}_{\Lambda}:M_{X^{\alpha}} \rightarrow \Lambda$, $s^{\alpha}_{\cL}:M_{X^{\alpha} } \rightarrow \cL_{\Sigma}$ with the following property.
Suppose that  a sequence $[A^{\alpha}] \in M_{X^{\alpha}}$ converges to $([A_0^{\infty}], \dots, [A_{r}^{\infty}])$. Here $[ A_0^{\infty} ] \in M_{ \hat{X}_0}(\rho(0))$, $[A_i^{\infty}] \in \i0M$ ($i=1, \dots, r-1$), $[A_r^{\infty}] \in M_{ \hat{X}_1 }(\rho(r-1))$. Moreover suppose that $\rho(i)$ are irreducible for $i=0, \dots, r-1$. Then we have
\[
\begin{split}
s^{\alpha}_{\Lambda}([A^{\alpha}]) &\longrightarrow 
s_{\Lambda}([A_0^{\infty}]) \boxtimes \cdots \boxtimes s_{\Lambda}([A_r^{\infty}]), \\
s^{\alpha}_{\cL}([A^{\alpha}]) &\longrightarrow 
s_{\cL}([A_0^{\infty}]) \boxtimes \cdots \boxtimes s_{\cL}([A_r^{\infty}]).
\end{split}
\]
A standard dimension-counting argument shows that for any sequence $[A^{\alpha}] \in M_{X^{\alpha}}(L_3) = M_{X^{\alpha}} \cap V_{\Lambda} \cap V_{\cL}$, there is a subsequence $\{ [A^{\alpha'}] \}_{\alpha'}$ such that $[A^{\alpha'}]$ converges to $([A_0^{\infty}], [A_1^{\infty}])$, where $[A_0^{\infty}] \in M_{ \hat{X}_0 }(\rho; L_{\beta})$, $[A_1^{\infty}] \in M_{ \hat{X_1} }(\rho; L_3 \backslash L_{\beta})$ for some $0 \leq \beta \leq 3$ and $[\rho] \in CF_{8n_0 - \beta}(Y)$, or $[A_0^{\infty}] \in M_{\hat{X}_0}(\theta_{n_0})$, $[A_1^{\infty}] \in M_{\hat{X}_1}(\theta_{n_0})$. Hence for large $\alpha$ we have
\[
M_{X^{\alpha}}(L_3) \cong 
\bigcup_{\beta} \ \bigcup_{[\rho]} \ 
M_{ \hat{X}_0 }(\rho; L_{\beta}) \times M_{\hat{X}_1}(\rho;L_3 \backslash L_{\beta})
\cup
\bigcup_{ 
\begin{subarray}{c}
\unA \ \in \\
M_{\hat{X}_0}(\theta_{n_0}) \times M_{ \hat{X}_1 }(\theta_{n_0})
\end{subarray}
}
SO(3)_{\unA} \cap V_{\Lambda} \cap V_{\cL}.
\]
Here $SO(3)_{\unA}$ is the image of $SO(3)$ by a gluing map as in the proof of Lemma \ref{lem torsion complex}.
Therefore we get
\[
\Psi_{X}^{u_1}(z) = 
\sum_{\beta} \sum_{[\rho]} 
\# M_{ \hat{X}_0 }(\rho;L_{\beta}) \cdot \# M_{ \hat{X}_1 } (\rho;L_3 \backslash L_{\beta})
+
\sum_{ \unA }  \# \big( SO(3)_{ \unA } \cap V_{\Lambda} \cap V_{\cL} \big).
\]
Thus we must show
\[
\# (SO(3)_{ \unA } \cap V_{\Lambda} \cap V_{\cL}) =
\left\{
\begin{array}{cl}
1 & \text{ if $n_0 = a + 1 \mod 2$, } \\
0 & \text{ otherwise. }
\end{array}
\right.
\]
This follows from the following lemma.

\begin{lemma}
For each $\unA = ([A_0], [A_1])$, the restriction $\cL_{\Sigma}|_{SO(3)_{ \unA }}$ is non-trivial.
And the restriction $\Lambda|_{SO(3)_{\unA}}$ is non-trivial if and only if
$n_0 \equiv a + 1 \mod 2$.
\end{lemma}

\begin{proof}
As in the proof of Lemma \ref{lem SO(3)}, we can see that a necessary and sufficient condition for the restrictions of the line bundles to be non-trivial is that the indexes $\Ind^+ \bar{\partial}_{A_0}$, $\Ind^+ \Dir_{A_0}$ are odd. The calculations in the proof of Lemma \ref{lem line descend} show that
\[
\begin{split}
& \Ind^+ \bar{\partial}_{A_0} \equiv z_0 \cdot z_0 \equiv 1 \mod 2 \\
& \Ind^+ \Dir_{A_0} \equiv n_0 - a \mod 2.
\end{split}
\]
Therefore we obtain the statements.
\end{proof}

\section{A non-existence result}
\label{s Thm B}

Let $X$ be a closed, oriented, simply connected 4-manifold and $Q_X$ be the intersection form of $X$ on $H_2(X;\Z)$.
For positive integers $k$ and $[\Sigma_1], \dots, [\Sigma_{2k}] \in H_2(X;\Z)$, we define
\[
Q_{X}^{(k)}([\Sigma_1], \dots, [\Sigma_{2k}]) :=
\frac{1}{2^k k!} \sum_{ \sigma \in \frak{S}_{2k} }
Q_{X}( [\Sigma_{ \sigma(1) }], [\Sigma_{\sigma(2)}]) \cdots 
Q_{X}( [\Sigma_{ \sigma( 2k- 1 ) }], [\Sigma_{\sigma( 2k )}]).
\]
In \cite{conn}, Donaldson showed that when $X$ is spin and $b^+ = 1$,  $Q_{X}$ must satisfy the condition
\begin{equation} \label{eq intersection}
Q_{X}^{(2)}([\Sigma_1], \dots, [\Sigma_4]) \equiv 0 \mod 2
\end{equation}
for all $[\Sigma_1], \dots, [\Sigma_4] \in H_2(X;\Z)$.

In this section, we will consider the intersection forms of compact 4-manifolds with boundaries homology 3-spheres.
We prove that the condition also holds for compact, spin 4-manifold $X$ with $b^+ = 1$ and with boundary a homology 3-sphere $Y$, if $HF_1(Y;\Z_2)$ and $HF_2(Y;\Z_2)$ are trivial.
Here $HF_*(Y;\Z_2)$ are the usual Floer homology groups with coefficients $\Z_2$.

\begin{theorem} \label{thm B}
Let $Y$ be an oriented homology 3-sphere with $HF_1(Y;\Z_2)=0$, $HF_2(Y;\Z_2)=0$. If $Y$ bounds a compact, spin, simply connected 4-manifold $X$ with $b^+(X)=1$, then we have
\[
Q^{(2)}_X ([\Sigma_1], \dots, [\Sigma_4]) \equiv 0 \mod 2
\]
for any $[\Sigma_1], \dots, [\Sigma_4] \in H_2(X;\Z)$.

\end{theorem}

This is originally due to Fukaya, Furuta and Ohta. (See \cite{Fur}.)
Making use of techniques developed in this paper, we write down the proof of this theorem.

\subsection{Relative invariants for spin 4-manifolds}

In the previous section, we defined relative invariants for non-spin 4-manifolds with boundary. Here we define relative 2-torsion instanton invariants for spin 4-manifolds, which are used to prove Theorem \ref{thm B}.
To define the invariants, we modify the construction in the previous section.

Let $Q=Y \times SU(2)$ be the trivial $SU(2)$-bundle over $Y$ and fix $a \in \{ 0, 1 \}$.
For irreducible flat connections $\rho, \sigma$ on $Q$ with $\dim \0M < 2$, we can take sections $s_{\Lambda}:\0M \rightarrow \aLam$ which are compatible with gluing maps and satisfy the transversality conditions.

\begin{definition}
Put $C_j(Y):=CF_{j}(Y) \oplus CF_{j-1}(Y)$. We define $\partial^{(a)}:C_j(Y) \rightarrow C_{j-1}(Y)$ as follows.
Let $\beta_1, \beta_2$ be integers with $0 \leq \beta_1 \leq \beta_2 \leq 1$ and take generators $[\rho] \in CF_{j-\beta_0}$, $[\sigma] \in CF_{j-\beta_1-1}(Y)$. If $\beta_2 - \beta_1 = 0$,  put
\[
< \partial^{(a)} ([\rho]), [\sigma] > := \# \0M \mod 2.
\]
If $\beta_2 - \beta_1 = 1$, put
\[
< \partial^{(a)} ([ \rho ]), [\sigma] > := \# \big( \0M \cap V_{\Lambda} \big) \mod 2.
\]
We define $\partial^{(a)}$ by using these numbers as usual.
\end{definition}

Let $\{ [A^{\alpha}] \}_{\alpha}$ be a sequence in $\0M$ for some flat connections $\rho, \sigma$ with $\dim \0M < 2$. We can show that if $[A^{\alpha}]$ converges to some $([A_1^{\infty}], \dots, [A_r^{\infty}]) \in M_{Y \times \R}^0(\rho, \rho(1)) \times \cdots \times M_{Y \times \R}^0(\rho(r-1), \sigma)$ then $\rho(i)$ are irreducible for $i=1, \dots, r-1$. This means that we can prove
\[
\partial^{(a)} \circ \partial^{(a)} = 0
\]
as in Section \ref{section Fukaya-Floer}.

\begin{definition}
$I_*^{(a)}(Y):= H_*( C_*(Y), \partial^{(a)})$.
\end{definition}

Consider a compact, simply connected, spin 4-manifold $X$ with boundary $Y$ and $b^+(X)=1$. 
We will define a relative invariant
\[
\Psi_{X}^{u_1}:H_2(X;\Z)^{\otimes 4} \longrightarrow I_{2}^{(0)}(Y).
\]
Take a $SU(2)$-bundle $P$ over $X$ and a trivialization $\varphi$ of $P|_{Y}$ with
\[
c_2(P; \varphi) = 2 \in H^4(X, Y; \Z) \cong \Z.
\]
We denote by $\theta$ the trivial connection on $P|_{Y}$ with respect to $\varphi$.
The dimension of the moduli space $M_{ \hat{X} }(\theta)$ of instantons on $\hat{X}$ with limit $\theta$ is 
\[
8c_2(P; \varphi) - 3 (1 + b^+(X)) = 8 \cdot 2 - 3 ( 1 + 1 )= 10.
\]
For generators $[\rho] \in CF_{\beta}(Y)$, we have moduli spaces $M_{\hat{X}}(\rho)$ of dimension $10 - \beta$.

Let $\Sigma$ be a closed, oriented surface embedded in $X$. We put $\cB_{\Sigma,+}^{*} = \cB_{\Sigma}^* \cup \{ [\theta_{\Sigma}] \}$. It is known that the determinant line bundle $\cL_{\Sigma} \rightarrow \cB_{\Sigma}^*$ extends to $\cB_{\Sigma, +}^*$. (See \cite{DK}.) We use the same notation $\cL_{\Sigma}$ for the extension. We can take a section $s_{\Sigma}:\cB_{\Sigma,+}^* \rightarrow \cL_{\Sigma}$ such that $s_{\Sigma}$ does not vanish near $[\theta_{\Sigma}]$. We always assume that sections of $\cL_{\Sigma}$ have this property. 
Let $V_{\Sigma}$ be the zero locus of $s_{\Sigma}$ and we define $M_{ \hat{X} }(\rho) \cap V_{\Sigma}$ to be
\[
\{ \ [A] \in M_{ \hat{X} }(\rho) \ | \ s_{\Sigma} ([A |_{\Sigma}]) = 0 \  \}.
\]

Take homology classes $[\Sigma_1], \dots, [\Sigma_4] \in H_2(X;\Z)$. For generic surfaces $\Sigma_1, \dots, \Sigma_4$, the intersection of any two surfaces is transverse and the intersection of any three surfaces is empty. Moreover for generic sections $s_{\Sigma_1}, \dots, s_{\Sigma_4}$, all generators $[ \rho ] \in CF_{\beta}(Y)$ and all subsets $\{ l_1, \dots, l_{s} \} \subset \{ 1, \dots, 4 \}$, the intersections
\[
M_{ \hat{X} }(\rho;\Sigma_{l_1}, \dots, \Sigma_{l_s } ) = 
M_{ \hat{X} }( \rho ) \cap \bigcap_{ m=1}^s V_{\Sigma_{l_m}}
\]
are transverse.

A standard dimension-counting argument shows the following:

\begin{lemma}
Let $\rho$ be a flat connection on $Q$.
Assume that a sequence $\{ [A^{\alpha} ] \}_{\alpha}$ in $M_{\hat{X}}(\rho; \Sigma_1, \dots, \Sigma_4)$ converges to $(([A_0^{\infty}], Z_1), \dots, ([A_r^{\infty}], Z_r) )$. Here $( [A_0^{\infty}], Z_0) \in M_{ \hat{X} }(\rho(0)) \times \Sym^{n_0}( \hat{X})$ and  $([A_i^{\infty}], Z_i) \in \big( \iM \times \Sym^{n_i}(Y \times \R) \big)/\R$ for $i=1,\dots, r$.

\begin{enumerate}

\item
If $[\rho] \in CF_{2}(Y)$ and $\dim M_{\hat{X}}(\rho)=8$ then $r=0$, $Z_0=\emptyset$.

\vspace{2mm}

\item
If $[\rho] \in CF_1(Y)$ and $\dim M_{\hat{X}}(\rho)=9$ then $r \leq 1$ and all $Z_i$ are empty. Moreover when $r=1$, $\rho(0)$ is irreducible.

\vspace{2mm}

\item
If $\rho=\theta$, then one of the following holds:

\begin{itemize}

\item
$r=0$, $Z_0 = \emptyset$.

\item
$r=0$ and $Z_0=\{ x_1, x_2 \}$ for some $x_1 \in \Sigma_{l_1} \cap \Sigma_{l_2}$, $x_2 \in \Sigma_{l_3} \cap \Sigma_{l_4}$ with $l_1, \dots, l_4$ distinct.

\item
$r=1$, $\rho(0)$ is irreducible and all $Z_i$ are empty.

\item
$r=2$, $\rho(0), \rho(1)$ are irreducible and all $Z_i$ are empty.

\end{itemize}

\end{enumerate}

\end{lemma}

It follows from the lemma that we may take admissible sections $s_{\Lambda}(\rho):M_{ \hat{X} }(\rho;\Sigma_1, \dots, \Sigma_4) \rightarrow \Lambda(\rho)$ for $[\rho] \in CF_{1}(Y)$.
We define a chain $\psi_{X}^{u_1}=\psi_{X}^{u_1}(\Sigma_1, \dots, \Sigma_4) \in C_2^{(0)}(Y)$ as follows.
For $[\rho] \in CF_{2}(Y)$ put
\[
< \psi_{X}^{u_1}, [\rho] > := 
\# \big( M_{ \hat{X} }(\rho;\Sigma_1, \dots, \Sigma_4) \cap V_{\Lambda} \big) \mod 2.
\]
For $[\rho] \in CF_{1}(Y)$, put
\[
< \psi_{X}^{u_1}, [\rho] > :=
\# M_{ \hat{X} }(\rho;\Sigma_1, \dots, \Sigma_4) \mod 2.
\]
These define a chain $\psi_{X}^{u_1} \in C_2^{(0)}(Y)$ as usual.

\begin{lemma}
The chain $\psi_{X}^{u_1} \in C_2^{(0)}(Y)$ is a cycle. Moreover the class $\Psi_{X}^{u_1}([\Sigma_1], \dots, [\Sigma_4]) = [\psi_{X}^{u_1}] \in I^{(0)}_2(Y)$ is independent of the metric and the sections and depends only on the homology classes $[\Sigma_1], \dots, [\Sigma_4] \in H_2(X;\Z)$.
\end{lemma}

\subsection{Proof of Theorem \ref{thm B}}

We start with the following lemma.

\begin{lemma} \label{lem Floer vanish}
If both of $HF_{j-1}(Y;\Z_2)$ and $HF_j(Y;\Z_2)$ are trivial, then $I_j^{(a)}(Y)$ is also trivial.
\end{lemma}

We can easily show this by considering the long exact sequence associated with a short exact sequence
\[
0 \longrightarrow CF_{*-1}(Y;\Z_2) \longrightarrow C_*^{(a)}(Y) \longrightarrow CF_{*}(Y;\Z_2) \longrightarrow 0.
\]

\vspace{2mm}

Next we define a map
\[
D^{(a)}:I_2^{(a)}(Y) \longrightarrow \Z_2.
\]
For $[\rho] \in CF_{1}(Y)$, put
\[
D^{(a)}([\rho]) := \# M_{ Y \times \R}^0(\rho, \theta) \mod 2
\]
and for $[\rho] \in CF_{2}(Y)$, put
\[
D^{(a)}( [\rho] ) := 
\# \big( M_{Y \times \R}^0( \rho, \theta) \cap V_{\Lambda}  \big) \mod 2.
\]
Then we have a linear map $D^{(a)}$ from $C^{(a)}_2(Y)$ to $\Z_2$.
Counting the number of the ends of appropriate 1-dimensional moduli spaces, we get
\[
D^{(a)} \circ \partial^{(a)} = 0.
\]
Hence we obtain
\[
D^{(a)}:I_2^{(a)}(Y) \longrightarrow \Z_2.
\]
(We denote the map by the same notation.)

Theorem \ref{thm B} immediately follows from Lemma \ref{lem Floer vanish} and the following proposition.

\begin{proposition}
The image of $\Psi_{X}^{u_1}([\Sigma_1], \dots, [\Sigma_4]) \in I_{2}^{(0)}(Y)$ by $D^{(0)}$ is
\[
Q_{X}^{(2)}([\Sigma_1], \dots, [\Sigma_4]) \mod 2.
\]
\end{proposition}

To prove this, we consider the ends of the cut-down moduli space $M_{ \hat{X} }(\theta;\Sigma_1, \dots, \Sigma_4) \cap V_{\Lambda}$ with dimension 1.
Let $\{ [A^{\alpha}] \}_{\alpha}$ be a sequence in the cut-down moduli space. Then there is a subsequence $\{ [A^{\alpha'}] \}_{\alpha'}$ which goes to $\big( ([A_0], Z_0), \dots,  ([A_r^{\infty}], Z_r) \big)$. 
We can see that we are in one of the following cases:

\begin{enumerate}

\item
$r=0$, $Z_0 = \emptyset$.

\vspace{2mm}

\item \label{theta}
$r=0$, 
$Z_0=\{ x_1, x_2 \}$ 
for some $x_1 \in \Sigma_{l_1} \cap \Sigma_{l_2}$, $x_2 \in \Sigma_{l_3} \cap \Sigma_{l_4}$ with $l_1, \dots, l_4$ distinct and
$[A_0^{\infty}] = [\theta_{ \hat{X} }] $.

\vspace{2mm}

\item \label{deg 2}
$r=1$, all $Z_i$ are empty, and $[\rho(0)] \in CF_{2}(Y)$, 
$[A_0^{\infty}] \in M_{ \hat{X} }(\rho(0); \Sigma_1, \dots , \Sigma_4)$, $[A_1^{\infty}] \in M^0_{ Y \times \R }(\rho(0), \theta) \cap V_{\Lambda}$.

\vspace{2mm}

\item \label{deg 1}
$r=1$, $[\rho (0) ] \in CF_1(Y)$, $[A_0^{\infty}] \in M_{\hat{X}}(\rho(0); \Sigma_1, \dots, \Sigma_4) \cap V_{\Lambda}$, $[A_1^{\infty}] \in M^0_{ Y \times \R }(\rho(0), \theta)$.

\end{enumerate}

We describe the ends of the moduli space relevant to the case (\ref{theta}).
Let $x_1 \in \Sigma_{l_1} \cap \Sigma_{l_2}$, $x_2 \in \Sigma_{l_3} \cap \Sigma_{l_4}$ for distinct $l_1,\dots, l_4$ and take small neighborhoods $U_{x_1}$, $U_{x_2}$ in $\hat{X}$.
There is an $SO(3)$-equivariant map
\[
\kappa:U_{x_1} \times U_{x_2} \times SO(3) \times SO(3) \times (T_0, \infty) \times (T_0, \infty)
\longrightarrow 
H_{ \theta_{ \hat{X} } }^2 = H^+(X) \otimes \su(2)
\]
such that $\kappa^{-1}(0)/SO(3)$ is homeomorphic to an open set in $M_{ \hat{X} }(\theta)$.
As in \cite{conn} we can see that
\[
( \kappa^{-1}(0)/SO(3) ) \cap V_{\Sigma_1} \cap \cdots \cap V_{\Sigma_4}
\cong
\{ x_1 \} \times \{ x_2 \} \times \ell \times \{ 1 \} \times \{ \ ( T, T ) \ | \ T > T_0 \ \},
\]
where $\ell$ is a loop in $SO(3)$ which represents the generator of $\pi_1(SO(3)) \cong \Z_2$.
For each $\underline{x}=(x_1, x_2)$, we have a gluing map
\[
Gl_{\underline{x}}: \ell \times \{ \ (T, T) \ | \ T > T_0 \ \}
\longrightarrow
M_{ \hat{X} }(\theta; \Sigma_1, \dots, \Sigma_4)
\]
which is given by gluing the trivial connection on $\hat{X}$ and two copies of the fundamental instanton $J$ over $S^4$ at $x_1, x_2$.
The intersection of the image of this map and $V_{\Lambda}$ is one of the ends of the moduli space.

Let  $\unA = ([A_0], [A_1])$ be an element of $M_{\hat{X}}( \rho; \Sigma_1, \dots, \Sigma_4 ) \times \big( M^0_{ Y \times \R }(\rho, \theta) \cap V_{\Lambda} \big)$ for $[\rho] \in CF_2(Y)$,  or an element of $\big( M_{\hat{X}}( \rho; \Sigma_1, \dots, \Sigma_4 ) \cap V_{\Lambda} \big) \times M^0_{ Y \times \R }(\rho, \theta) $ for $[\rho] \in CF_1(Y)$. Then we have a gluing map
\[
Gl_{ \unA }:(T_0, \infty) \longrightarrow M_{ \hat{X} }(\theta; \Sigma_1, \dots, \Sigma_4).
\]
The ends of the moduli space relevant to the case (\ref{deg 2}), (\ref{deg 1}) are the intersections of the images of these gluing maps and $V_{\Lambda}$.

Put
\[
M' = 
M_{ \hat{X} }(\theta;\Sigma_1, \dots, \Sigma_4) \ \backslash \ 
\big(
\bigcup_{ \underline{x} } \Im Gl_{\underline{x}} \cup
\bigcup_{ \unA } \Im Gl_{ \unA }
\big).
\]
If we perturb the section $s_{\Lambda}$ over a compact set in $M'$ then the intersection $M' \cap V_{\Lambda}$ is a smooth compact manifold with boundary
\[
\begin{split}
& \bigcup_{ [\rho] }
M_{ \hat{X} } (\rho; \Sigma_1, \dots, \Sigma_4) \times \big( M_{Y \times \R}^0(\rho, \theta) \cap V_{\Lambda} \big) \\
\cup & \bigcup_{ [\sigma] }
\big( M_{ \hat{X} } (\sigma; \Sigma_1, \dots, \Sigma_4) \cap V_{\Lambda} \big)
\times
M_{ Y \times \R}^0(\sigma, \theta) \\
\cup & \bigcup_{ \underline{x}}
Gl_{ \underline{x} }(\ell, T_0) \cap V_{\Lambda}.
\end{split}
\]
Here $[\rho]$ runs over the generators of $CF_2(Y)$ and $[\sigma]$ runs over the generators of $CF_1(Y)$.
Therefore we get
\[
D^{(0)}( \Psi_{X}^{u_1} ([\Sigma_1], \dots, [\Sigma_4]) ) + 
\sum_{ \underline{x} } \# \big( Gl_{\underline{x}}( \ell, T_0) \cap V_{\Lambda} \big)
\equiv
0 \mod 2.
\]
We can see that
\[
\# \big( Gl_{\underline{x}}( \ell, T_0) \cap V_{\Lambda}) \equiv 1 \mod 2
\]
for each $\underline{x}$.
This follows from arguments similar to those which deduce (\ref{eq S_A}).
The point is that $J$ is an instanton on an $SU(2)$-bundle $P'$ with $c_2(P') = 1$ and hence the index $\Ind D_{J}$ is odd. Therefore we get
\[
D^{(0)}( \Psi_{X}^{u_1} ([\Sigma_1], \dots, [\Sigma_4]) ) 
\equiv 
\sum_{ \underline{x} } 1
\equiv 
Q_{X}^{(2)}([\Sigma_1], \dots, [\Sigma_4]) \mod 2
\]
as required.


\end{document}